\documentclass[12pt, leqno]{amsart}

\usepackage{esint} 
\usepackage{amsmath}
\usepackage{amssymb}
\usepackage{amsfonts}
\usepackage{enumerate}
\usepackage{enumitem}
\usepackage{mathrsfs}
\usepackage{mathtools}
\usepackage{xcolor}

\usepackage{float}
\usepackage{tikz}
\usetikzlibrary{lindenmayersystems}
\usepackage{pgfplots}
\usepackage{comment}

\usepackage{amsmath}
\usepackage{amssymb}
\usepackage{amsthm} 
\usepackage{epsf}
\usepackage{graphicx}
\usepackage{esint} 
\usepackage{dsfont}
\usepackage{hyperref} 
\hypersetup{backref,  pdfpagemode=FullScreen,  colorlinks=true} 
\usepackage[french,english]{babel}

\usepackage{ulem}

\numberwithin{equation}{section}

\newcommand{\R}{{\mathbb R}}

 \renewcommand{\H}{{\mathcal H}} 
\newcommand{\sm}{\setminus}

\newcommand{\bD}{\mathbb D} 
\newcommand{\ol}{\overline} 
\newcommand{\wt}{\widetilde} 
\newcommand{\cA}{\mathcal A} 
\newcommand{\p}{\partial}

\newcommand{\nn}{{\nonumber}}

\renewcommand{\d}{\partial}
\newcommand{\ms}{\medskip}

\newcommand{\diam}{{\rm diam}}
\newcommand{\dist}{{\rm dist}}

\newtheorem{theorem}{Theorem}[section]
\newtheorem*{theorem*}{Theorem}
\newtheorem{lemma}[theorem]{Lemma}

\newtheorem*{claim*}{Claim}
\newtheorem{corollary}[theorem]{Corollary}

\theoremstyle{remark}
\newtheorem{remark}{Remark}[section] 
\theoremstyle{definition}

\newtheorem{definition}{Definition}
\DeclareMathOperator{\Tr}{Tr}

\title{Robin Green Function Estimates and a model of Mammalian Lungs}
\author[David, Decio, Engelstein, Filoche, Mayboroda, Michetti]{Guy David, Stefano Decio, Max Engelstein, Marcel Filoche, Svitlana Mayboroda and Marco Michetti}

\date{}
\thanks{M.E. would like to thank Guillaume Bal who first pointed out the possible connection with boundary layers from kinetic and fluid theory. He would also like to thank Alex Watson for patiently explaining to him what these boundary layers are. M.E. was supported by NSF DMS CAREER 2143719. G.D. and M.M. were partially supported by the Simons Foundation grant 601941, G.D. S.M. and S.D. were supported in part by the Simons foundation grant 563916, SM. SM was in addition supported by the NSF grant DMS 1839077.  M.F. was supported by the Simons foundation grant 1027116, MF. This work began during a research in groups meeting hosting by Professor Carlos Vincente at the Universidad de Puerto Rico. All the authors would like to thank both Carlos and the institute for their hospitality. All the authors thank Doug Arnold, Qile Yan, and Chuning Wang for numerical simulations which guided us to Theorem \ref{thm:Fluxest}. }

\begin{document}
\maketitle

\begin{abstract}
	The present paper establishes delicate properties of the Green function with Robin boundary conditions, in particular, elucidating the nature of the passage between the Dirichlet-like and Neumann-like behaviour. This yields sharp quantifiable bounds on the corresponding harmonic measure and proves the phase transition in the behavior of the total flow earlier conjectured in physics literature in concert with the efficacy of mammalian lungs.
	\end{abstract}

\section{Introduction}
\label{intro}

We continue our study of the Robin problem in rough domains, started in \cite{Robin1}. Recall the Robin problem, with parameter $a\in (0, \infty)$ in a domain $\Omega \subset \mathbb R^n$\begin{equation}\label{e:classicrobin} \begin{aligned}
-\mathrm{div}(A(x)\nabla u) =& 0, \qquad x\in \Omega\\
\frac{1}{a}A(Q)\cdot \nabla u(Q)\cdot \nu(Q) + u(Q) =& f(Q) \qquad Q\in \partial \Omega,
\end{aligned}
\end{equation}
(where $\nu(Q)$ is the outward pointing unit normal to $\partial \Omega$ at $Q$). As $a$ goes from $0$ to $\infty$ the boundary condition (suitably renormalized) seems to interpolate between the well-studied Dirichlet and Neumann boundary conditions. Yet, the properties of the Dirichlet and Neumann solutions near the boundary are so different that understanding and quantifying this passage is a challenging task. Partially motivated by applications to the study of mammalian lungs, we embarked on this journey in \cite{Robin1}. Conjectures in the biomedical engineering community predicted a phase transition in the efficacy of a lung and hence, likely, in the dimension of the harmonic measure with the Robin boundary conditions depending on the value of the parameter $a$. Expectations based on the mathematics of the classical Dirichlet harmonic measure also seemed to indicate that the Robin problem should exhibit a phenomenon akin to dimension drop when $a$ is large. Yet, we showed that, contrary to the Dirichlet case, the harmonic measure with the Robin boundary condition is absolutely continuous with respect to the Hausdorff measure of the boundary, even on fractional dimensional domains and even for general elliptic operators, and hence, in that sense, it changes its nature suddenly, showing dramatically different behavior for any $a>0$ (see \cite[Theorem 1.2]{Robin1}). In the present paper we establish delicate estimates on the Robin Green function, and show that quantitatively the system indeed undergoes a phase transition as the balance between the scales and $a$ changes. Our results are interesting in their own right, as even on smooth domains many of the estimates are new, but in particular, they allow us to provide a more detailed analysis of the harmonic measure itself and elucidate the nature of the aforementioned phase transition.  

Our main result, Theorem \ref{thm:GFest}, gives sharp bounds on the behavior of the Robin Green function, $G_R^a(X,Y)$, in terms of the Dirichlet Green function $G_D(X,Y)$ and the relationship between $X,Y$, $a$ and the geometry of the domain.  Roughly stated, we prove that if $1/a$ is sufficiently small, then the Robin Green function is comparable to the Dirichlet Green function but averaged at a scale that grows monotonically with $1/a$. If $1/a$ is sufficiently large, then the Robin Green function is essentially constant, except in a small neighborhood of the pole, where it is comparable to $|X-Y|^{2-n}$. This precise relation between the two solutions does not seem to have an analogue in the elliptic PDE literature, but is reminiscent of boundary layer phenomena in fluid and kinetic theory. 

We give two applications of this result; first we establish some fine properties of the Robin harmonic measure, including the asymptotic behavior of the $A_\infty$-character of Robin harmonic measure as $a\uparrow \infty$, a point left open in our previous work. The second is a theoretical confirmation of a phenomenon observed both in physical and numerically simulated mammalian lungs (see, e.g. \cite[Figure 5]{SFW},  and \cite[Figure 5]{RWLungs} and also the references in those papers). Namely, that total oxygen flow rate through the lung does not decrease linearly as the local alveolo-capillary membrane permeability decreases (say due to fluid build-up in the case of pulmonary edema).  Indeed, in Theorem \ref{thm:Fluxest} we show that the total flow rate of oxygen through the lung is essentially constant (as a function of the local membrane permeability) as long as permeability is larger than a multiple of one over the surface area of the lung. When the permeability decreases even further, the total oxygen transfer decays linearly to zero with the permeability. The resulting robustness of the oxygen transfer performance can be seen as an evolutionary justification for the large surface area (i.e. fractal-like behavior) of mammalian lungs. See Section \ref{ss:lungs} for a more thorough description of the model and our results. 

Finally, we point out that a key component of our results (i.e. Theorems \ref{thm:GFest}, \ref{thm:Fluxest}) is that they make novel statements about the behavior of solutions to \eqref{e:classicrobin}, in \emph{intermediate scales}. While large scale rigidity results (see, e.g. \cite{LewisVogel, KenigToroRigid}) and small scale regularity results are common in the literature on boundary value problems in rough domains, we know of very few results which try to capture behavior in intermediate scales, much less identify a regime of a scale invariant behavior in such a situation. We view this as an essential aspect of a robust Robin theory, since the Robin problem is not scale invariant. In future work we intend to apply these ``meso-scale" theorems in other contexts (such as shape optimization problems). 

Let us turn to our underlying assumptions. As in \cite{Robin1}, we will work in bounded domains $\Omega \subset \mathbb R^n$ with $n\geq 3$. Furthermore, these domains satisfy the following quantitative topological assumptions (sometimes called one-sided NTA conditions):

	
	\begin{enumerate}[label = (C\arabic*)]
		\item \label{CC1} Interior corkscrew condition: there exists a $M > 1$ such that for every $Q \in \partial \Omega$ and $0 < r < \mathrm{diam}(\Omega)/10$,
 there exists an $A_r(Q) \in \Omega \cap B(Q,r)$ such that $\mathrm{dist}(A_{r}(Q), \partial \Omega) \geq M^{-1}r$. 
		\item \label{CC2} Interior Harnack chains: there exists a constant $M > 1$ such that for all $0 < r < \mathrm{diam}(\Omega)/10$, 
$\varepsilon > 0$, $k \geq 1$, 
 and all $x, y\in \Omega$ with $\mathrm{dist}(x,y)=r$ and $\mathrm{dist}(x,\partial \Omega)\geq \epsilon$, $\mathrm{dist}(y,\partial \Omega) \geq \epsilon$ with $2^k \epsilon \geq r$ there exists $Mk$ balls $B_1,..., B_{Mk} \subset \Omega$ with centers $x_1,\ldots, x_{Mk}$ such that $x_1 = x, x_{Mk} = y$, 
 $r(B_i) \leq \frac{1}{10}\mathrm{dist}(x_i, \partial \Omega)$ for all $i$, 
 and $x_{i+1} \in B_i$ for $i < Mk$. 
\end{enumerate}

 In addition, in order to define the Robin condition, we will always assume that we are given a measure $\sigma$, whose (closed) support is $\d\Omega$, and which satisfies the 
 following:
 
  \begin{definition} \label{d:mixed}
 We will say that the pair $(\Omega,\sigma)$ is a {\bf (one-sided NTA) pair of mixed dimension}
 when $\Omega$ satisfies \ref{CC1} and \ref{CC2}, $\sigma$ is a doubling measure 
 whose support is $\d\Omega$, and in addition there are constants $d > n-2$ and $c_d > 0$ such that
 \begin{equation} \label{1n3}
\sigma(B(Q,r)) \geq c_d (r/s)^{d}\sigma(B(Q,s))
\ \text{ for $Q \in \d\Omega$ and } 0 < s < r < \diam(\Omega).
\end{equation}
Recall that doubling means that there is a constant $C \geq 0$ such that 
\begin{equation} \label{1n4}
\sigma(B(Q,2r)) \leq C \sigma(B(Q,r))
\ \text{ for $Q \in \d\Omega$ and } 0 < r < \diam(\Omega).
\end{equation}
A {\bf bounded (one-sided NTA) pair of mixed dimension} will simply be a (one-sided NTA) pair of mixed dimension
for which $\Omega$ is bounded.
 \end{definition}
 
 An interesting special case is when $\d\Omega$ is Ahlfors regular of dimension $d$ for some $d \in (n-2,n)$, and $\sigma$ is the 
 restriction to $\d\Omega$ of the Hausdorff measure $\H^d$.
 For a pair of mixed dimension, we will refer to the constants $d$, $C$ and $c_d$
in the conditions above as the geometric constants of $\Omega$. 
In particular, for us, the diameter of $\Omega$ is not a geometric constant, and we will keep explicit the way our estimates depend on it.
\\

We consider domains of the above generality not only because our methods allow it, but because domains of mixed dimension are a good model of ``pre-fractal lungs." Indeed, at small scales (comparable to the size of a cell), the lungs look smooth and so if $\sigma = \mathcal H^2|_{\text{lung surface}}$ then $\sigma(B(x,r))\simeq r^2$ small $r$. On the other hand, if $r$ is much larger than the cell length, we have $\sigma(B(x,r)) \simeq r^\alpha$ where different sources put $\alpha \in (2.7, 3.1)$ (see, e.g. \cite{FractalDim}). 

Fixing a parameter $a \in (0,\infty)$ and a measurable, uniformly elliptic matrix valued $x\mapsto A(x)$, in \cite{Robin1} we showed that domains as above admit a Robin Green function, which is the unique solution of $$\int_{\Omega} A(X)\nabla G_R^a(X,Y) \nabla \varphi(X)\, dX +a \int_{\partial \Omega} G_R^a(P, Y) \varphi(P)\, d\sigma(P) = \varphi(Y), \qquad \forall \varphi \in C^\infty_c(\mathbb R^n).$$ For a precise statement of existence and the properties of $G^a_R(X,Y)$, see Theorem \ref{t:GreenfunctionExistence} below.  

Our main theorem gives sharp bounds on $G_R^a(X,Y)$ depending on $X, Y, a$ and the geometric constants of $\Omega$:

\begin{theorem}\label{thm:GFest} 
Let $\Omega \subset \mathbb R^n, n\geq 3$, be such that $(\Omega, \sigma)$ is a one-sided NTA pair of mixed dimension (i.e. satisfies Definition \ref{d:mixed}). Let $A$ be uniformly elliptic 
and let $G^a_R(X,Y)$ be the Robin Green function with parameter $a \in (0,\infty)$ as in \eqref{e:Greenidentity} and \eqref{e:representationformula}. There exists a $C > 0$ (depending on the dimension
$n$, 
the geometric constants of $(\Omega,\sigma)$, 
and the ellipticity of $A$, 
but independent of $a,r$) such that the following holds:
\begin{itemize}
\item If $a \sigma(\partial \Omega) \leq \diam(\Omega)^{n-2}$, 
define  $\displaystyle \rho:= \left(a\sigma(\partial \Omega)\right)^{\frac{1}{n-2}}$.
Then 
\begin{equation}\label{e:GFestNeumannend}
\begin{aligned}
    C^{-1}\rho^{2-n} \leq& G^a_R(X,Y) \leq C\rho^{2-n} 
    \qquad \text{ when } |X-Y| \geq \rho,\\
    C^{-1}|X-Y|^{2-n} \leq& G^a_R(X,Y) \leq C|X-Y|^{2-n} 
    \qquad \text{ when } |X-Y| \leq \rho. 
\end{aligned}
\end{equation} 
\item If $a \sigma(\partial \Omega) \geq \diam(\Omega)^{n-2}$
and $\delta(X), \delta(Y) > C^{-1}|X-Y|$ then \begin{equation}\label{e:GFestMedium}
G_D(X,Y) \leq G_R^a(X,Y) \leq CG_D(X,Y).
\end{equation}
\item Finally, if $a \sigma(\partial \Omega) \geq \diam(\Omega)^{n-2}$
but $\min\{\delta(X), \delta(Y) \} \leq C^{-1}|X-Y|$, then $G_R^a$ is comparable to $G_D$ but evaluated at the appropriate corkscrew points. To be precise, let $Q_X\in \partial \Omega$ be such that $|Q_X-X| = \delta(X)$ (and similarly for $Y$). 
Let $r_X := \min\{|X-Y|/10, 
\sup \{\rho > 0\mid \rho^{2-n}\sigma(B(Q_X, \rho))\leq 1/a\}\}$ (and similarly for $Y$). Finally, let $A_X = X$ if $\delta(X) \geq r_X$ and let $A_X$ be equal to the corkscrew point at scale $r_X$ for the point $Q_X$ otherwise. Then 
\begin{equation}\label{e:GFestclose}
C^{-1} G_D(A_X,A_Y) \leq G_R^a(X,Y) \leq G_D(A_X,A_Y).
\end{equation}
\end{itemize}
Here $G_D(X,Y)$ is the Dirichlet Green function for the operator. 
\end{theorem}

To understand Theorem \ref{thm:GFest} better, we recall that at least formally, we recover the Neumann problem when $a \downarrow 0$ and the Dirichlet problem when $a\uparrow \infty$. Furthermore, the Robin problem is not scale invariant. That is, assume $0 \in \partial \Omega$, $X,Y\in \Omega$ and $r > 0$. Then if we define $u(X) = G_R^a(rX,Y)$ we can see that $u$ is the Robin Green function in the domain $\Omega/r$ with surface measure $\sigma_r(E):= \sigma(E/r)$, Robin parameter $ar^{2-n}$ and pole at $Y/r$.  That is to say, changing the scale is equivalent to changing the Robin parameter (and underlying surface measure). This analysis shows that if $ar^{2-n}\sigma(B(0, r)) \leq 1$ we should expect the Robin Green function to exhibit ``Neumann-like" behavior in $B(0,r)\cap \Omega$ whereas if $ar^{2-n}\sigma(B(0,r)) \geq 1$ we should expect the behavior to be more ``Dirichlet-like". Theorem \ref{thm:GFest} gives a much more precise quantification of this phenomenon.

Take for example, \eqref{e:GFestclose}. For each $Q\in \partial \Omega$ let $\rho_Q$ be such that $a\rho_Q^{2-n}\sigma(B(Q, \rho_Q)) = 1$. Then let $\Omega_N = \bigcup_{Q\in \partial \Omega} B(Q,\rho_Q)\cap \Omega$. This is our ``boundary layer". The estimate \eqref{e:GFestclose} tells us that if the pole $Y$ is far away from $\partial \Omega$, then for $X\in \Omega\backslash \Omega_N$ the Robin Green function is comparable to the Dirichlet Green function. Inside $\Omega_N$, the Robin Green function exhibits ``Neumann behavior" and oscillates as little as possible (see, e.g. Lemma \ref{l3a3}). This is reminiscent of the construction of solutions to fluid or kinetic equations with boundary layers: we have an ``inner solution" given by the Dirichlet Green function and an ``outer solution" which interpolates between the Dirichlet Green function on $\partial \Omega_N\cap \Omega$ and the desired Robin conditions on $\partial \Omega$.

\medskip

\noindent {\bf Sharpness of Theorem \ref{thm:GFest}:} Letting $A = Id$ (so that the operator is the Laplacian) and $\Omega = B(0, R)$ with $Y = 0$ we can explicitly solve for $$\begin{aligned} G_R^a(X, 0) =& \frac{c_n}{|X|^{n-2}} -\frac{c_n}{R^{n-2}} + \frac{(n-2)c_n}{aR^{n-1}}\\
G_D(X,0) =& \frac{c_n}{|X|^{n-2}} - \frac{c_n}{R^{n-2}}.\end{aligned}$$ This example shows that the estimate \eqref{e:GFestNeumannend} is sharp (the estimates \eqref{e:GFestclose}, \eqref{e:GFestMedium} have to be sharp because they say two quantities are equivalent). This example also shows that the respective ranges of $X,Y,a$ in which \eqref{e:GFestclose}, \eqref{e:GFestMedium}, \eqref{e:GFestNeumannend} hold are also sharp. 

It is an interesting question whether the geometric assumption on $\Omega$ can be loosened. We do not have explicit examples, but the trapped domains in \cite{BBC} show that the Robin problem has pathological behavior in the presence of cusps, and so some condition like interior corkscrews should be needed.  

We expect that the connectivity assumption on $\Omega$ can be relaxed to some extent (e.g. replacing Harnack chains with the weak local John condition) especially if the results are suitably localized. We choose to work in the simplified setting to avoid further complications. This is also why we choose to work with bounded domains, though we expect analogues of these results to hold for unbounded domains.

As in our previous paper, we can allow for $a$ to be variable, with $a(x), a^{-1}(x) \in L^\infty(d\sigma)$. In fact, this follows by considering the measure $d\tilde{\sigma} := a(x)d\sigma$ and noticing that for $a, \sigma$ as above, $\tilde{\sigma}$ also satisfies the conditions \eqref{1n3}, \eqref{1n4}.

Finally, we expect that an analogue of these results holds for $n=2$ but that will be addressed in future work.  

\medskip

\noindent {\bf Comparisons with Prior Work:} Green functions are a natural tool for understanding the solution both of boundary value and Poisson problems and thus are a well studied object (see, e.g.  \cite{HS, JosephLinhan, JosephLinhan2, MPT, MT, BrunoJoseph} for some examples from the last few years). 

There is a particularly large literature devoted to establishing upper bounds for Green's functions with a variety of boundary conditions (including Robin). See \cite{Davies} for a survey of some of the more classical work, mostly with Neumann and Dirichlet boundary conditions. For some more recent work devoted to Robin boundary conditions, see e.g. \cite{Dutch, Mitrea2, ChoiKim, Daners}. The introduction of \cite{Mitrea1}, in particular, gives a comprehensive introduction to the area. Key to these results are powerful and flexible functional analytic tools which allow the authors to also treat the cases of systems of equations and time-dependent operators. 

We are certainly unable to handle systems of equations with our techniques, though we hope to address the time-dependent problem in future work. This is a limitation of our tools, which rely heavily on the maximum principle.  On the other hand, we provide lower bounds and track the precise dependence of the upper bounds on the Robin parameter, something we did not find elsewhere in the literature. Indeed, we believe that our lower bounds in \eqref{e:GFestNeumannend} and \eqref{e:GFestclose} are new and of interest even for the Laplacian in smooth domains. 

 For the Dirichlet problem the most general bounds are due to \cite{GW}, which state roughly that the Dirichlet Green function, $G_D(X,Y)$ is comparable to the fundamental solution to the Laplacian as long as $X,Y$ are not much closer to $\partial \Omega$ than they are to each other (i.e. \eqref{e:GFestMedium}). Our results are a natural analogue of theirs (for the Robin condition) though again the arguments are quite different. 
  
For the Dirichlet Green function it is of interest to understand the precise rate at which the function vanishes when one of the points goes to the boundary. This is impossible to do in complete generality (indeed, it is intimately connected to the famous open problem of the dimension of Dirichlet harmonic measure, see \cite{Badger1, Badger2} for discussion). As a consequence, we cannot give more precise bounds than \eqref{e:GFestclose} in general. However, lower bounds on the Dirichlet Green function can be established under extra assumptions on the flatness of $\partial \Omega$ (see, e.g. \cite{KenigToroDuke}) or if the domain has lots of symmetries (see e.g. \cite{ColeHM}). Of course, the estimate \eqref{e:GFestclose} implies that in any scenario where we can say more about the Dirichlet Green function, we can transfer similar estimates to the Robin Green function. 

\medskip

\noindent {\bf Consequences of Theorem \ref{thm:GFest}}: Before turning to the model of lungs, we give some applications of Theorem \ref{thm:GFest} to the fine properties of Robin harmonic measures in Section \ref{sec:hmproperties}. First we use the upper bounds of Theorem \ref{thm:GFest} to prove several estimates relating the boundary behavior of the Robin Green function and the Robin harmonic measure. These estimates are now standard for the Dirichlet problem (see, e.g. \cite{kenigbook}), though our proofs often differ from the Dirichlet setting. In particular, we prove a lower bound on the Robin harmonic measure of ball when the pole is close to the ball (often known as a Bourgain estimate, see Lemma \ref{l:Bourgain}), we give a comparison between the Robin Green function and Robin harmonic measure (see Theorem \ref{t:hmgfequiv}) and consequently show that the Robin harmonic measure is uniformly doubling (Theorem \ref{t:doubling}). Finally, we end Section \ref{sec:hmproperties} with a boundary comparison principle (Theorem \ref{t:boundarycomp}). We believe this comparison principle to be new for the Robin problem even in smooth domains, though our proof is a close adaptation of \cite{BoundaryHarnack} (which is written for Dirichlet boundary conditions but very flexible). 

The second consequence of Theorem \ref{thm:GFest} is a further study of the decay of the $A_\infty$-character of Robin harmonic measure as $a\uparrow \infty$ (or equivalently as $r \uparrow \infty$). In \cite[Theorems 1.3 and 1.4]{Robin1} the authors proved that the Robin Harmonic measure satisfied the following $A_\infty$ condition $$C^{-1}\frac{\sigma(E)}{\sigma(B(Q,r))} \leq \frac{\omega^X_{R,a}(E)}{\omega^X_{R,a}(B(Q,r))} \leq C\frac{\sigma(E)}{\sigma(B(Q,r))}, \qquad \forall E \subset B(Q,r)\cap \partial \Omega, \forall X\in B(Q,Cr)^c\cap \Omega.$$ In the above, the constant $C> 1$ was uniform for all $r,a$ such that $ar^{2-n}\sigma(B(Q,r)) \leq 1$ and degenerated as $ar^{2-n}\sigma(B(Q,r))\uparrow \infty$. However, the precise rate of degeneracy given in \cite[Theorem 1.4]{Robin1} was not sharp.   In Subsection \ref{ss:Ap} we use Theorem \ref{thm:GFest} to relate the $A_\infty$ behavior of $\omega_R$ with that of $\omega_D$. In particular, we make more precise the intuition that the Robin harmonic measure is an averaging out of the Dirichlet harmonic measure, see Lemma \ref{l:ainfinitylocal}. We also show that if $\omega_D\in A_\infty$ then the constants in the $A_\infty$ character of $\omega_R$ are independent of $a$ (a point left open in \cite{Robin1}). More precisely, we show the following: 

	\begin{theorem}\label{t:keepainfinity}
			Let $\Omega \subset \mathbb R^n, n \geq 3$ be a one-sided NTA pair of mixed dimension. Let $A$ be an elliptic matrix and let $\omega_D, \omega_{R,a}$ be the Dirichlet and Robin harmonic measures associated to the operator $-\mathrm{div}(A\nabla)$. Assume that $\omega_D \in A_\infty(\sigma)$.\footnote{We follow the literatue and use this as shorthand to denote that $\omega_D^{X_0}|_{B(Q,r)} \in A_{\infty}(\sigma|_{B(Q,r)})$ whenever $X_0\in \Omega \sm B(Q, Cr)$ with constants that are indepedent of $X_0, Q, r$.} There exist  constants $C > 0, \theta \in (0,1)$ such that if $E \subset B(P,s)\cap \partial \Omega \subset B(Q,r)\cap \partial \Omega$ then $$\frac{\omega^{X_0}_{R,a}(E)}{\omega^{X_0}_{R,a}(B(P,s))} \leq C\left(\frac{\sigma(E)}{\sigma(B(P,s))}\right)^\theta, \qquad \forall X_0 \in \Omega \sm B(Q, Cr).$$
			
	We emphasize that $C > 1$ depends on the $A_\infty$ character of $\omega_D$, the geometric constants of $\Omega$ and the ellipticity of $A$ (but not on $a, r, Q, X_0$). The constant $\theta \in (0,1)$ depends only on the $A_\infty$ character of $\omega_D$. 
			\end{theorem}

\subsection{A model for mammalian lungs}\label{ss:lungs} Inspired by a large body of work in the physics and biology literature (see, e.g. the first chapter of the conference proceedings \cite{FractalBook}), we consider the following model of oxygen transfer through the lungs of mammals. Once air is convected through the bronchial tree and has entered the acinar region (where alveoli are found), the concentration of oxygen molecules is governed by diffusion. When these oxygen molecules come into contact with the walls of the alveoli, namely the alveolo-capillary membrane which separates air and blood compartments, they undergo a transfer process across the membrane to reach the blood flow in the capillaries. The presence of fluid on the membrane or other physiological conditions may alter the parameters of this transfer. Viewed from the air compartment it can be seen as an absorption process with some probability. We represent the (steady state) local oxygen concentration as a function, $u_a$, which is harmonic in the alveoli, has a constant (averaging over time) source at the entrance of the acinar region where oxygen transport is diffusion-dominated (a sub-acinus, see~\cite{SFW}) and satisfies Robin conditions on the interface between the alveoli and the blood vessels surrounding them. So a simplified, but well used, see e.g. \cite[Figure 2]{RWLungs}, model is that $u_a$ solves 
\begin{equation}\label{e:fluxguy}\begin{aligned}
    -\Delta u_a =& 0, \qquad \text{in}\,\, \Omega \backslash B_1\\
u_a=& 1, \qquad \text{on}\,\, B_1\\
\partial_\nu u_a + au_a =& 0, \qquad \text{on}\,\, \partial \Omega.
\end{aligned}
\end{equation}

We are interested in the transfer of oxygen into the blood stream, which we capture by computing the total oxygen flow rate  \begin{equation}\label{e:totalfluxdef} F(a):= -\int_{\partial \Omega} \partial_\nu u_a\, d\sigma.\end{equation} Here $\nu$ is the outward pointing unit normal to $\partial \Omega$. When modeling the alveoli, $\partial \Omega$ is piecewise Lipschitz so the normal derivative exists (almost) everywhere.

As alluded to above, our domain $(\Omega, \sigma)$ will look smooth at small scales (i.e. approximately the length of a cell) but then will have geometric complexity at larger scales. We capture this with what we call ``pre-fractal" domains. 

\begin{definition}\label{d:piecewisesmooth}
    We say that the pair $(\Omega,\sigma)$ of mixed dimension is of pre-fractal type if 
    there is a length scale $\ell > 0$, $\ell < \diam(\Omega)$, 
    and constants $C, L> 1$ such that:
    \begin{itemize}
    \item For any $P, Q\in \partial \Omega$ and $r > 0$, $\sigma(B(Q,r)) \leq C\sigma(B(P,r))$. 
    \item For any $r \leq \ell$, $\partial \Omega \cap B(Q,r)$ is the graph of an $L$-Lipschitz function. 
    \end{itemize}
\end{definition}

We think of prefractal domains as approximations, at some small scale $\ell >0$, of fractal domains. Although, of course, they also approximate the more general class of pairs $(\Omega, \sigma)$ of mixed dimensions. Notice that in the first condition we ask for some form of homogeneity. For this reason, we may assume that $\sigma$ is
(equivalent to) a multiple of the restriction of $\H^{n-1}$ to $\d\Omega$ (see Section \ref{s:Flux} for more discussion). 
We work with these general domains, as opposed to a specific approximation of human lungs, in order to cover a wide range of potential shapes of the lung. This is partly to account for inter-species variety, and partly because in the future we may investigate questions related to optimal (for oxygen transfer in the lung) shapes. 

In addition to assuming that $(\Omega, \sigma)$ satisfies Definition \ref{d:piecewisesmooth} we make several other simplifying assumptions: namely, there exists a $C_0 >4$ such that
\begin{equation}\label{1a7}
B(0,4) \subset \Omega \subset B(0, C_0).\footnote{It was suggested to us by Bob Kohn that it may be interesting to consider also the case when 
we do not require $B(0,4) \subset \Omega$; certainly some of the proofs below (for instance, the 
discussion of cases for \eqref{7a19} and \eqref{7a21}) need to be adapted if we want universal bounds. 
For the moment we decided not to follow suit.}
\end{equation}

 Under these conditions, there exists a unique, continuous $u_a$ which satisfies \eqref{e:fluxguy} (see Section \ref{s:Flux} for more discussion).

We have different estimates, depending on the values of $a$, which we summarize as follows.

\begin{theorem}\label{thm:Fluxest}
Let $(\Omega, \sigma)$ satisfy Definition \ref{d:piecewisesmooth} in $\mathbb R^n$, $n\geq 3$, with $B(0,4) \subset \Omega \subset B(0, C_0)$. Let $u_a$, and $F(a)$ be as in \eqref{e:fluxguy}, \eqref{e:totalfluxdef}. Then

\begin{enumerate}
\item $a\mapsto F(a)$ is strictly increasing and bounded above by a constant that depends only
on $n$. 
Set $F(\infty) = \lim_{a \to +\infty} F(a) \in(0,+\infty)$.

 \item For $a \leq \sigma(\partial \Omega)^{-1}$, 
 $$
 F(a) \simeq a\sigma(\partial \Omega).
 $$ 
 The constants of comparability depend only on the constants in Definition \ref{d:mixed}.

  \item There exists a constant $C > 0$, that depends only on $C_0$ in
  \eqref{1a7} and the mixed dimension constants, 
  such that for all $a \geq \sigma(\partial \Omega)^{-1}$ we have 
  $$
  C^{-1}F(\infty) \leq F(a) \leq F(\infty). 
  $$

Further assume that  $\sigma$ is a multiple of the surface measure $\H^{n-1}$ on $\d\Omega$; then
  
   \item There exists a constant $C >1$ (depending on parameters in Definition \ref{d:piecewisesmooth}) 
    such that if $\frac{1}{a} \leq C\ell$  we have 
    $$F(\infty) - F(a) \simeq \frac{1}{a}.$$ 
    Here the constants of comparability depend on the constants in Definitions \ref{d:mixed} and \ref{d:piecewisesmooth}.

    \item Finally,  when $C\ell\leq \frac{1}{a} \leq \diam(\Omega)^{2-n}\sigma(\partial \Omega)$ 
    (the intermediate range), 
    we can define $r_a > 0$ such that $r_a^{2-n}\sigma(B(Q,r_a)) \simeq \frac{1}{a}$ for all $Q\in \partial \Omega$,
    and then 
    \begin{equation}\label{e:approxsquarefunction}
F(\infty)-F(a) \simeq r_a^{2-n} S(\omega^0_D, r_a, 2), 
    \end{equation} 
    where $\omega^0_D$ is the Dirichlet harmonic measure of $\Omega$ with a pole at $0$ and 
    $S(\omega^0_D, r_a, 2)$ is the entropy functional of \cite{Makarov} (defined in \eqref{7a26}) 
    Here the constants of comparability depend only on the constants in Definition \ref{d:mixed}.
\end{enumerate}

\end{theorem}

Theorem \ref{thm:Fluxest}, confirms both observations (e.g. \cite[Figure 72]{LungsLungs} and the discussion around it) and numerical simulations (e.g. \cite[Figure 5]{SFW}) which indicate that, at rest, total oxygen transfer initially changes very slowly as the local absorption rate decreases (e.g. as fluid builds up in the lungs). However, once the local membrane permeability passes a threshold (see the discussion of screening in \cite{SFW}), the total flow rate of oxygen decreases quickly. In Theorem \ref{thm:Fluxest}, this is captured by the second and third bullet points. Although it is not physically relevant, the fourth bullet point has also been captured in numerical simulations and states roughly that as the local membrane permeability approaches infinity the total flow rate approaches the optimal value linearly.

Finally, the fifth bullet point gives finer information on the change of the total absorption for intermediate local absorption rates. In particular, it states that this change is comparable to a multi-scale function of the lung's Dirichlet harmonic measure, known as the Makarov 2-entropy. This entropy is notoriously difficult to understand, especially for subsets of $\mathbb R^n$ with $n\geq 3$. It might be possible that Theorem \ref{thm:Fluxest} will allow for a better study of this entropy, since the flux itself is (numerically) stable and easy to compute in a number of circumstances. 

\subsection{Outline of the Paper} To end the introduction, we quickly sketch the outline of the paper. In Section \ref{s:prelim}, we give some  Poincar\'e inequalities, technical geometric lemmas and facts about the Robin problem in rough domains which will be useful later in the paper. Some of these are taken directly from our previous paper \cite{Robin1} or are minor modifications of the work there. 

In Section \ref{s:Annuli} we give some preliminary estimates on how solutions with Neumann or Robin conditions oscillate on annuli. In particular the ``Balance Lemma" (Lemma \ref{l3a3}) is a crucial tool in our proof of \eqref{e:GFestNeumannend} in Theorem \ref{thm:GFest}. In Section \ref{s:neumannend} we complete the proof of the estimates in \eqref{e:GFestNeumannend}. As a consequence of our arguments we also show Corollary \ref{cor:upperestalways}, which gives an upper bound on the Robin Green function even in the ``Dirichlet regime". 

We then finish the proof of Theorem \ref{thm:GFest} by addressing the ``Dirichlet regime" in Section \ref{s:dirichletregime}. This section contains our most delicate estimates (see Lemma \ref{l:iterateforguy}).

As mentioned above in Section \ref{sec:hmproperties} we give some applications of our Green function estimates to the behavior of Robin harmonic measure. Finally in Section \ref{s:Flux} we give a more in depth discussion of the total flux problem and prove Theorem \ref{thm:Fluxest}.

\section{Preliminaries}\label{s:prelim}

\subsection{Some geometric lemmas related to NTA Domains} 

We recall first a localization lemma from \cite{Robin1} that allows us to find one-sided NTA domains of mixed dimension that contain a given $\Omega \cap B(y, r)$, $y \in \d\Omega$, 
and have a diameter comparable to $r$. Here is the statement.

\begin{lemma}\label{l:tentspaces}[Lemma 2.6 in \cite{Robin1}]
Let $(\Omega, \sigma)$ be a one-sided NTA pair of mixed dimension, as in Definition \ref{d:mixed}.
For each $y \in \partial\Omega$ 
and $0 < r \leq \diam(\Omega)$, we can find a one-sided NTA pair of mixed dimension 
$(T(y,r), \sigma_\star)$, with
\begin{equation} \label{2a9}
B(y, r)\cap \Omega \subset T(y,r) \subset B(y,Kr)\cap \Omega,
\end{equation} 
and the restriction of  $\sigma_\star$ to $\d\Omega \cap \d\Omega_\ast$ 
is $\sigma$. Here $K \geq 1$ and the geometric
constants for $(T(y,r), \sigma_\ast)$ 
can be chosen to depend only on $n$ and the geometric constants for
$(\Omega,\sigma)$.
\end{lemma}

We shall call $T(y,r)$ a tent domain; it will be useful on occasions that we need sometimes to localize some estimates. 

 New to this paper, we construct chains of overlapping balls which (quantitatively) avoid a given pole in the domain. These ``croissants" will be extremely useful in our Green function estimates, especially in the Neumann regime. This is very similar to \cite[Lemma 2.41]{HMMTZ} but we include a rapid proof for the sake of completeness.

\begin{lemma}\label{lem:donutsareconnected}
    Let $\Omega$ satisfy \ref{CC1} and \ref{CC2}. For every $K, C_0 > 1$, there exist constants $N> 1$ 
    and $c> 0$  (depending on the constants in \ref{CC1}, \ref{CC2} and on $C_0, K$) such that if
   $0 < r \leq \diam(\Omega)$ 
   and  $x, y, z \in \ol\Omega$  
   satisfy $|x- z| \geq r$ and $|y-z| \geq r$, but $|x-y| \leq Kr$, 
   then there exists a chain of $N+1$ overlapping balls $B_j = B(x_j, r_j)$, $0 \leq j \leq N$, 
   with centers in $x_j \in \overline{\Omega}$ and radii $r_j \geq c r$, 
   such that $x = x_0$, $y= y_N$, and for each $j$, $|z-x_j| \geq C_0 r_j$
   (so that $z \notin C_0 B_j$).
\end{lemma}


Note that 
we will not require that the balls $B_j$ (or any multiple of them) are contained inside $\Omega$, even though
our construction gives this, except naturally for the first and last balls when we need to touch $x$ and $y$.
The proof will also give the information that the $x_j$ are neither too close nor too far from $z$, and more precisely
\begin{equation} \label{2d2}
C^{-1} r \leq |x_j-z| \leq C (|x-z|+|y-z|),
\end{equation}
where $C$ depends only on $n$ and the constants in  \ref{CC1} and \ref{CC2}. The reader should not be too worried
a priori about the large constant $C_0$, because if we can prove the lemma with $C_0 = 10$, and except for the first and last balls that are too close to $\d\Omega$ and require a special treatment (but notice that we can make them very small),
getting a larger $C_0$ amounts to making the balls $B_j$ smaller, and adding more intermediate ones
 to compensate and keep the connectedness. We'll just do it directly and not worry about the dependence of $N$ on
 $C_0$.

\begin{proof}
Since we want to use the NTA property of $\Omega$, we first connect $x$ to a point $x'$ such that
$\delta(x') \geq c_0 r$, where $c_0 > 0$ will be chosen soon. If 
$\delta(x) = \dist(x, \d\Omega) \geq c_0 r$, simply take $x'=x$. Otherwise, 
let $\ol x\in \d\Omega$ be such that $|\ol x - x| \leq \delta(x) \leq c_0 r$, and then
pick $x'$ to be a corkscrew point in $B(\ol x, M c_0 r)$, with $M$ is as in \ref{CC1}. 
If $c_0 \leq r/(2C_0 M)$, we can take $B_0 = B(x, c_0 M r) \subset B(x, C_0^{-1}r)$ and it satisfies 
our constraint. Similarly, we can find a point $y' \in B(y, c_0 Mr) \subset B(y, C_0^{-1}r)$ such that 
$\delta(y') \geq c_0 r$, and now we just need to connect $x'$ to $y'$. This will be easier because they are
reasonably far from $\d\Omega$.

By \ref{CC2}, we can find a path $\gamma$ in $\Omega$ that connects $x'$ to $y'$,  lies in 
$B(z, C_1 |x-z| + C_1 |y-z|) \subset B(x, C_1Kr)$, where $C_1$ depends only on the constants in \ref{CC1} and  \ref{CC2}
(we say this because we want to ensure \eqref{2d2}), and 
and stays at distance more than $3C_2^{-1}r$ from $\d\Omega$.
Here, if we just use the definition \ref{CC2}) as it is, we get that $C_2$ depends also on $c_0$,
because we may start from $x'$ and $y'$ very close to $\d\Omega$, 
but if we want to prove the lower bound in \eqref{2d2}
(we shall actually not need this one), we have to say a bit more. It turns out that 
\ref{CC1} and \ref{CC2} also give, with a little bit more of (classical) work, the existence of a path $\gamma$ that look
like a cigar, and here this means that 
\begin{equation} \label{2d3}
\dist(\xi,\d\Omega) \geq C^{-1} (|\xi-x'| + |\xi - y'|)
\ \text{ for } \xi \in \gamma,
\end{equation}
with $C$ depending only on the constants in \ref{CC1} and \ref{CC2}. In particular this means that
$\dist(\xi,\d\Omega) \geq C^{-1} r$ when $|\xi - z| \leq r/2$, because $x'$ and $y'$ are still far from $z$.
The construction below will then give the full \eqref{2d2}; we leave the details for the sake of fluidity.

We can't always use the path $\gamma$, because maybe it goes very close to $z$, but when this happens we will
simply take a detour. 

Let us first deal with the case when when $\delta(z) \leq 2C^{-1} r$.
Then $\dist(\gamma, z) \geq C^{-1} r$, that is the path does not get too close to $z$, and we can take the 
desired collection of balls, with $c = C_0^{-1}C^{-1}$,  by selecting successive points of $\gamma$, 
at distances roughly $c r$ from the next one. The balls $B(x_j, c r)$ will satisfy the desired constraint
that $z \notin C_0 B_j$. Also, just by counting the number
of points at mutual distances larger than $cr/2$ and that lie in a ball of radius $3Kr$, 
we see that we need less than $N = N(c)$ points $x_j$. 
That is, our curve $\gamma$ may be too long if it was chosen stupidly, 
but we just need a $cr$-connected net of points of $\gamma$ that connect $x'$ to $y'$, 
and any minimal one will have less that $N$ points. 

This settles the case when $\delta(z) \leq 2C^{-1} r$. Otherwise, we proceed as before, 
except that if $\gamma$ meets $B(z, C^{-1} r)$, we  only follow
$\gamma$ from $x'$ to the first point $\xi$ of $\d B(z, C^{-1} r)$, 
and similarly (backwards) from $y'$ to the last point $\zeta$ of $\d B(z, C^{-1} r)$.
We also need to connect $\xi$ to $\zeta$, and we can do this in the sphere $\d B(z, C^{-1} r)$,
which is now far both from $z$ and $\d\Omega$.
Then proceeding as above gives a sequence of points $x_j$, all in $\Omega \sm B(z, C^{-1} r)$
and at distance at least $C^{-1} r$ from $\d\Omega$. As before, we can make sure that 
$|x_{j+1}-x_j| \leq c r$, with $c = C^{-1}/2$, and that we use less than $N = N(c_0)$ points.
The lemma follows.
\end{proof}

In a similar spirit, we will need to know 
that we can remove small balls in the middle of 
1-sided NTA domains and maintain their 1-sided NTA character:

\begin{lemma}\label{l:onesidedpunch}
    Let $\Omega$ be a 1-sided NTA domain and $Y\in \Omega$. 
    Then for $\lambda \in (0, 1/2)$ the domain 
 $\wt{\Omega} = \Omega \backslash \overline{B(Y, \lambda \delta(Y))}$ is also 1-sided NTA, 
 with constants that depend only on $n$ and the original NTA constants of $\Omega$.
\end{lemma}

When $\lambda$ is small, the doubling properties of $\d\Omega$ will get worse, but 
nothing bad happens with the one-sided NTA constants.

\begin{proof}
First we need to find corkscrew balls. Let $Q \in \d\wt\Omega$ and $r >0$ be given, and let us first assume
that $Q \in \d\Omega$. Let $A= A_r(Q) \in \Omega \cap B(Q,r)$ be given by \ref{CC1}, so that 
$\dist(A, \d\Omega)\geq M^{-1} r$. If  $A$ does not work for $\wt\Omega$, with the constant $2M$,
then $B_A =  B(A, r/(2M))$ meets $B_Y = B(Y, \lambda \delta(Y))$.
Since $\delta(Z) \geq \delta(X)/2$ on $B_Y$ (because $\lambda \leq 1/2$) and $\delta(Z) \leq \dist(Z, Q) \leq 2r$
on $B_A$, this forces $r \geq  \delta(Y)/4$.

But also, $\delta(Z) \leq 2 \delta(Y)$ on $B_Y$ and $\delta(Z) \geq (2M)^{-1} r$ on $B_A$, so
$r \leq 4M \delta(Y)$. In this case, any point $A' \in \d B(Y, 2\delta(Y)/3) \cap B(Q, r)$ will be a good corkscrew 
point for  $B(Q,r)$ in $\wt\Omega$; such a point exists because $Q \notin 2 B_Y$ but $B_A$ meets $B_Y$.
So we have corkscrew points.

This was the easy part; we also need to find Harnack chains in $\wt\Omega$. Fortunately, 
given two points $x',y'\in \wt\Omega$, the construction sketched 
for Lemma \ref{lem:donutsareconnected} above to connect $x'$ to $y'$ works here too:
we follow a Harnack chain in $\Omega$, and take a small detour when this chain meets $B_Y$.
The reader may also find a proof in \cite[Lemma 2.41]{HMMTZ}. 
\end{proof}

%

\subsection{Poincar\'e Inequalities} 
We will need two Poincar\'e inequalities from \cite{Robin1}. 
Note that, as it should, the estimates below are not changed when $\sigma$ is 
multiplied by a positive constant. 

\begin{lemma} \label{lem:poincareoverlap}[Lemma 2.3 in \cite{Robin1}]
Let $(\Omega,\sigma)$ be a one sided NTA pair of mixed dimension in $\R^n$.  
There exist a constant $K \geq 1$, that depends only on the geometric constants
for $(\Omega,\sigma)$, and for each $c >0$, a constant $C_c$ that depends only on
the geometric constants for $(\Omega,\sigma)$ and $c$, such that for $x \in \d\Omega$, 
$0 < r \leq K^{-1}\diam(\Omega)$, $u\in W= W^{1,2}(\Omega)$, and $E \subset B(x,2r)\cap \Omega$  
satisfying $|E| \geq c |B(x,r)\cap \Omega|$,
	\begin{equation} \label{2nn}
	\fint_{B(x,r)\cap \d\Omega} |u(y)- \bar u|^{2} \, d\sigma(y) 
	\leq C_c r^2 \fint_{B(x,Kr)\cap \Omega} |\nabla u(z)|^2 \, dz ,
		\end{equation}
where 
$\bar u$ denotes 
the average of $u$ on $E$.
	\end{lemma}
	
	Here and below, when we integrate $u$ on $\d\Omega$, we should really write $\Tr(u)$,
where $\Tr(u) \in L^2(\sigma)$ is the trace of $u$, which with our assumptions is well defined as soon as 
$u \in W^{1,2}(\Omega)$ (i.e., when $\nabla u \in L^2(\Omega)$).
See the discussion near (2.10) in \cite{Robin1}, or directly \cite{DFM23}.

    \begin{lemma}\label{lem:poincareboundary}[Lemma 2.8 in \cite{Robin1}]
Let $(\Omega, \sigma)$ be a one sided NTA pair of mixed dimension in $\R^n$.
There exists $K \geq 1$,  and $C \geq 1$,
depending only on the geometric constants for $(\Omega,\sigma)$ 
such that if $0 \in \d\Omega$, $0 < r \leq \diam(\d \Omega)$, and $E \subset \partial \Omega \cap B(0,r)$ 
are such that $\sigma(E) > 0$ and $u\in W^{1,2}(\Omega\cap B(0, Kr))$, then
		\begin{equation} \label{e:poincaretrace}
\fint_{B(0,r)\cap \Omega} u(y)^2 \, dy 
	\leq C \, \frac{\sigma(B(0,r))}{\sigma(E)}  \,   r^2 \fint_{B(0,K r)\cap \Omega} |\nabla u(y)|^2 \, dy
	+ 2   \fint_{E} u^2 d\sigma .
		\end{equation}
\end{lemma}

\subsection{Previous results on Robin solutions}

We say 
that $u$ is a weak solution to the Robin problem with data $f\in L^2(\partial \Omega)$ in $B(Q, r)$ if \begin{equation}\label{e:weaksol}
\frac{1}{a}\int_{\Omega} A\nabla u \nabla \varphi + \int_{\partial \Omega} u\varphi d\sigma 
= \int_{\partial \Omega} f\varphi d\sigma, 
\qquad \forall \varphi \in C_c^\infty(B(Q, r)); 
\end{equation}
compare to (1.5) in  \cite{Robin1}; here we also localize the notion to $B(Q,r)$, but notice that 
$B(Q,r)$ is not necessarily contained in $\Omega$, and then our test functions are defined across $\d\Omega$.

While we use several estimates on solutions to the Robin problem throughout this paper, the boundary Harnack inequality is used most extensively, and so we reproduce it here. 

\begin{theorem}\label{thm:bdryharnack}[Compare with \cite[Theorem 4.4]{Robin1}]
    Let $K > 1$ be large enough (depending on the geometric constants of $\Omega$) and assume 
    $0 \in \partial \Omega$ and $4K^2r < \diam(\partial \Omega)$. 
    Let $u \in W^{1,2}(\Omega \cap B(0, K^2r))$ be a weak solution to the homogeneous Robin problem in $B(0, K^2r)$ (i.e. satisfying \eqref{e:weaksol} with $f= 0$).  There exists a constant $\theta \in (0,1)$ (depending on the geometric constants of $\Omega$ and the ellipticity of $A$) such that if $ar^{2-n}\sigma(B(0,r)) \leq 1$ then $$\inf_{\Omega \cap B(0,r)} u \geq \theta \sup_{\Omega \cap B(0,r)} u.$$
\end{theorem}
We now state the theorem 
from \cite{Robin1} that gives the definition and the construction of the Robin Green function (the central object of consideration in the present 
paper). We drop the dependence of the parameter $a$ from the notation $G_R^a(X,Y)$ when it is obvious. Let $f,g\in W^{1,2}(\Omega)$. We define the natural  quadratic form associated to the Robin problem by 
\begin{equation}
    b(f,g):= \int_{\Omega}A\nabla f\nabla g+a\int_{\partial \Omega}fg d\sigma. 
\end{equation}

\begin{theorem}[Existence of the Robin Green function] 
\label{t:GreenfunctionExistence}[Theorem 5.6 in \cite{Robin1}]
Let $(\Omega, \sigma)$ be a 1-sided NTA pair of mixed dimension, with $\Omega$ bounded. 
Then for every $y\in \Omega$ there exists a non-negative function $G_R(\cdot,y)$ such that 
	\begin{equation*}
		G_R(\cdot,y)\in W^{1,2}(\Omega\setminus B(y,r))\cap W^{1,1}(\Omega)
		\quad \text{for any $r > 0$,}
	\end{equation*}
	and such that for all $\varphi\in C^{\infty}(\Omega)$
\begin{equation}\label{e:Greenidentity}
		b(G_R(\cdot,y),\varphi)= \varphi(y).
	\end{equation}
	Moreover, if $u$ is a weak solution to the Robin problem with data $f\in L^2(\partial \Omega)$,
	i.e., satisfies \eqref{e:weaksol} for all $\varphi \in C_c^\infty(\R^n)$, then 
	
	\begin{equation}\label{e:representationformula}
		u(x)= a\int_{\partial \Omega}f(y)G_R(x,y)d\sigma(y)  
		\,\,\text{for every } x\in \Omega. 
	\end{equation}
	\end{theorem}


It is not hard to see that if $\wt G_R(x,y)$ denotes the Green function for the adjoint problem where we replace 
the matrix $A(x)$ by its transpose everywhere on $\Omega$, then
\begin{equation}\label{2b8}
\wt G_R(x,y) = G_R(y,x)  \quad \text{ for } x, y \in \Omega, x\neq y.
\end{equation}
This is a rather classical consequence of \eqref{e:Greenidentity}, so we skip the details. The reader may follow,
for instance, Lemma 14.78 in \cite{DFM23}; the main idea is that if we were allowed to use \eqref{e:Greenidentity}
on the function $z \mapsto \wt G(z,x)$, we would get that 
$b(G_R(\cdot,y),\wt G_R(\cdot, x)) = \wt G_R(y,x)$, while if we transpose everything,
\eqref{e:Greenidentity} for the adjoint problem would yield
$b(G_R(\cdot,y),\wt G_R( \cdot, x)) = \wt b(\wt G_R(\cdot,x), G_R(\cdot,y)) = G_R(x,y)$.
A limiting argument is then used to go from bump functions to Dirac masses. Notice that when $A$
is symmetric, we get that $G_R$ is symmetric too.

Also notice that \eqref{e:Greenidentity}, applied with $\phi$ supported away from $y$, 
says that $G_R(\cdot, y)$ is, locally on $\Omega \sm \{ y\}$ a weak solution of our equation
${\rm{div}} A \nabla u = 0$, with homogeneous Robin conditions. 
In particular, away from $y$, it is locally H\"older continuous up to the boundary (see \cite[Theorem 4.1]{Robin1}).
And, because of \eqref{2b8} and the fact that the transpose of $A$ satisfies the same requirements,
$ \wt G_R(x, \cdot)$ is, locally on $\Omega \sm \{ x\}$ a weak solution of the adjoint equation, and 
locally H\"older continuous up to the boundary. 

We may use the fact that, countrary to the Dirichlet Green function, here we have that 
\begin{equation}\label{2b9}
G_R(x,y) > 0 \quad \text{for } x \in \ol \Omega \sm \{ y \}.
\end{equation}
Indeed $G(\cdot,y) > 0$ somewhere (apply \eqref{e:Greenidentity} to a bump function), hence, by Harnack, 
everywhere on $\Omega$. But it is also a weak solution with vanishing Robin data on the boundary, 
so Theorem \ref{thm:bdryharnack} says that it is also positive on the boundary. 
In other words, with the vocabulary of \cite{BBC}, the whole boundary is ``active'' in the present context.

Additionally, if we let $f \equiv 1$ in \eqref{e:representationformula} above, then $u(x) \equiv 1$ 
(by uniqueness of the solution of \eqref{e:weaksol} in $\Omega$, see Theorem 2.5 in \cite{Robin1}) 
and we have the flux condition 
\begin{equation}\label{e:fluxcondition}
a\int_{\partial \Omega} G_R(X,y)\, d\sigma(y) = 1 
\qquad \text{ for every } X\in \Omega. 
\end{equation} 

\subsection{Maximum principles and monotonicity} 
We begin this section by giving a weak maximum principle to a mixed boundary value problem. 

\begin{lemma}\label{lem:mixedmax}
    Let $B= B(x,r) \subset \Omega$ be an open ball such that $\overline{B} \subset \subset \Omega$ and $\Omega \backslash \overline{B}$ is a 1-sided NTA domain. Let $f \in L^2(\Omega)$ and 
   let $u\in W^{1,2}(\Omega\backslash \overline{B})$ satisfy 
    \begin{equation}\label{2a11}
\int_{\Omega} A\nabla u \nabla \varphi + a \int_{\partial \Omega} u\varphi = \int_{\partial \Omega} f\varphi, \qquad \forall \varphi \in C_c^\infty(\mathbb R^n\backslash \overline{B}).
\end{equation}
 Then if $u \geq 0$ on $\partial B$ and $f\geq 0$ in $\partial \Omega$ we have $u \geq 0$ in $\Omega \backslash \overline{B}$.
\end{lemma}

\begin{proof}
As before, our integral on $\d\Omega$ actually concerns the trace of $u$.
 We consider the Sobolev function $u^- \in W^{1,2}(\Omega \backslash \overline{B})$ and note that 
 by assumption its trace on $\d B$ vanishes (of course $\d B$ is smooth enough for the trace to exist and lie in $L^2$).
 We can approximate $u^-$ (for the $W^{1,2}$ norm) 
 by smooth functions $\varphi \in C^\infty(\R^n)$, because $\Omega$ is a one-sided NTA domain 
 and thus an extension domain (see for instance \cite{Robin1}). In addition, we can further
 approximate $u^-$ by functions of $C_c^\infty(\mathbb R^n\backslash \overline{B})$, because 
 its trace vanishes on $\d B$ and by multiplying by simple cut-off functions.
 This allows to apply \eqref{2a11} also with $\varphi = u_-$, and get that
    $$\int_{\Omega} A\nabla u \nabla u^- + a \int_{\partial \Omega} uu^- = \int_{\partial \Omega} fu^- \geq 0.$$ However, observe that $A \nabla u \nabla u^- \leq 0$ and $uu^- \leq 0$ pointwise, with a strict inequality in a set of positive measure unless $\nabla u^- = 0$ almost everywhere in $\Omega$ and $u^-= 0$ almost everywhere in $\partial \Omega$. The result follows from the continuity of $u$ in $\Omega$ (as a weak solution).  
\end{proof}

It will also be useful for us to compare Robin Green functions with varying parameter $a$.

\begin{lemma}\label{lem:changea}
Fix $Y\in \Omega$ and let $u_a(X) := G_R^a(X,Y)$, $a \in (0,+\infty)$, be the Green function with pole at $Y$ and parameter $a$. Then, for $0 < a < b$, 
\begin{equation}\label{2a12} 
u_b(X) < u_a(X)  \quad \text{ for } X \in \ol\Omega \sm \{Y\}.
\end{equation}
 \end{lemma}

 \begin{proof}
   Note that $v(x):= u_a(x) - u_b(x)$. By \eqref{e:Greenidentity},  
  \begin{equation}\label{2a13}
   \int_{\Omega}A\nabla u_a\nabla \varphi+a\int_{\partial \Omega}u_a \varphi d\sigma = \varphi(Y)
\end{equation}
   for $\varphi \in C^{\infty}(\Omega)$, and similarly for $u_b$. We subtract and get that
   $$\int_{\Omega} A\nabla v \nabla \varphi  + a\int_{\partial \Omega} v\varphi d\sigma
   = (b-a)\int_{\partial \Omega} u_b \varphi d\sigma, 
   \qquad \forall \varphi \in C_c^\infty(\mathbb R^n).
   $$

    Since $u_b > 0$ on $\Omega$ and $u_b|_{\partial \Omega}$ is continuous (see \cite[Theorems 4.5 and 5.6]{Robin1}), we can apply the maximum principle (\cite[Corollary 5.3]{Robin1}) with $f = -u_b$
 to obtain that $v \geq (b-a)\inf_{\d \Omega}{u_b}> 0$ (by \eqref{2b9}); we get this 
 first on $\Omega$, and then on $\ol\Omega$ by continuity. 
         \end{proof}

 Arguing as above, we can also compare Dirichlet and Robin Green functions. 

 \begin{lemma}\label{lem:changea2} 
 Denote by $G_D$ the (usual) Green function with Dirichlet conditions on $\d\Omega$.
For any $a\in (0,\infty)$ we have that $$G_R^a(X,Y) > G_D(X,Y), \qquad \forall X\in \ol\Omega \backslash \{Y\}.$$
 \end{lemma}

 \begin{proof}
Here the analogue of \eqref{e:Greenidentity} or \eqref{2a13} for $u =G_D(\cdot,Y)$ is that 
$ \int_{\Omega}A\nabla u \nabla \varphi  = \varphi(Y)$, hence $v = u_a- u$ satisfies
$\int_{\Omega} A\nabla v \nabla \varphi  + a\int_{\partial \Omega} v\varphi d\sigma = 0$.
The maximum principle (\cite[Corollary 5.3]{Robin1}) still applies here, and yields $v \geq 0$ on $\ol\Omega$.
We announced a strict inequality, so take $b > a$ and observe that $G_R^a > G_R^b \geq G_D$.
\end{proof}

\section{Variations of $u = G(\cdot, Y)$  on annuli}\label{s:Annuli}

In this section we fix $Y \in \Omega$ (but the constants in our estimates will not depend on $Y$)
and consider a nonnegative solution $u$ defined in $\Omega \sm \{Y\}$ and with vanishing Robin 
boundary conditions on $\d\Omega$. Of course we will apply this with $u(X) = G_R^a(X,Y)$, but for this section the
main assumptions are that
\begin{equation} \label{3a1}
(\Omega,\sigma) \text{ is a one-sided NTA pair of mixed dimension, as in Definition \ref{d:mixed},}
\end{equation}
\begin{equation} \label{3a2}
\text{$A$ is uniformly elliptic,}
\end{equation}
and
\begin{equation} \label{3a3}
\begin{aligned}
&\text{$u$ is a weak solution of $-{\rm div} A \nabla u = 0$ in $\Omega \sm \{Y\}$, }
\\&
\text{with vanishing Robin condition at the boundary.}
\end{aligned}
\end{equation}
By this we mean that $u \in W^{1,2}(\Omega \sm B(Y,r))$ for $r > 0$ and 
\begin{equation}\label{3a4}
\frac{1}{a}\int_{\Omega \sm \{ Y \}} A\nabla u \nabla \varphi 
+ \int_{\partial \Omega} u\varphi \, d\sigma 
= 0 
\qquad \text{for every } \varphi \in C_c^{\infty}(\R^n \sm \{ Y \}).
\end{equation}
In this section we prove various inequalities on $u$ and $\nabla u$, that we'll combine later to estimate the Green 
function. We start with a Caccioppoli inequality on annuli. 

 \begin{lemma}\label{l3a1} 
Assume \eqref{3a1}, \eqref{3a2}, and \eqref{3a3}.
Then for $0 < r < \diam(\Omega)$,
\begin{equation} \label{3a5}
\int_{\Omega \sm B(Y,2r)} |\nabla u|^2 \leq C r^{-2}\int_{\Omega \cap B(Y,2r) \sm B(Y,r)} u^2,
\end{equation}
where $C$ depends on the ellipticity constants for $A$, but not on the geometry of $\Omega$
(in fact, we only use \eqref{3a1} to guarantee that $\Omega$ is an extension domain).
 \end{lemma}

\begin{proof}
We almost have the same result in \cite{Robin1}, but write a short proof anyway.
We use the equation \eqref{3a4} for $u$, and take $\varphi = \theta^2 u$, 
where $\theta$ is a bump function supported away from
$\ol B(Y,r)$. Since $\varphi$ is well approximated by smooth functions, we also get that
\begin{equation} \label{3a6}
\frac{1}{a}\int_{\Omega \sm \{ Y \}} A\nabla u \nabla \varphi 
\leq \frac{1}{a}\int_{\Omega \sm \{ Y \}} A\nabla u \nabla \varphi 
+ \int_{\partial \Omega} u\varphi d\sigma  = 0.
\end{equation}
Thus the estimate for Robin is even better than the one for Neumann. 
We pick $\theta$ so that $\theta = 1$ on $\R^n \sm B(0,2r)$, $0 \leq \theta \leq 1$ everywhere,
and $|\nabla \theta| \leq 2 r^{-1}$. We write $\nabla \varphi = \theta^2 \nabla u + 2 u \theta \nabla \theta$,
get that 
$$
\int \theta^2 A \nabla u \nabla u + \int 2 u \theta A \nabla u \nabla \theta 
= \int_{\Omega \sm \{ Y \}} A\nabla u \nabla \varphi  \leq 0, 
$$
so 
$$
\begin{aligned}
\int \theta^2 |\nabla u |^2 &\leq C \int \theta^2 A \nabla u \nabla u 
\leq 2C \Big| \int u \theta A \nabla u \nabla \theta \Big|
\\
&\leq 2 C ||A||_\infty 
\Big\{ \int_{B(Y,2r) \sm B(0,Y)} u^2 |\nabla \theta|^2 \Big\}^{1/2}
\Big\{ \int_{B(Y,2r) \sm B(0,Y)} \theta^2 |\nabla u|^2 \Big\}^{1/2}
\\
& \leq C  \Big\{ \int_{B(Y,2r) \sm B(Y,r)} r^{-2} u^2 \Big\}^{1/2}
\Big\{ \int_{B(Y,2r) \sm B(Y,r)} \theta^2 |\nabla u|^2 \Big\}^{1/2}.
\end{aligned}
$$
We notice that the second integral is equal to the left-hand side,
simplify, square, and get \eqref{3a5}.
\end{proof}

\begin{remark}\label{rk31}
With almost the same proof, we also get that
\begin{equation} \label{3a7}
\int_{\Omega \sm B(Y,2r)} |\nabla u|^2 \leq C r^{-2}\int_{\Omega \cap B(Y,2r) \sm B(Y,r)} (u-u_0)^2,
\end{equation}
where $u_0 = \inf_{B(Y,2r) \sm B(Y,r)} u$.
First observe that the infimum makes sense, because $u$ is continuous on $\Omega$ and even up to the boundary.
For the proof, we now take $\varphi = \theta^2 (u-u_0)_+$, which vanishes on
$\ol B(Y,r)$ and is equal to $\theta^2 (u-u_0)$ on $B(Y,2r) \sm B(Y,r)$. We still get \eqref{3a6}, and 
now $\nabla \varphi = \theta^2 \nabla u + 2 (u-u_0) \theta \nabla \theta$, so we can follow the rest of the 
proof with $u$ replaced by $u-u_0$ and get the result.
\end{remark}

\ms
Next we use Lemma \ref{lem:donutsareconnected} and the boundary Harnack inequality to prove that,
in Neumann-like situations, $u$ does not vary too much in not-too-thick annuli centered at $Y$. Set
\begin{equation}\label{3a8} 
\cA(r) = \Omega \cap B(Y,5r) \sm B(Y,r)
\end{equation}
for $0 < r < \diam(\Omega)/4$ (we want to make sure that $\cA(r) \neq \emptyset$), and then
\begin{equation}\label{3a9}
m(r) = \inf_{X \in \cA(r)} u(X) \ \text{ and } \ M(r) = \sup_{X \in \cA(r)} u(X).
\end{equation}
 
 \begin{lemma}\label{l3a2} 
Assume \eqref{3a1}, \eqref{3a2}, and \eqref{3a3}, and also that 
\begin{equation}\label{3a10}
a r^{2-n} \sigma(\d\Omega \cap B(Y,10r)) \leq 1.
\end{equation}
There is a constant $C_m \geq 1$, 
that depends only on the geometric constants 
for $(\Omega, \sigma)$ and the ellipticity constant for $A$, such that 
then
\begin{equation}\label{3a11} 
u(X) \leq C_m m(r) \ \text{ for  $X \in \cA(r)$}.
\end{equation}
 \end{lemma}

That is, $u$ does not vary much on $\cA(r)$. 
This will be useful to simply further discussions. The quantity 
\begin{equation} \label{3c12}
I_Y(r) = a r^{2-n} \sigma(\d\Omega \cap B(Y,r))
\end{equation}
will appear a lot in the sequel. It is dimensionless (as much as possible in our mixed dimension context),
and we see it as an index that describes the position of $B(Y, r)$ in the Neumann-Dirichlet scale. 
Note, but do not be worried by, the fact that $I_Y(r) = 0$ when $r \leq \delta(Y)$. In most cases we will take $Y \in \d\Omega$.
And we shall often use the fact that for $0 < s < r \leq \diam(\Omega)$,
\begin{equation}\label{3b12}
I_Y(s) := a s^{2-n} \sigma(\d\Omega \cap B(Y,s)) \leq C (s/r)^{d+2-n} I_Y(r) \leq C  I_Y(r),
\end{equation}
where $d > n-2$ is as in \eqref{1n3}.
Of course this is trivial when $s \leq \delta(Y)$; otherwise let $\xi \in \d\Omega$ be such that 
$|\xi-Y| = \delta(Y) \leq s$, and observe that by \eqref{1n3}, 
$$
\sigma(\d\Omega \cap B(Y,s)) \leq \sigma(\d\Omega \cap B(\xi,2s))
\leq c_d^{-1} (s/r)^d  \sigma(\d\Omega \cap B(\xi,2r)).
$$
We may assume that $r \geq 2s$, because otherwise $I_Y(s) \leq 2^{n-2} I_Y(r)$ by definition,
and then, by the doubling property of $\sigma$, 
$$ 
\sigma(\d\Omega \cap B(\xi,2r)) \leq C \sigma(\d\Omega \cap B(\xi,r/2)) \leq C \sigma(\d\Omega \cap B(Y,r)))
$$
because $|\xi-Y|  \leq s \leq r/2$. Hence $\sigma(\d\Omega \cap B(Y,s)) \leq C (s/r)^d \sigma(\d\Omega \cap B(Y,r)))$,
and \eqref{3b12} follows.

\begin{proof} 
Let $X, Z \in \cA(r)$ be given; we want to show that $u(Z) \leq C u(X)$. 
Set $C_0 = 12 K^2$, where $K$ is as in Theorem \ref{thm:bdryharnack}.
By Lemma \ref{lem:donutsareconnected}, we can connect $X$ to $Z$
by a chain of less than $C = C(K)$ overlapping balls $B_j = B(X_j, r_j)$, so that
$X_j \in \ol\Omega$ and  $Y \notin C_0 B_j$, and \eqref{3a11} will follow as soon as we prove that 
for all $j$,
\begin{equation}\label{3a12} 
\sup_{B_j} u \leq C \inf_{B_j} u,
\end{equation}
for some $C$ that depends only on the geometric and ellipticity constants.
If $2B_j \subset \Omega$, this follows from the regular Harnack inequality, so we
may assume that $2B_j$ meets $\d \Omega$ at some point $\xi_j$. 

For this it will be enough to apply Theorem \ref{thm:bdryharnack} to $u$ and the ball $B(\xi_j, 3r_j)$.
Note that Theorem \ref{thm:bdryharnack} only requires that $u$ is a weak solution in $\Omega \cap B(\xi_j, 3K^2 r_j)$ which is good since $u$ is not a solution near $Y$ but the condition that $Y\notin C_0B_j$ guarantees that $u$ is a weak solution in $\Omega \cap B(\xi_j, 3K^2 r_j)$.

There is a last condition to check, namely that 
$I_{\xi_j}(3r_j) : = a (3r_j)^{2-n} \sigma(B(\xi_j, 3r_j)) \leq 1$. 
First observe that by the construction in Lemma \ref{lem:donutsareconnected}, all the $\xi_j$ lie in $B(Y, R)$,
where $R = Cr$ with a constant $C$ that depends only on the Harnack chain constant for $\Omega$, 
and not on $C_0$. 

Next, $ \sigma(B(\xi_j, 3r_j)) \leq C (r_j/ R)^{d}  \sigma(B(\xi_j, R))$ by the decay property \eqref{1n3}.
First assume $B(Y, 6r)$ meets $\d\Omega$ (this will be the main case). Then by the doubling property
\eqref{1n4}, $\sigma(B(\xi_j, R)) \leq C \sigma(B(Y, 10r))$ and hence
\begin{equation}\label{3b14}
\begin{split}
I_{\xi_j}(3r_j) = a (3r_j)^{2-n} \sigma(B(\xi_j, 3r_j)) \leq C a r_j^{2-n} (r_j/ R)^{d}  \sigma(B(Y, 10r))
\\
= C (r_j/r)^{2-n} (r_j/ R)^{d} I_{Y}(10r)
\leq C (r_j/R)^{d+2-n} I_{Y}(10r) \leq C (r_j/R)^{d+2-n}
\end{split}
\end{equation}
by \eqref{3a10}.
But  $r_j/R \leq C_0^{-1}$ because $R \geq |Y-X_j| \geq C_0 r_j$ (recall that $Y \notin C_0 B_j$).
Recall also that $d > n-2$,
and now we may assume that $C_0$ was chosen so large, compared to the Harnack chain constant for $\Omega$, that
$(r_j/ R)^{d+2-n} \leq C_0^{n-d-2}$ is much smaller than the constant $C$ above; then $I_{\xi_j}(3r_j)\leq 1$ 
as desired, we can apply Theorem~\ref{thm:bdryharnack}, we obtain \eqref{3a12}, and 
\eqref{3a11} follows.

We are left with the case when $B(Y, 6r) \subset \d\Omega$. In this case $\cA(r)$ lies well inside of $\Omega$,
and \eqref{3a11} follows from the classical Harnack inequality. This completes our proof of Lemma~\ref{l3a2}.
\end{proof}

\begin{remark} \label{rk32}
The conclusions of Lemma \ref{l3a2} stay true when we only assume that 
$a r^{2-n} \sigma(\d\Omega \cap B(Y,10r)) \leq C_1$ instead of \eqref{3a10}, with, for instance, 
$C_1$ depending only on geometric constants,
and then of course the constant $C_m$ in \eqref{3a11} becomes larger, depending on $C_1$.
We just keep the same proof and take $C_0$ even larger in the last lines of the proof, except that
now \eqref{3b14} yields $I_{\xi_j}(3r_j) \leq C (r_j/R)^{d+2-n} C_1$, 
and we need to take $C_0$ a little bit larger to compensate for the extra $C_1$.
\end{remark}

\ms
The next tool of this section is a lemma that tells us that, at least in the Neumann case,
the amount of energy $\int A \nabla u \nabla u$ localized in the region where $u(X)$ lies in an interval
is proportional to the length of this interval in the range of $u$, and with an additional error term for Robin solutions. 
This will be very useful to estimate the  variations of $u$ in annular regions, even though some of the
authors are disappointed that in the present paper we do not use the full power of the lemma. In fact, the lemma is even stronger than it appears here as the regions do not have to be annular as soon as some mild geometric conditions are satisfied, we just have to be sure of the separation of values on the inner and outer boundary. 

\begin{lemma}\label{l3a3}[The balance lemma.] 
Assume  \eqref{3a1}, \eqref{3a2}, and \eqref{3a3}, and let 
$0 < r < R \leq \diam(\Omega)/3$ be given. Assume $Y\in B(X,r/2)$ and
set $\tilde{\Omega} := \Omega \cap(B(X,R) \sm B(X,r))$, 
$\d_{\text{out}} = \Omega \cap \d B(X,R)$, and $\d_{\text{in}} = \Omega \cap \d B(X,r)$. Assume that 
$$\sup_{\d_{\text{out}}} u < \inf_{\d_{\text{in}}}  u$$
and then take numbers $t_i$, $1 \leq i \leq 4$, so that
\begin{equation} \label{3a13} 
\sup_{\d_{\text{out}}} u \leq t_1 < t_2 < t_3 < t_4 \leq \inf_{\d_{\text{in}}}  u,
\end{equation}
and let $U_{12} = \big\{ x\in \tilde{\Omega}\, ; \, t_1 < u(x) < t_2\big\}$ and
$U_{34} = \big\{ x\in \tilde{\Omega}\, ; \, t_3 < u(x) < t_4\big\}$. Then
\begin{equation}\label{3a14} 
\frac{1}{t_2 - t_1} \int_{U_{12}} A \nabla u \nabla u
\leq
\frac{1}{t_4 - t_3} \int_{U_{34}} A \nabla u \nabla u
\leq 
\frac{1}{t_2 - t_1} \int_{U_{12}} A \nabla u \nabla u
+ a\int_{\tilde{\Omega} \cap \d\Omega}  u(x) d\sigma(x).
\end{equation}
The estimate \eqref{3a14} also holds when  
\begin{equation} \label{3b17}
R = \diam(\Omega) \ \text{ and } \  0 = t_1 < t_2 < t_3 < t_4 \leq \inf_{\d_{\text{in}}}  u
\end{equation}
(and so there is no $\d_{\text{out}}$).
\end{lemma}

%

This lemma seems to be new, even though its main idea, of using a competitor based on the values of
$u$ rather than the geometry, is not. 

Note that the lemma does not require $\tilde{\Omega}$ to be centered at $Y$,
and the proof does not use the fact that $u$ is a solution in $B(X,r/2) \sm \{Y\}$, say, but this is the context where we shall use it.
The fact that it is a solution in a domain a little larger than $\tilde{\Omega}$ is useful, or at least convenient,
because this way we know that $u$ is H\"older continuous all the way to 
$\d_{\text{in}}$ and $\d_{\text{out}}$, and for instance \eqref{3a13} is then easier to define.

Notice that  when $a=0$ (the purely Neumann case), we simply get that 
$$
\frac{1}{t_2 - t_1} \int_{U_{12}} A \nabla u \nabla u =
\frac{1}{t_4 - t_3} \int_{U_{34}} A \nabla u \nabla u.
$$ 
We will use the lemma only at small scales, close to Neumann, and will see \eqref{3a14} as a perturbation of this.
Finally observe that the two inequalities of \eqref{3a14} have versions for subsolutions or supersolutions, i.e.,
$h$ and $\varphi$ are nonnegative.

\begin{proof}
We intend to use the fact that by the definition \eqref{3a4}, 
\begin{equation} \label{3a15}
\int_{\Omega \sm \{ Y \}} A\nabla u \nabla \varphi 
+ a \int_{\partial \Omega} u\varphi \, d\sigma = 0 
\end{equation}
for every $\varphi \in C^\infty_c(\R^n \sm \{Y \})$.
On $\tilde{\Omega}$ we wish to take $\varphi$ of the form $\varphi = h \circ u$, where 
\begin{equation} \label{3a16}
h(t) = 0 \ \text{ when $t \leq t_1$ or $t \geq t_4$,} 
\end{equation}
\begin{equation} \label{3a17}
h(t) = \frac{t-t_1}{t_2 - t_1}
\ \text{ on } [t_1, t_2],
\end{equation}
\begin{equation} \label{3a18}
h(t) = 1
\ \text{ on } [t_2, t_3], 
\end{equation}
and 
\begin{equation} \label{3a19}
h(t) = 1 - \frac{t-t_3}{t_4 - t_3}
\ \text{ on } [t_3,t_4].
\end{equation}
Of course we need to define $\varphi$ in $\tilde{\Omega}^c\cap \Omega$; we can do so by letting $\varphi = 0$ on $\tilde{\Omega}^c\cap \Omega$ (note that $\varphi = 0$ on $\partial_{\text{in}}$ and $\partial_{\text{out}}$). In particular $\varphi$ is supported away from $Y$. We deal with issues of smoothness momentarily, but for now suppose that \eqref{3a15} holds for this $\varphi$. 

 Since 
$\nabla \varphi(X) = h'(u(X)) \nabla u(X)$, we get that
\begin{equation} \label{3a20}
\frac{1}{t_2 - t_1} \int_{U_{12}} A \nabla u \nabla u 
- \frac{1}{t_4 - t_3} \int_{U_{34}} A \nabla u \nabla u
+ a \int_{\tilde{\Omega} \cap \d\Omega}  h(u(x)) u(x)  d\sigma(x) = 0
\end{equation}
 then \eqref{3a14} follows, because $u \geq 0$ and $0 \leq h \leq 1$.

We now turn to the issue that $\varphi$ is not $C_c^\infty$. Note, by \eqref{3a13}, that $\varphi$
it is continuous across $\d_{\text{in}}$ and $\d_{\text{out}}$ (if $\d_{\text{out}} \neq \emptyset$), and more precisely it has vanishing traces
on $\partial\tilde{\Omega}\cap \Omega$. In fact, to be completely safe, we can first make $h$ 
start at some $t'_1 > t_1$ (very close to $t_1$) and end  at $t'_4 < t_4$ (and then we take a limit),
and this way $\varphi$ is even smooth (in fact vanishes) on neighborhoods of $\d_{\text{in}}$ and $\d_{\text{out}}$.
In light of this we may assume $\varphi \in W^{1,2}(\Omega)$ and, using that $\Omega$ is an extension domain, we can approximate it by $C_c^\infty(\mathbb R^n\sm \{Y\})$ functions (in the $W^{1,2}$ norm) and then 
we can take limits in \eqref{3a15}.
So \eqref{3a15} holds for our $\varphi$, and Lemma \ref{l3a3} follows.
\end{proof}

\ms
We conclude 
this section with an estimate, based on the Poincar\'e inequalities, 
that will allow to bound the variations of $u$
on thick annuli $\Omega \cap B(Y,R) \sm B(Y,r)$ in terms of $\int |\nabla u|^2$.
We take the same notation and assumptions as in Lemma \ref{l3a2}.

\begin{lemma}\label{l3a4} 
Asume  \eqref{3a1}, \eqref{3a2}, and \eqref{3a3}, and let 
$0 < r < R \leq \diam(\Omega)/3$. 
Also assume that 
\begin{equation}\label{3a21}
a R^{2-n} \sigma(\d\Omega \cap B(Y,10R)) \leq C_1.
\end{equation}
Let $m(r)$ and $m(R)$ be as in \eqref{3a9}. Then
\begin{equation} \label{3a22}
|m(r) -m(R)| \leq C m(R)
+ \tilde{C}  r^{1- \frac{n}{2}} \Big\{ \int_{\Omega \cap B(Y, K R) \sm B(Y, r/3)} |\nabla u|^2 \Big\}^{1/2},
\end{equation}
with constants $\tilde{C}, K \geq 1$ that depends only on the geometric constants for $(\Omega,\sigma)$, and
$C$ that depends only on the ellipticity constant for $A$, the 
geometric constants for $(\Omega,\sigma)$, and $C_1$.
\end{lemma}

\begin{proof}
This will be an easy consequence of Poincar\'e inequalities; we will only use the fact that 
$u$ is a solution late in the proof, when we apply Lemma \ref{l3a2} to relate $m(r)$ and $M(R)$
to averages of $u$ on balls. 
For $r \leq \rho \leq R$, we first choose a ball $B_\rho$, with
\begin{equation} \label{3a23}
B_\rho = B(Z_\rho, C^{-1} \rho) \subset \cA(\rho) \cap \Omega 
:= \Omega \cap (B(Y,5\rho) \sm B(Y,\rho)).
\end{equation}
To define $Z_\rho$, we first pick $Z \in \Omega \cap \d B(Y, 2\rho)$; such a point exists because $\Omega$ is connected and contains both $Y$ and points that are far from $Y$. If $\dist(Z, \d\Omega) \geq \rho/10$, we can take $Z_\rho = Z$. Otherwise, select $z\in \d\Omega$
such that $|z-Z| \leq \rho/10$, and take for $B_\rho$ a corkscrew ball for $B(z, \rho/2)$ (i.e. take $Z_\rho$ to be the corresponding corkscrew point).
So $B_\rho$ exists.

Now 
we claim that for $r \leq s \leq t \leq R$, with $t \leq 2s$, 
\begin{equation} \label{3a24}
\Big|\fint_{B_s} u-\fint_{B_t} u \Big| \leq C t^{1- \frac{n}{2}}  \Big\{\int_{\Omega \cap B(Y, Kt) \sm B(Y, s/3)} |\nabla u|^2 \Big\}^{1/2},
\end{equation}
where $K$ depends on the geometric constants. For this we shall connect $Z_s$ to $Z_t$ 
by a Harnack chain. Indeed by Lemma \ref{l:onesidedpunch}, applied with $\lambda = 1/3$,
$\Omega \sm \ol B(Y, s/3)$ satisfies \ref{CC2}, so we can find balls $B_j  = B(X_j,r_j)$, 
$0 \leq j \leq N$, with
$B_0 = B_s$ and $B_N = B_t$, every $B_{j}$, $1 \leq j \leq N$, meets $B_{j-1}$, all the 
radii $r_j$ are comparable to $t$, and $B(X_j, 2r_j) \subset \Omega \sm \ol B(Y, s/3)$
for all $j$. Notice that $B(X_j, 2r_j) \subset B(Y, Kt)$ because $N$ and each $t^{-1} r_j$
are bounded by constants that depend only on the geometric constants.

Then by the Poincar\'e inequality on $B(X_j, 2r_j) \cup B(X_{j-1}, 2r_{j-1})$, 
\begin{eqnarray} \label{3a25}
\Big|\fint_{B(X_j, r_j)} u - \fint_{B(X_{j-1}, r_{j-1})} u\Big|
&\leq& C \, t \Big\{ \fint_{B(X_j, 2r_j) \cup B(X_{j-1}, 2r_{j-1})} |\nabla u|^2 \Big\}^{1/2}
\nn\\
&\leq& C \, t^{1- \frac{n}{2}} 
\Big\{ \int_{B(X_j, 2r_j) \cup B(X_{j-1}, 2r_{j-1})} |\nabla u|^2 \Big\}^{1/2}.
\end{eqnarray}
Now notice that $B(X_j, 2r_j) \cup B(X_{j-1}, 2r_{j-1}) \subset \Omega \cap B(Y, Kt) \sm B(Y, s/3)$;
we sum all these inequalities and get that
\begin{equation} \label{3a26}
\Big|\fint_{B_t} u - \fint_{B_s} u\Big| \leq C \, t^{1- \frac{n}{2}} 
\Big\{ \int_{\Omega \cap B(Y, Kt) \sm B(Y, s/3)} |\nabla u|^2 \Big\}^{1/2},
\end{equation}
which is \eqref{3a24}. Now we connect $r$ to $R$ by a string of intermediate radii
in  geometric progression, and sum again. Let us not try to optimize too much, say that the
$\int_{\Omega \cap B(Y, Kt) \sm B(Y, s/3)} |\nabla u|^2$ are all less than 
$\int_{\Omega \cap B(Y, KR) \sm B(Y, r/3)} |\nabla u|^2$, and $t^{1- \frac{n}{2}}$
is largest for $t=r$; this gives
\begin{equation} \label{3a27}
\Big|\fint_{B_R} u - \fint_{B_r} u\Big| 
\leq C r^{1- \frac{n}{2}} \Big\{ \int_{\Omega \cap  B(Y, K R) \sm B(Y, r/3)} |\nabla u|^2 \Big\}^{1/2}.
\end{equation}
This is almost \eqref{3a22}. We may assume that $m(r) > m(R)$, because otherwise \eqref{3a22} is trivial,
and now we observe that
$$
m(r) \leq \fint_{B_r} u  \leq \fint_{B_R} u + \Big|\fint_{B_R} u - \fint_{B_r} u\Big|
\leq \sup_{B_R} u +  \Big|\fint_{B_R} u - \fint_{B_r} u\Big|
$$
Now \eqref{3a22} and the lemma follow from \eqref{3a27}, 
 Lemma \ref{l3a2} and Remark \ref{rk32}.
 \end{proof}

\section{Green Function estimates when $a \sigma(\partial\Omega) \leq \diam(\Omega)^{n-2}$.}\label{s:neumannend}

Recall the critical length scale 
\begin{equation} \label{4b1} 
\rho := (a\sigma(\partial\Omega))^{\frac{1}{n-2}}.
\end{equation}

Generally speaking, we showed in \cite{Robin1} that the Robin boundary condition in a ball $B(x,r)$ 
centered on $\Omega$ tends to behave more like a Neumann condition  when the quantity 
$I_x(r) = a\sigma(B(x,r)) r^{2-n}$ is small, and more like a Dirichlet condition  when it 
is large. 
In this section we prove the estimates of Theorem \ref{thm:GFest} when 
$\rho \lesssim \diam(\Omega)$, i.e. when the quantity $a\sigma(B(x,r)) r^{2-n}$ is already small
at the scale of $\diam(\Omega)$; note that $I_x(r) = a\sigma(B(x,r)) r^{2-n}$ gets smaller at smaller scales,
by \eqref{3b12}. So we only expect Neumann-like behavior in this section.

Let us recall the salient part of Theorem \ref{thm:GFest} in the present case. 
We keep the same general assumptions on $(\Omega,\sigma)$ (one sided pair of mixed dimension)
and $A$ (elliptic). 

\begin{theorem}\label{l:neumannregime}
Let $(\Omega,\sigma)$ satisfy Definition \ref{d:mixed}, let $a\in (0,\infty)$ be such that 
$a \sigma(\partial\Omega) \leq \diam(\Omega)^{n-2}$,
and let $G_R$ be the Robin Green function with parameter $a$ for the operator 
$\mathrm{div} A\nabla$.

There exists a constant $C > 0$ (depending on the geometric constants of and the ellipticity of $A$) 
such that for $X, Y\in \Omega$, $X \neq Y$,
\begin{equation}\label{e:Neumannregimefar} 
    C^{-1}\rho^{2-n}\leq G_R(X,Y)\leq C\rho^{2-n},\qquad \text{ when } 
    |X-Y| \geq \rho,
\end{equation}
and 
\begin{equation}\label{e:Neumannregimeclose}
    C^{-1}|X-Y|^{2-n} \leq G_R(X,Y) \leq C|X-Y|^{2-n}, \qquad \text{ when } 
    |X-Y| \leq \rho.
\end{equation}
\end{theorem}

\ms 
\begin{proof}
The rest of this section is devoted to the proof, which we break 
up into four parts, corresponding to the lower and upper estimates when $|X-Y| \geq \rho$ and when $|X-Y|\leq \rho$.

\ms 
\noindent {\bf Lower estimate in \eqref{e:Neumannregimeclose}}.

In fact we will directly check that 
\begin{equation}\label{e:lowerboundforall} 
G_R(X,Y) \geq C^{-1}|X-Y|^{2-n}, \qquad \text{ for } 
X,Y\in \Omega, X \neq Y,
\end{equation} 
which will be useful to us later. 

Let $K > 1$ be a large constant (depending on the geometric constants \ref{CC1}, \ref{CC2}) to be fixed later. 
By  \cite{GW} and Lemma \ref{lem:changea2}, 
there exists a constant $C > 0$ (depending only on $K > 1$) such that if $\delta(X), \delta(Y) \geq \frac{1}{K}|X-Y|$, then $G_R(X,Y)\geq G_D(X,Y) \geq C^{-1}|X-Y|^{2-d}$.

Otherwise, if $X$ (or $Y$, or both, the argument is the same) 
satisfies $\delta(X) \leq \frac{1}{K}|X-Y|$ then we shall 
apply the boundary Harnack principle, Theorem \ref{thm:bdryharnack}, to replace $X$ by a corkscrew point $A_X$
close to $X$. 
Pick $Q \in \d\Omega$ such that $|Q-X| = \delta(X)$, set $r = 2|X-Y|/K$ and 
$B_X = B(Q,r)$, and finally let $A_X$ be a corkscrew point for $B_X$.
If $K$ here is large enough, we can apply Theorem \ref{thm:bdryharnack} to $B_X$, because 
$r \leq 2 K^{-1} \diam(\Omega)$ and hence by \eqref{3b12} and our assumption on $a \sigma(\partial\Omega)$, $I_Q(r) \leq 1$ (as long as $K$ is large enough). 
 Theorem \ref{thm:bdryharnack} yields 
$G_R(X,Y)  \simeq G_R(A_X, Y)$ and in this way can reduce to the case where 
$\delta(X), \delta(Y) \geq |X-Y|/K$. This proves \eqref{e:lowerboundforall}.

%

\medskip
\noindent {\bf Preparation for estimates near $Y$}. 

To simplify notation,  we fix the pole $Y$ and let, as in the previous section.   
\begin{equation}\label{4a6}
u(X) = G_R(X,Y).
\end{equation}

We will need an estimate very near $Y$ to ensure our use of the Balance Lemma. 
We claim that there is a small $\varepsilon > 0$, that is allowed to depend on $Y$, such that 
\begin{equation} \label{4a7}
C^{-1} |X-Y|^{2-n} \leq u(X) \leq C |X-Y|^{2-n} 
\ \hbox{ when } |X-Y| \leq \varepsilon.
\end{equation}
Here $C$ depends only on $n$ and the ellipticity constant for $A$. Indeed, let 
$v$ denote the Dirichlet Green function for the domain $B(Y, \delta(Y))$, and with pole
at $Y$.  Then by  \cite{GW}, for instance, 
$C^{-1} |X-Y|^{2-n} \leq v(X) \leq C |X-Y|^{2-n}$ for $X \in B(Y, \delta(Y)/2)$, 
and \eqref{4a7} will follow as soon as we check that $u-v$ is bounded near $Y$.
However $u-v$ no longer has a singularity near $Y$, i.e., $\mathrm{div} A\nabla (u-v) = 0$ near $Y$,
by \eqref{e:Greenidentity}, applied to functions $\varphi$ supported near $Y$, and the 
analogous property of $v$. 
So $u-v$ is bounded near $Y$ and \eqref{4a7} holds.

Let us estimate energies now. Let $0 < r_1$ be small, and choose
$r_2 = K r_1$ for some large constant $K$ that will be chosen soon. We will need to assure that $r_2<\varepsilon$.
By \eqref{4a7}, $u \leq Cr_1^{2-n}$ on $B(Y, r_2)\sm B(Y, r_1/2)$ if we choose $K$ large enough and $r_1$ small enough (so that $Kr_1$ is still small). 
Then by the Caccioppoli inequality (Lemma \ref{l3a1}) and then \eqref{4a7} we have
$$
 \int_{B(Y, r_2) \sm B(Y, r_1)} |\nabla u|^2 \leq  \int_{\Omega \sm B(Y,r_1)}|\nabla u|^2
 \leq C r_1^{-2} \int_{\Omega \cap B(Y, 2r_1) \sm B(Y,r_1/2)} u^2
 \leq C r_1^{2-n}
$$
so that altogether
\begin{equation}\label{4a9}
\int_{B(Y, r_2) \sm B(Y, r_1)} |\nabla u|^2 \leq C r_1^{2-n}.
\end{equation}

To apply the balance lemma, set
$U_{t_3, t_4} = \big\{ X \in \Omega \sm \{Y\} \, ; \, t_3 < u(X) < t_4 \big\}$
for $0 < t_3 < t_4$ to be chosen shortly. Since $u$ is only unbounded near $Y$, 
$U_{t_3, t_4} \subset B(Y, \varepsilon)$ as soon as $t_3$ is large enough, and then we can use
\eqref{4a7} and its consequences. For instance, if we take $t_4 = C^{-1} r_1^{2-n}$, with $r_1$ small as above, 
and $t_3 = t_4/2$, then by the discussion above, $u(X) > t_4$ on $\d B(r_1)$ (and on the smaller balls)
$u(X) < t_3$ on $\d B(r_2)$ (here is where we pick $K$ large enough), so that $U_{t_3, t_4} \subset B(Y, r_2) \sm B(Y, r_1)$,
with the consequence that 
\begin{equation}\label{4a10}
\frac{1}{t_4-t_3} \int_{U_{t_3, t_4}} |\nabla u|^2 
\leq \frac{1}{t_4-t_3} \int_{B(Y, r_2) \sm B(Y, r_1)} |\nabla u|^2
\leq  \frac{C}{t_4-t_3} \,  r_1^{2-n} \leq C
\end{equation}
by \eqref{4a9}. 
%
We may now apply the easy part of the Balance Lemma \ref{l3a3}.
We can take the simpler second option with $R = \diam(\Omega)$ and $t_1=0$, any $t_2 > 0$, 
and $t_4 > t_3 > t_2$, with $t_4, t_3$ so large that \eqref{4a10} holds. 
Then \eqref{3a14} says that for
\begin{equation}\label{4b11}
U(t_2) = \big\{ X \in \Omega \sm \{Y\} \, ; \, u(X) < t_2 \big\},
\end{equation}
\begin{equation} \label{4a15}
\int_{U(t_2)} | \nabla u |^2
\leq C \int_{U(t_2)} A \nabla u \nabla u \leq C t_2.
\end{equation}

We emphasize that $C > 0$ in \eqref{4a15} depends on the geometry of $\Omega$ and the ellipticity of $A$ but does not depend on the $\epsilon$ or small radii chosen above. It is also independent of $t_2$ and holds for any $t_2 > 0$. 

\ms
\noindent {\bf Upper estimates in \eqref{e:Neumannregimefar} and \eqref{e:Neumannregimeclose}}. 

We finally estimate $u$ in the bulk. 
We start when $X$ is far from $Y$ and check that
\begin{equation} \label{4a16} 
u(X) \leq C (a\sigma(\d\Omega))^{-1} = C \rho^{2-n}\ \text{ for } |X-Y| \geq \diam(\Omega)/4.
\end{equation}
For this we apply Lemma \ref{l3a2} with $r =  \diam(\Omega)/4$; 
the additional assumption \eqref{3a10} is satisfied
(with $1$ replaced by the slightly larger constant $4^{n-2}$, but see Remark \ref{rk32}), 
because $a  \sigma(\d\Omega) \leq \diam(\Omega)^{n-2}$, and  we get that 
$u(X) \leq C m(r)$ for $X \in \Omega \sm B(Y, r) \equiv \Omega \cap \mathcal A(r)$. Then we observe that
$\d\Omega$ contains at least one point $z$ at distance $2r$ from $Y$, and by definition
$m(r) \leq u(y)$ for all $y \in \d\Omega \cap B(z,r)$, just because $B(z,r) \subset \cA(r)$
(see \eqref{3a8} and \eqref{3a9}).  

Now we use the flux estimate \eqref{e:fluxcondition}, i.e., 
the fact that $a\int_{\partial \Omega} G_R(X,y)\, d\sigma(y) = 1$
to estimate an average of $u = G_R^a(\cdot,Y)$:
$$
1 = a\int_{\partial \Omega} G_R(X,y)\, d\sigma(y) \geq a \int_{\partial \Omega \cap B(z,r)} u\, d\sigma
\geq C a \sigma(\partial \Omega \cap B(z,r)) \, m(r) \geq C a \sigma(\d\Omega) \, m(r)
$$
because $\sigma$ is doubling; \eqref{4a16} follows (recall \eqref{4b1}). 

Next we want the upper bound for $X$  closer to $Y$. Recall the quantitaties $m(r)$ and $M(r)$
from \eqref{3a9}. Our assumption that $a \sigma(\partial\Omega) \leq \diam(\Omega)^{n-2}$
implies that $I_Y(r) = a r^{2-n} \sigma(B(Y,r)) \leq C$ for $0 < r \leq \diam(\Omega)$
(see \eqref{3b12}),  which allows us to apply Lemma \ref{l3a2} or Remark \ref{rk32}
for all radii. Thus we get that 
\begin{equation}\label{4a17}
M(r) \leq C_m m(r)  \ \text{ for } 0 < r \leq \diam(\Omega)/4.
\end{equation}

\ms
We will have proven the upper estimates in \eqref{e:Neumannregimefar} and \eqref{e:Neumannregimeclose} if, for some large $C_\infty$, to be chosen soon,
\begin{equation} \label{4a18}
u(X) \leq C_\infty (\rho^{2-n} +  |X-Y|^{2-n}), \forall X \in \Omega \backslash \{Y\}.
\end{equation}

Let us assume that this fails and derive a contradiction. By  \eqref{4a16}, \eqref{4a18} holds on 
$\Omega \sm B(Y,R)$, with $R = \diam(\Omega)/4$.
Set 
$$
r = \inf\big\{ s \geq 0 \, ; \, \text{ \eqref{4a18} holds on } \Omega \sm B(Y, s)\big\};
$$
our assumption for the sake of contradiction is that $r > 0$.
Since $u$ is continuous on $\ol\Omega \sm \{ Y \}$, \eqref{4a18} still holds on $\Omega \sm B(Y, r)$,
but the converse inequality also holds somewhere on $\Omega \cap \d B(Y, r)$. That is,
$\sup_{\Omega \cap \d B(Y, r)}u  = C_\infty (\rho^{2-n} +  r^{2-n})$.

Set $t = C_\infty (\rho^{2-n} +  r^{2-n})$; by construction, $u(X) \leq t$ on 
$\Omega \sm B(Y, r)$, hence $\Omega \sm B(Y, r)$ is contained in $U(t_2)$ for all $t_2 > t$
and \eqref{4a15} says that
\begin{equation} \label{4a19}
\int_{\Omega \sm B(Y,r)} |\nabla u|^2 \leq \lim_{t_2 \downarrow t} \int_{U(t_2)} |\nabla u|^2 \leq C t.
\end{equation}
 By \eqref{4a17}, Lemma \ref{l3a4}, \eqref{4a16}, and \eqref{4a19},
\begin{multline} \label{4a20}
t \leq M(r) \leq C m(r) \leq CM(3r)\leq Cm(3r)\\ \leq C m(R) + C r^{1 - \frac{n}{2}} \Big\{ \int_{\Omega \sm B(Y, r)} |\nabla u|^2 \Big\}^{1/2} 
\leq C \rho^{2-n} + C r^{1 - \frac{n}{2}} \sqrt{t}.
\end{multline}
Note above we use the fact that $m(r)\leq M(3r)$, which follows from the definitions and the fact that $\mathcal A(r)\cap \mathcal A(3r) \neq \emptyset$. 

But $t$ is much larger than $C \rho^{2-n}$ by assumption (and if $C_\infty$ is chosen large enough),
so we are left with $t \leq C r^{2-n}$, a contradiction again.
This proves the desired upper bounds.

\medskip 
\noindent {\bf Lower estimate in \eqref{e:Neumannregimefar}}.

We need to prove that $u(X) \geq C^{-1}\rho^{2-n}$ for $X \in \Omega \sm B(Y, \rho)$,
which with the notation of  \eqref{3a9} translates into 
\begin{equation}\label{4a21}
m(r) \geq C_\ell^{-1}\rho^{2-n} \ \text{ for }  \rho \leq r \leq \diam(\Omega)/4.
\end{equation}
(because the points of $\Omega$ such that $|X - Y| \geq \diam(\Omega/4)$ are also contained
in $\cA(\diam(\Omega/4)$). Notice that we know this estimate for $r = \rho$, by \eqref{e:lowerboundforall},
so a few applications of  Lemma \ref{l3a2} and Remark \ref{rk32} show that it stays true for
$\rho \leq r \leq C \rho$, with $C$ arbitrary large (but then $C_{\ell}$ gets larger).  
Because of this, we may assume that $\rho^{-1} \diam(\Omega)$ is large enough (to be determined later, but depending only on the usual constants).

We start our proof with the case of $r = R := \diam(\Omega)/4$, and for this we use again our estimate
\eqref{e:fluxcondition} on the flux, that says that $a\int_{\partial \Omega} u(x) d\sigma(x) = 1$.
We cut the integral into three pieces, which we estimate independently. 
We start with $\d_1 = \partial \Omega \cap B(Y,\rho)$ 
which we cut into pieces $\d_{1,k}$ where 
$2^{-k-1} \rho \leq |x-Y| \leq 2^{-k} \rho$, and say that
\begin{eqnarray}\label{4a22}
a \int_{\d_1} u d\sigma &=& a \sum_k   \int_{\d_{1,k}} u d\sigma 
\leq C \sum_{k} a \sigma(B(Y, 2^{-k} \rho)) (2^{-k}\rho)^{2-n}
=   C \sum_{k} I_Y(2^{-k}\rho) 
\nn \\
&\leq& C I_Y(\rho) 
\leq C \, (\rho/R)^{d+2-n} I_Y(R) \leq C \, (\rho/R)^{d+2-n}
\leq \frac14
\end{eqnarray}
by the upper estimate in \eqref{e:Neumannregimeclose}, then \eqref{3b12} (twice),
the running assumption for this section, and the assumption that $\rho^{-1} \diam(\Omega)\simeq \rho/R$ is large enough.
Next consider $\d_2 = \partial \Omega \cap B(Y, \eta R) \sm B(Y,\rho)$,
where the small $\eta > 0$ will be chosen soon. Then
\begin{equation}\label{4a23}
a \int_{\d_2} u d\sigma \leq C a \rho^{n-2} \sigma(\d_2) \leq C a \rho^{n-2} \eta^d \sigma(\d\Omega) 
= C \eta^d \leq \frac14
\end{equation}
by the upper estimate in \eqref{e:Neumannregimefar}, \eqref{1n3}, the definition \eqref{4b1} of $\rho$,
and if $\eta$ is chosen small enough. Thus for the remaining part 
$\d_3 =  \partial \Omega \sm B(Y, \eta r)$, we have $a \int_{\d_2} u d\sigma \geq \frac12$.
Yet by a few applications of  Lemma \ref{l3a2} and Remark \ref{rk32}, $u(x) \leq C(\eta) m(R)$
on $\d_3$, so
\begin{equation}\label{4a24}
a \int_{\d_3} u d\sigma \leq C(\eta)  m(R) a \sigma(\d\Omega)
= C(\eta)  m(R) \rho^{n-2},
\end{equation}
so we get that $m(R) \geq C^{-1} \rho^{2-n}$, i.e.,  \eqref{4a21} holds for $r =  \diam(\Omega)/4$.

Now we deal with the general remaining case, and for this we use Lemma \ref{l3a4} to compare
$m(r) $ with $m(R)$.  Here we are afraid that $m(r)$ may be much smaller than $m(R)$,
so we exchange $r$ and $R$ in the last four lines of the proof and get that
\begin{equation} \label{4a25}
|m(r) - m(R)| \leq C m(r)
+ C  r^{1- \frac{n}{2}} \Big\{ \int_{\Omega \cap B(Y, K R) \sm B(Y, r/3)} |\nabla u|^2 \Big\}^{1/2}.
\end{equation}
Then we use Lemma \ref{l3a1} (Caccioppoli) to control the energy integral: 
\begin{equation}\label{4a26}
\int_{\Omega \sm B(Y,r/3)} |\nabla u|^2 \leq C r^{-2}\int_{\Omega \sm B(Y,r/6)} u^2
\leq C r^{n-2} m(r)^2
\end{equation}
by Lemma \ref{l3a2} again. Altogether, $|m(r) - m(R)| \leq C m(r)$, and this forces
$m(r) \geq C^{-1} m(R) \geq C^{-1} \rho^{2-n}$. 
This completes our proof of \eqref{4a21} and Theorem \ref{l:neumannregime}.
\end{proof}

This completes our proof of the estimates in Theorem \ref{thm:GFest}, in the Neumann regime, i.e. when 
$a \sigma(\partial \Omega) \leq \diam(\Omega)^{n-2}$.
Before moving onto the other regimes, we observe that arguing as above, we can actually get an upper estimate on $G^a_R(X,Y)$ also when $a$ is large. 

\begin{corollary}\label{cor:upperestalways}
 Suppose $a \sigma(\partial \Omega) \geq \diam(\Omega)^{n-2}$.
  Then there exists a constant $C> 0$ (depending only on the geometric constants 
    of $\Omega$ and the ellipticity of $A$) such that 
    $$G_R^a(X,Y) \leq C|X-Y|^{2-n}, \qquad \forall X,Y\in \Omega.$$
\end{corollary}

\begin{proof}
    Let $X,Y\in \Omega$ and define $\bar{a}$ to be such that 
    $$|X-Y| = \rho_{\bar{a}}:= (\bar{a}\sigma(\partial \Omega))^{\frac{1}{n-2}}.$$ 
    Note that $\rho_{\bar{a}} = |X-Y|  \leq \diam(\Omega)$, which implies that 
    $\bar{a} \sigma(\partial \Omega) \leq \diam(\Omega)^{n-2}$ and hence $\bar{a} \leq a$.
    Applying Lemma \ref{lem:changea} and the upper bound in \eqref{e:Neumannregimefar} to $G_R^{\bar{a}}(X,Y)$ we have $$G^a_R(X,Y) \leq G_R^{\bar{a}}(X,Y) \leq C\rho_{\bar{a}}^{2-n} \equiv C|X-Y|^{2-n},$$ which is the desired result. 
\end{proof}

This upper bound will be helpful below, but also when proving finer properties of the Robin harmonic measure in Section \ref{sec:hmproperties}.

\section{The Dirichlet regime: $\frac{1}{a} \leq \diam(\Omega)^{2-n}\sigma(\partial\Omega)$}\label{s:dirichletregime}

In this section we complete the proof of Theorem \ref{thm:GFest}.  We did the Neumann regime already, so 
we are left with the two statements that concern the remaining case (the Dirichlet regime) when 
\begin{equation}\label{5a1}
a \sigma(\partial \Omega) \geq \diam(\Omega)^{n-2}.
\end{equation}
Here we compare $G_R(X,Y)$ to $G_D(X,Y)$, perhaps evaluated at the appropriate corkscrew points,
but before we recall the statement, some notation will be useful.

For $X \in \ol\Omega$, we denote by $Q_X$ a point of $\d\Omega$ such that 
$|Q_X-X| = \delta(X) : =  \dist(X, \d\Omega)$, and define $\rho_X$ by
\begin{equation}\label{5a2}
\rho_X =  \sup\big\{ \rho > 0 \, ; \, a \rho^{2-n} \sigma(B(Q_X, \rho)) \leq 1 \big\}
= \sup\big\{ \rho > 0 \, ; \, I_{Q_X}(\rho)  \leq 1 \big\}
\end{equation}
with the notation of \eqref{3b12}.
Notice that $\rho_X > 0$ because \eqref{3b12} says that $\lim_{\rho \to 0} I_{Q_X}(\rho) = 0$,
and that $\rho_X \leq \diam(\Omega)$ because 
taking $\rho > \diam(\Omega)$  yields $I_{Q_X}(\rho) > 1$ by \eqref{5a1}.
Also,
\begin{equation}\label{5a3}
C^{-1} \leq I_{Q_X}(\rho_X) \leq 1
\end{equation}
by \eqref{3b12} and because $\rho \mapsto I_{Q_X}(\rho)$ is semicontinuous.
The worried reader may check that $\rho_X$ is well defined up to a bounded multiplicative factor even when there are multiple choices for  
$Q_X$, by \eqref{3b12} and because $\sigma$ is doubling.
As before, $\rho_X$ will mark the passage from Neumann to Dirichlet.

With this notation, we can reformulate the precise statement  of the missing part of  Theorem \ref{thm:GFest}
as follows (in the introduction the statement is split into two pieces for ease of reading; we merge the two statements here).

\begin{theorem}\label{t:DirichletRegime}
Let $(\Omega, \sigma)$, $A$, and $a$ be as usual, and assume \eqref{5a1}.
For $X, Y \in \Omega$, $X \neq Y$, define 
$r_X = \min( 10^{-1} |X-Y| , \rho_X)$ and $r_Y = \min( 10^{-1} |X-Y| , \rho_Y)$, 
then take $A_X = X$ if $\delta(X) \geq r_X$, and otherwise let $A_X$ be a corkscrew point
for the ball $B(Q_X, r_X)$. Define $A_Y$ similarly. Then  
\begin{equation} \label{5a4} 
C^{-1} G_D(A_X,A_Y) \leq G_R^a(X,Y) \leq CG_D(A_X,A_Y).
\end{equation}
Here $C$ depends only on the geometric constants for $(\Omega, \sigma)$, the dimension $n \geq 3$, 
and the ellipticity constants for $A$. 
\end{theorem}


Our main theorem will actually follow from the following, which may also be more pleasant to manipulate in some cases. 

\begin{theorem}\label{t:GuysConj}
Let $(\Omega, \sigma)$, $A$, and $a$ be as before, and assume \eqref{5a1}.
  For $X,Y\in \Omega$ let $Q_X, Q_Y\in \partial \Omega$ be such that $|Q_X-X| = \delta(X)$ and $|Q_Y- Y| = \delta(Y)$. Suppose in addition that 
\begin{equation}\label{5a5} 
a^{-1} \leq C_0 \delta(X)^{2-n}\sigma(B(Q_X, \delta(X))), \qquad a^{-1} \leq  C_0 \delta(Y)^{2-n}\sigma(B(Q_Y, \delta(Y))),
\end{equation}
where $C_0$ is allowed to depend on the geometric constants for $(\Omega,\sigma)$. 
 Then 
\begin{equation}\label{5a6}
G_D(X,Y) \leq G_R^a(X,Y) \leq CG_D(X,Y),
\end{equation}
where as usual $C > 1$ depends only on $C_0$,  the geometric constants for $(\Omega,\sigma)$, 
and the ellipticity of $A$.
\end{theorem}

\ms 
Let us first see how Theorem \ref{t:GuysConj} implies Theorem \ref{t:DirichletRegime} 
and hence Theorem \ref{thm:GFest}.

\begin{proof}[Theorem \ref{t:GuysConj} implies Theorem \ref{t:DirichletRegime}] 
    Let $X,Y \in \Omega$ be given; we want to prove \eqref{5a4}.
 
 Let us first treat the case when $r_X < \rho_X$. In this case, 
 $r_X = |X-Y|/10$, and by construction $\delta(A_X) \geq C^{-1} r_X$ is rather large.
 Also, since $ \rho_X > r_X$,  we get that $I_{Q_X}(r_X) \leq CI_{Q_X}(\rho_X) \leq C$,
 and then, since $r_X = |X-Y|/10$ and $\sigma$ is doubling, we also get that 
 $I_{Q_Y}(r_X) \leq C$, which by the decay in \eqref{3b12} implies that 
 $\rho_Y \geq C^{-1} r_X \geq C^{-1}  |X-Y|$. Then $\delta(A_X) \geq C^{-1} |X-Y|$
 as well.  In this case,
 \begin{equation}\label{5a7}
C^{-1} |X-Y|^{2-n} \leq G_D(A_X,A_Y) \leq G_R^a(A_X, A_Y) \leq C |X-Y|^{2-n}
\end{equation}
by the classic estimates of Gr\"uter-Widman \cite{GW}, Lemma \ref{lem:changea2}, and 
Corollary \ref{cor:upperestalways}.
 
 In addition, recall that $I_{Q_X}(r_X) \leq C$, i.e., $B(Q_X, r_X)$
 still lies in the Neumann range. This allows us to apply the boundary
  Harnack principle (Theorem \ref{thm:bdryharnack}), maybe to a slightly smaller ball
  (use \eqref{3b12} if needed), and get that $G_R^a(Z, Y) \simeq G_R^a(Z', Y)$ for 
  $Z, Z' \in B = B(Q_X, |X-Y|/C)$. By the regular Harnack inequality,
  we also get this for $Z = X$, even if $X$ happens to be a
  little outside of $B(Q_X, |X-Y|/C)$, because $X$ is a corkscrew point for
  $B(Q_X, 2\delta(X))$ and hence is well connected to $B$ by a Harnack chain.
  And also for $Z' = A_X$, for the same reason. 
  Thus $G_R^a(X, Y) \simeq G_R^a(A_X, Y)$.
  
  Similarly, $I_{Q_Y}(|X-Y|/100) \leq C$, so the same argument shows that 
  $G_R^a(A_X, Y) \simeq G_R^a(A_X, A_Y)$.
  Here we use the symmetry of our assumptions, i.e., \eqref{2b8}, and the fact that we only needed
  $Y$ to lie reasonably far from $Q_X$. Altogether, $G_R^a(X, Y) \sim G_R^a(A_X, A_Y)
  \sim G_D(A_X, A_Y)$ and \eqref{5a4} holds in this case. 
  The argument when $r_Y < \rho_Y$ is the same, so we may now assume that 
  $r_X = \rho_X < |X-Y|/10$ and $r_Y = \rho_Y < |X-Y|/10$.
  
  Our next case is when $r_X = \rho_X \leq \delta(X)$  and $r_Y= \rho_Y \leq \delta(Y)$.
  In this case we just need to check \eqref{5a6}, and since by \eqref{3b12} and \eqref{5a3} we have
 $I_{Q_X}(\delta(X)) \geq C^{-1} I_{Q_X}(\rho_X) \geq C^{-1}$,  \eqref{5a5} holds and
 \eqref{5a6} follows from Theorem \ref{t:GuysConj}.
  
  Next suppose that $r_X = \rho_X  >  \delta(X)$. This is the typical case where we need
  to replace $X$ with $A_X$, the corkscrew point for $B(Q_X, r_X)$. By \eqref{5a3},
  $B(Q_X, r_X)$ is still in the Neumann range, so we may again apply the boundary
  Harnack principle (Theorem \ref{thm:bdryharnack}) to a slightly smaller ball.
  We get as before that $G_R^a(Z, Y) \simeq G_R^a(Z', Y)$ for 
  $Z, Z' \in B = B(Q_X, r_X/C)$, and also, by applying Harnack a few more times if needed,
  for $Z=X$ (this is only needed if $\delta(X) \geq C^{-1} r_X$, and true because $X$ is a corkscrew point 
  at the scale $2\delta(X)$), and $Z'=A_X$ (which is also a corkscrew point at the scale $r_X$). 
  So $G_R^a(X, Y) \simeq G_R^a(A_X, Y)$. The same argument, only needed when $r_Y = \rho_Y  >  \delta(Y)$,
  shows that in turn $G_R^a(A_X, Y) \sim G_R^a(A_X, A_Y)$. Now we want to apply
  Theorem \ref{t:GuysConj} to the pair $(A_X, A_Y)$, so we need to check that they satisfy
  \eqref{5a5}. Once we do this, we have \eqref{5a6} for this pair and \eqref{5a4} follows.
  So we need to show that for $Q = Q_{A_X}$, 
  $$
 a^{-1} \leq C_0  \delta(A_X)^{2-n}\sigma(B(Q, \delta(A_X)));
  $$
  the estimate for $Y$, if needed, would be the same. We already checked this when $A_X = X$.
  Otherwise, we know that $A_X \in B(Q_X, r_X)$, that 
  $\delta(A_X) \sim r_X = \rho_X$, and that $1 \sim I_{Q_X}(r_X) = a r_X^{2-n} \sigma(Q_X, r_X)$.
  The doubling poperty of $\sigma$ implies that $\sigma(Q_X, r_X) \leq C \sigma(B(Q, \delta(A_X))$,
  and the desired inequality follows, with a constant $C_0$ that depends only on the geometric constants, as promised.
\end{proof}

So we are left with proving Theorem \ref{t:GuysConj}. Of course this is the most delicate estimate, and it will be proven by 
 iterating the following key lemma. 

\begin{lemma}\label{l:iterateforguy}
We can find constants $M > 1$, 
$C \geq 1$, and $\beta > 0$,  
(that depend only on the geometric constants of $(\Omega, \sigma)$ and 
the ellipticity constants of $A$), so that if $B(Q, R)$ is a ball centered on $\d\Omega$, 
$u_R$ and $u_D$ are non-negative solutions of $\mathrm{div} A \nabla$ on $B(Q, R)$, 
satisfying respectively Robin and Dirichlet boundary conditions on $B(Q, R)\cap \Omega$, 
if $\delta_0 \leq 1/M$ is such that 
 \begin{equation} \label{5a8}
 I_z(\delta_0 R) =: a (\delta_0 R)^{2-n}\sigma(B(z,\delta_0 R)) > M
 \ \text{ for every } z \in \d\Omega \cap B(Q, 9R/10)
\end{equation}
 and if furthermore $u_R \geq u_D$ in $B(Q,R)$ but for some $\lambda \geq 1$, 
 \begin{equation} \label{5a9}
u_R(X) \leq \lambda u_D(X) \ \text{ for $X\in B(Q,R)$ such that } 
\mathrm{dist}(X, \partial \Omega) \geq \delta_0 R,
\end{equation}
then 
\begin{equation} \label{5a10} 
u_R(\xi) \leq (1+\eta) \lambda u_D(\xi), \ \text{ for all $\xi \in B(Q, 4R/5)$ such that  } 
\mathrm{dist}(\xi, \partial \Omega) \geq \delta_0 R/2, 
\end{equation}
with
\begin{equation}\label{5a10b}
\eta \leq  C \min\Big( M^{-\beta}, e^{-\frac{\beta}{\delta_0}}
+ \sup_{z \in \d\Omega \cap B(Q, 9R/10) } \big[ I_z(\delta_0 R)^{-\beta}\big] \Big).
\end{equation}
\end{lemma}

The statement is complicated, so we should explain a bit.
The general idea is to extend \eqref{5a9}, which controls $u_R$ by $u_D$ as long as we are far enough away from the boundary,
to a region that lies a little bit closer to $\d\Omega$. This is what we do in \eqref{5a10},
but we lose a factor $(1+\eta)$ in the battle, so we hope to make $\eta$ as small as possible.

One might be tempted to simplify the lemma by taking $\delta_0 = 1/M$, say. Then
we just need to assume that $I_z(\delta_0 R) > M$, and the same proof will give us a similar result with $\eta \leq C M^{-\beta}$.
This looks good, but in the application to Theorem \ref{t:GuysConj} we will have to estimate a product of numbers
$(1+\eta_k)$ coming from iterations of the lemma, and we will need the better estimate
\eqref{5a10b} to get an additional decay and make a large product converge. This will force us
to take different $\delta_0$ depending on the scale. This is naturally more complicated, but at the end we get a 
control of $u_R$ by $u_D$ on a region that can get very close to $\d\Omega$.

It is slightly unpleasant that we need a supremum of values of $I_z(\delta_0 R)$;
we could have made a stronger assumption on the single $I_Q(R)$, but this would have added
unpleasant correction terms $\delta_0^K$ for some $K$ that depends on the doubling constant,
so we decided to stick to the unpleasant supremum. 

\ms
Assuming Lemma \ref{l:iterateforguy}, the proof of Theorem \ref{t:GuysConj} goes as follows.

\begin{proof}[Proof of Theorem \ref{t:GuysConj} assuming Lemma \ref{l:iterateforguy}]
Our first note is that $G_R \geq G_D$  always, 
by Lemma~\ref{lem:changea2}.

Fix, $X, Y \in \Omega$, $X \neq Y$, and let $R_0 =  |X-Y|/10$, 
and first assume that  $\delta(X), \delta(Y) \geq c_0 R_0$, where the small $c_0 > 0$
will be chosen soon, depending on $M$ in Lemma \ref{l:iterateforguy}.
Then by Corollary \ref{cor:upperestalways}, Lemma~\ref{lem:changea2},  
and the Gr\"uter-Widman lower bound on
$G_D$, see \cite{GW}, we have that 
\begin{equation} \label{5a11}
G_R(X,Y)\simeq G_D(X,Y) \simeq |X-Y|^{2-n}
\end{equation}
(with constants that depend on $c_0$, but this is all right), as required for \eqref{5a6}.

So we may assume that 
\begin{equation} \label{5e13}
\delta(X) \leq c_0 R_0/2,
\end{equation}
and for the moment let us assume for simplicity that  $\delta(Y)\geq c_0 R_0$.
We want to replace $X$ by a point $X_0$ that lies further from $\d\Omega$.
Choose $Q = Q_{X} \in \d\Omega$ such that $|Q_{X}-X| = \delta(X)$, set $B_0 := B(Q_{X}, R_0)$
(a large ball), and choose a corkscrew point $X_0$ for $B(Q_X, R_0)$. 
We want to estimate $w = \frac{u_R}{u_D}$, and we start on
$U_0 = \big\{ \xi \in \Omega \cap B_0\, ; \, \delta(\xi) \geq \delta_0 R_0 \big\}$, with $\delta_0 = 1/M$. 
Then by Harnack's inequality (applied to $u_R$ and to $u_D$),
\begin{equation} \label{5e14}
\lambda_0 : = \sup_{U_0} w(\xi) \leq C w(X_0) \leq C,
\end{equation}
by the proof of of \eqref{5a11}, with constants that are allowed to depend on $\delta_0 = 1/M$.

We want to get closer to $\d\Omega$, and for this we shall use Lemma \ref{l:iterateforguy} 
repeatedly, applied to the functions $u_R = G^a_R(\cdot,Y)$ and $u_D = G_D(\cdot,Y)$, 
in smaller and smaller balls $B_k = B(Q_X, R_k)$, and with $\delta_0$ replaced
by varying depth parameters $\delta_k$, which will help some sums in the estimates to converge geometrically.
So we start with a definition of numbers $R_k$, $\delta_k$, and
$\rho_k = R_k \delta_k$. 
For the largest ball $B_0$, we already chose $\delta_0 = 1/M$; we also want to do this for the smallest ball
of our collection, but for the intermediate balls we want to take smaller values $\delta_k$.
We decide to keep $\rho_k := \delta_k R_k  = 2^{-k} \delta_0 R_0$, and also to stop
for the first index, which we shall call $\ell$, such that 
\begin{equation} \label{5e15}
\rho_\ell = \delta_\ell R_\ell \leq C_1\delta(X), 
\end{equation}
where the large constant $C_1$ will be chosen soon, depending on the usual constants (and $M$). 
As for $\delta_k$, we pick $\varepsilon > 0$ very small, take 
\begin{equation} \label{5e16}
\delta_k = \delta_0  \, \max\big(2^{-k \varepsilon}, 2^{- (\ell - k) \varepsilon}\big)
\end{equation}
for $0 \leq k \leq \ell$. Then take $R_k = \delta_k^{-1} \rho_k$; this is still a quite regular
sequence, since 
$$
2^{(1-\varepsilon) } \leq \frac{R_{k-1}}{R_{k}} \leq 2^{(1+\varepsilon) }
\ \ \text{for $k \geq 1$. }
$$

We want to apply Lemma \ref{l:iterateforguy} to $B_k$, with the given choice of $\delta_k$.
So need to check that \eqref{5a8} holds, i.e., that
\begin{equation} \label{5e17}
I_z(\delta_k R_k) \geq M  \ \text{ for } z \in \d\Omega \cap \frac{9}{10} B_k.
\end{equation}
We first use the doubling property of $\sigma$ to show that
$$
\sigma(B(Q_X, R_k)) \leq \sigma(B(z, 2 R_k)) \leq C \delta_k^{-N} \sigma(B(z, 2 \delta_k R_k))
$$
where $N$ depends only on the doubling constant in \eqref{1n4}. Hence
\begin{eqnarray} \label{5c19}
I_z(\delta_k R_k) &=& a(\delta_k R_k)^{2-n} \sigma(B(z, 2 \delta_k R_k))
\geq C^{-1} a\delta_k^{N} (\delta_k R_k)^{2-n} \sigma(B(Q_X, R_k))
\nn\\
&=&  C^{-1} \delta_k^{N+2-n} I_{Q_X}(R_k)
\geq C^{-1} \delta_0^{N+2-n}2^{-(N+2-n)\varepsilon (\ell-k)} I_{Q_X}(R_k).
\end{eqnarray}
In addition, by \eqref{3b12} and the definition of $\ell$,
\begin{eqnarray} \label{5c20}
I_{Q_X}(R_k) &\geq& C^{-1} (R_k/\delta(X))^{d+2-n} I_{Q_X}(\delta(X)) 
\geq C^{-1} \delta_0^{-(d+2-n)}(R_k/R_\ell)^{d+2-n} C_1^{d+2-n} I_{Q_X}(\delta(X)) 
\nn\\
&\geq&  C^{-1} \delta_0^{-(d+2-n)}C_1^{d+2-n} 2^{(1-\varepsilon) (d+2-n) (\ell-k)} I_{Q_X}(\delta(X)) 
\end{eqnarray}
Therefore, if $\varepsilon$ is so small that 
$(1-\varepsilon) (d+2-n)- (N+2-n)\varepsilon > (d+2-n)/2 =: \gamma > 0$,
\begin{equation} \label{5c21}
I_z(\delta_k R_k) \geq C^{-1} \delta_0^{N-d}C_1^{d+2-n} 2^{\gamma(\ell-k)} I_{Q_X}(\delta(X)) 
\geq C^{-1} C_0^{-1} C_1^{d+2-n}\delta_0^{N-d} 2^{\gamma (\ell-k)}, 
\end{equation}
where we have used the assumption \eqref{5a5}.
 Choose $C_1$ large enough (depending on $C_0$ as well as $M$ and the usual constants), so that
\begin{equation} \label{5f21}
I_z(\delta_k R_k) \geq 2^{\gamma (\ell-k)} M.
\end{equation}

So \eqref{5e17} holds and we can apply the lemma. 
Set $U_k = \big\{ \xi \in \Omega \cap B_k\, ; \, \delta(\xi) \geq \rho_k \big\}$
and $\lambda_k = \sup_{U_k} w(\xi)$; since we chose $\rho_{k+1} = \rho_k/2$, Lemma \ref{l:iterateforguy} says that
\begin{equation} \label{5e21}
\lambda_{k+1} \leq (1+ \eta_k) \lambda_k, \ \text{ with } 
\eta_k =  C \min\Big( M^{-\beta}, e^{-\frac{\beta}{\delta_k}}
+ \sup_{z \in \d\Omega \cap \frac{9}{10} B_k} \big[ I_z(\rho_k)^{-\beta}\big] \Big)
\end{equation}
(recall that $\delta_k R_k = \rho_k$). Thus 
$\lambda_{\ell} \leq \lambda_0 \prod_{k < \ell} (1+ \eta_k) \leq C \prod_{k < \ell} (1+ \eta_k)$
by \eqref{5e14}, and since $\eta_k \leq C M^{-\beta} < 1/2$, we can take logarithms and get that
\begin{equation} \label{5e22}
\log(\lambda_\ell) \leq C + C \sum_{k < \ell} e^{-\frac{\beta}{\delta_k}} +
C \sum_{k < \ell} \sup_{z \in \d\Omega \cap \frac{9}{10} B_k} \big[ I_z(\rho_k)^{-\beta} \big].
\end{equation}
The first sum is less than $C$, by our choice \eqref{5e16} of $\delta_k$, and for the second term we observe that
for $z \in \d\Omega \cap \frac{9}{10} B_k$, \eqref{5f21} says that 
$I_z(\rho_k)^{-\beta} \leq C \big[2^{\gamma (\ell-k)} M\big]^{-\beta}$. 
Again the sum converges, and at the end we get that $\lambda_\ell \leq C$.

Call $\xi_\ell$ any point of $U_{\ell}$. Then $w(\xi_\ell) \leq \lambda_\ell \leq C$,
and a last application of Harnack also yields $w(X) \leq C$. That is, we showed that
\begin{equation} \label{5e24}
1 \leq w(X) = \frac{u_R(X)}{u_D(X)} = \frac{G^a_R(X,Y)}{G_D(X,Y)} \leq C,
\end{equation}
which is the desired result \eqref{5a6}, in our second case.

We are left with the case where both $\delta(X)$ and $\delta(Y)$ are less than $c_0 R_0/2$, 
and as the reader may have guessed we need to move both $X$ and $Y$. 
Thus we also choose $Q_Y \in \d\Omega$ such that $|Q_{X}-X| = \delta(X)$ and
define $B_0^Y = B(Q_Y, R_0)$ as we did for $X$. We also 
construct a decreasing sequence of balls $B_k^Y = B(Q_Y, R_k)$,
$0 \leq k \leq \ell_Y$, to which we can apply Lemma~\ref{l:iterateforguy} as above. 

Also pick a corkscrew point $X_0$ for $B_0$ and a corkscrew point $Y_0$ for $B_0^Y$.

Let us first observe that argument above still works with $Y$ replaced by $Y_0$, because $Y_0$
is far from $\d\Omega$ as $Y$ was earlier. This way we get that 
$G^a_R(X_0, Y_0) \simeq G_D(X_0, Y_0)$, as in \eqref{5a11}.
 This way we get that
$\frac{G^a_R(X,Y_0)}{G_D(X,Y_0)} \leq C$. And now we apply the iterating part of the proof,
this time applied with $u_R = G^a_R(X,\cdot)$ and $u_D = G_D(X,\cdot)$
(and the adjoint operator, which is still of the same type, see \eqref{2b8}).
These two are still elliptic, and in this part of the argument we never used the fact that
$Y$ (and now $X$) is far from $\d\Omega$. We get that 
$$
\frac{G^a_R(X,Y)}{G_D(X,Y)} = \frac{u_R(Y)}{u_D(Y)} \leq C \, \frac{u_R(Y_0)}{u_D(Y_0)}
= C \, \frac{G^a_R(X,Y_0)}{G^a_R(X, Y_0)} \leq C.
$$
This is \eqref{5a6} (because $G^a_R \leq G_D$), and 
Theorem~\ref{t:GuysConj} follows from Lemma \ref{l:iterateforguy}.
\end{proof}

\ms 
To finish the proof of Theorem \ref{t:GuysConj} and in fact the proof of the entirety of Theorem \ref{thm:GFest}, we need only prove Lemma \ref{l:iterateforguy}. 

\begin{proof}[Proof of Lemma \ref{l:iterateforguy}]
Let $\tau, C$ be geometric constants 
to be fixed later and introduce the domain 
$$ 
Q_R = B(Q, 9R/10)  
\cap \{\tau\delta_0 R \leq \mathrm{dist}(x, \partial \Omega) \leq \delta_0 R\}. 
$$ 
Fix $\xi$ as in the statement of the lemma. Since \eqref{5a9} already gives us the desired inequality \eqref{5a10}
when $\dist(x, \d\Omega) \geq \delta_0 R$, we shall assume that 
\begin{equation} \label{5a17}
\xi \in Q_R \cap B(Q, 4R/5) \ \text{ and } \ \delta(\xi) \geq \delta_0 R/2. 
\end{equation}
Our strategy is to estimate $u_R(\xi)$ using the (Dirichlet) harmonic measure with pole $\xi$ for $Q_R$, 
which we will hereafter refer to as $\omega^{\xi}$. It will be natural to break $Q_R$ into the sections
$$
A_k:= \{y\in Q_R\mid k\delta_0 R \leq |y-\xi| < (k+1)\delta_0R \}, \ \ k \geq 0.
$$ 
For each $A_k$ we write a disjoint union 
$$
\d_k := A_k \cap \partial Q_R =: \partial^+_k \cup \partial^-_k \cup \partial^\infty_k, 
$$ 
where $x\in \partial^+_k$ means that $\mathrm{dist}(x, \partial \Omega) = \delta_0 R$, 
$x\in \partial^-_k$ means that $\mathrm{dist}(x, \partial \Omega) = \tau \delta_0 R$, 
and the remaining part $\partial^\infty_k$ is contained in $\partial B(Q, 9R/10)$. 
Observe that most $k$ we will have $\partial^\infty_k = \emptyset$.

Our estimates on $u_R$ in $A_k$ will get worse with $k$ but this is all right 
because $\omega^\xi(\partial^\pm_k)$ decays extremely quickly in $k$: we claim that 
\begin{equation}\label{e:hmdecay}
\omega^\xi(\partial_k) \leq Ce^{-ck},
\end{equation}
for some constants $C,c > 0$ that depend only on the dimension $n$. 
The proof of \eqref{e:hmdecay} follows from the Bourgain estimate on harmonic measure and the maximum principle. 
Here are the details: fix $k \geq 1$ and first 
let $u(z)\vcentcolon= \omega^z(\partial_k)$ and 
$v_0(z)=(1-\omega_{A_0}^z(\partial_0^+\cup\partial_0^-))\sup_{x\in\partial_0^L}\omega^x(\partial_k)$
for $z \in A_0$, 
where $\partial_0^L=\{y\in A_0 : |y-\xi|=\delta_0 R\}$. 
Since $u$ and $v_0$ are elliptic functions in $A_0$, the maximum principle implies that $u(z)\leq v_0(z)$ for all $z\in A_0$. Moreover, by the Bourgain estimate we have that $\omega_{A_0}^\xi(\partial_0^+\cup\partial_0^-) > \theta$ for some $\theta \in (0,1)$, and therefore $\omega^\xi(\partial_k) \leq (1-\theta)\sup_{x\in \partial_0^L} \omega^x(\partial_k)$. 
Arguing in a similar way, if $k \geq 2$, 
let $v_1(z)\vcentcolon=(1-\omega_{A_0\cup A_1}^z(\partial_1^+\cup\partial_1^-))\sup_{x\in\partial_1^L}\omega^x(\partial_k)$ and apply the maximum principle to get $u(z)\leq v_1(z)$ for $z\in A_1\cup A_0$. Using again the Bourgain estimate we 
get that $\omega^\eta(\partial_k) \leq (1-\theta)\sup_{x\in \partial_1^L} \omega^x(\partial_k)$ for all $\eta \in \partial_0^L$. From this it follows that $\omega^\xi(\partial_k) \leq (1-\theta)^2\sup_{x\in \partial_1^L} \omega^x(\partial_k)$. Iterating this we get $\omega^\xi(\partial_k) \leq (1-\theta)^j \sup_{|\xi - z| = j\delta_0 R} \omega^z(\partial_k)$ 
for any $j \leq k-1$ and thus, if $-c = \log((1-\theta))$, we have \eqref{e:hmdecay}. It is important to note here that the estimate in \eqref{e:hmdecay} is completely independent of $\tau \leq 1/2$, say.  

We now break things down into pieces:
\medskip

\noindent {\bf Part I: The contribution of $\partial^\infty_k$.}  
We want to estimate $u_R(y)$ for $y\in \partial^\infty_k$. 
First let $Q_y$ denote the the point in $\partial \Omega$ closest to $y$ and choose a corkscrew point
$A_y$ for $B(Q_y, \delta_0 R)$. Then $u_R(y) \leq Cu_R(A_y)$ by \cite[Lemma 3.3]{Robin1}
(a nontrivial lemma that says that $u_R$ tends to be largest at corkscrew points). Then we want to connect $A_y$ to a point $\xi_0$
near $\xi$ with a Harnack chain in $\Omega$; we choose $\xi_0 \in B(\xi, C\delta_0 R)$ so that 
$\delta(\xi_0) \geq \delta_0 R$, because this way \eqref{5e14} says that $u_R(\xi_0) \leq \lambda u_D(\xi_0)$. 
But since $\delta(\xi) \geq \delta_0 R/2$ by \eqref{5a17}, a few applications of Harnack's inequality
also yield $u_D(\xi_0) \leq C u_D(\xi)$.

Now we need to evaluate the number $N$ of elements in a Harnack chain that goes from $A_y$ to $\xi_0$.
Both lie in a same ball of size $R$ and at distance at least $C^{-1} \delta_0 R$ from $\d\Omega$,
and a fairly simple argument, using successive corkscrew points at larger and larger scales, yields $N \leq C \log(\delta_0^{-1})$.
Here we have to be slightly more careful, because we were only told that $u_R$ is a solution in $B(Q,R)$, 
so we and so want to make sure that our Harnack chains stays in $\Omega \cap B(Q,R)$. 
Close to $y$, and because we took $y \in B(Q, 9R/10)$, we are safe until we get to a corkscrew point $Z_y$ at scale $R/C_0$, 
where $C_0$ will be chosen soon. 
Similarly we can connect $\xi_0$ to a corkscrew point $Z_\xi$ at the same scale $R/C_0$, 
and now we need to be slightly more careful when we connect $Z_\xi$ to $Z_y$. 
But since $y \in \d^\infty_k$, there is a path in 
$Q_R$ that goes from $\xi$ to very near $y$, and along this path we can find less than $CC_0$ points, at successive distances
smaller than $R/C_0$. We pick corkscrew points at scale $R/C_0$ near these points, 
and connect them with Harnack chains in $\Omega$. If $C_0$ is chosen large enough, depending on the NTA constants for $\Omega$,
all these Harnack chains are contained in $\Omega \cap B(Q,R)$, so we are safe. And we only added $C C_0$ balls to the
chain, so we still have a total number $N \leq C \log(\delta_0^{-1})$. So multiple
applications of the Harnack inequality (or $u_R \geq 0$) yield
$u_R(A_y) \leq C^N u_R(\xi_0) \leq C \delta_0^{-K} \lambda u_D(\xi_0)$, where $K$ is also a geometric constant,
and altogether $u_R(y) \leq C \delta_0^{-K} \lambda u_D(\xi_0) \leq C \delta_0^{-K} \lambda u_D(\xi)$.


Let $k_\infty$ be the smallest $k$ such that $\partial^\infty_k\neq \emptyset$ and observe that 
$k_\infty \delta_0 R> R/10$. 
So we can estimate 
\begin{equation} \label{5a19}
\sum_k \int_{\partial^\infty_k} u_R(z)d\omega^\xi(z) 
\leq C \delta_0^{-K} \lambda u_D(\xi) \sum_{k > \delta_0^{-1}/10} \omega^\xi(\partial_k)
\leq C \delta_0^{-K} e^{-\frac{c}{10\delta_0}}   
\lambda u_D(\xi) \leq  
u_D(\xi) e^{-\frac{c}{20\delta_0}}
\end{equation}
by \eqref{e:hmdecay} and, for the last inequality, the facts that $\delta_0 < 1/M$ is as small as we want 
and $\lim_{t \to +\infty} t^{K} e^{-\frac{c t}{20}} = 0$.

\medskip
\noindent {\bf Part II: The contribution of $\partial^-_k$.} Pick any $y\in \partial_k^-$: we want to show that $u_R(y)$
is small, because of the relatively large value of $a$. Let $Q_{y}\in \partial \Omega$ be such that 
$|y - Q_{y}| \leq 2 \dist(y, \partial \Omega) = 2\tau \delta_0 R$.
Set $B_y = B(Q_y, \tau \delta_0 R)$, and more generally 
$CB_y = B(Q_y, C\tau \delta_0 R)$ for $C \geq 1$.

Notice that $u_R^2(y) \leq C \fint_{4B_y \cap \Omega} u^2_R$
by the Harnack inequality. Then our Poincar{\'e} inequality at the boundary, 
Lemma \ref{lem:poincareboundary}, implies that
\begin{equation}\label{e:poincareinminus}
u_R^2(y) \leq C \fint_{4B_y \cap \Omega} u^2_R
\leq C  \fint_{4B_y \cap \d\Omega} u^2_R d\sigma 
+ C (\tau \delta_0 R)^{2}  \fint_{CB_y \cap \Omega} |\nabla u_R|^2
\end{equation}
For the first integral, we use a variant of the Caccioppoli estimate: we apply the defining condition 
\eqref{e:weaksol} (with $f=0$) with $\varphi = \theta^2 u_R$, for some 
smooth cutoff supported in $8B_y$ that is equal to $1$ on 
$4B_y$. This yields 
\begin{eqnarray}\label{5a21}
\int_{\d \Omega} u_R^2 \theta^2 d\sigma 
&=& - \frac{1}{a} \int_\Omega A \nabla u_R \nabla (\theta^2 u_R)
= - \frac{1}{a} \int_\Omega \theta^2 A \nabla u_R \nabla u_R
-  \frac{1}{a} \int_{\Omega}  2 u_R \theta A \nabla u_R \nabla \theta
\nn\\
&\leq&
\frac{C}{a} \int_{\Omega} u_R \theta |A| |\nabla u_R| |\nabla \theta|
\leq \frac{C}{a} \int_{8B_y \cap\Omega} (\tau \delta_0 R)^{-1} u_R |\nabla u_R|
\nn\\
&\leq& \frac{C}{a} \int_{8B_y \cap\Omega}  \Big\{  (\tau \delta_0 R)^{-2} u_R^2  
 + |\nabla u_R|^2 \Big\}
\leq  \frac{C  (\tau \delta_0 R)^{-2}}{a} \int_{8B_y \cap\Omega} \Big\{ u_R^2 + (\tau \delta_0 R)^2 |\nabla u_R|^2 \Big\}.
\end{eqnarray}
Then we use the the Caccioppoli estimate itself to say that
$\int_{8B_y \cap\Omega} (\tau \delta_0 R)^2 |\nabla u_R|^2
\leq C \int_{16B_y \cap\Omega} u_R^2$, so that finally
\begin{equation}\label{5a22}
\fint_{4B_y \cap \d\Omega} u^2_R d\sigma
\leq  \frac{C  (\tau \delta_0 R)^{-2}}{a \sigma(4B_y)} \int_{16B_y \cap\Omega} u_R^2
\leq  \frac{C (\tau \delta_0 R)^{n-2}}{a \sigma(4B_y)}  \sup_{16B_y} u_R^2.
\end{equation}
We use \cite[Lemma 3.3]{Robin1} again to say that 
$\sup_{16B_y} u_R \leq C u_R(A_y) \leq C \lambda u_D(A_y)$,
where $A_y$ is a corkscrew point for $B(Q_y, C \delta_0 R)$, where $C$ is chosen large enough so that
$\delta(A_y) \geq \delta_0 R$ and hence $u_R(A_y) \leq \lambda u_D(A_y)$ by assumption.
Thus 
\begin{equation}\label{5a23}
\fint_{4B_y \cap \d\Omega} u^2_R d\sigma
\leq \frac{C (\tau \delta_0 R)^{n-2} \lambda^2 u_D^2(A_y)}{a \sigma(4B_y)}
\leq \frac{C (\tau \delta_0 R)^{n-2} \lambda^2 }{a \sigma(4B_y)} \ \sup_{B(Q_y, C \delta_0 R)} u_D^2.
\end{equation}
We leave this for the moment and estimate the last integral in \eqref{e:poincareinminus}.
Again we use the Caccioppoli estimate, but this time removing the infimum as in Remark \ref{rk31}
to let the oscillation (rather than the supremum) appear; this is good because the oscillation decay
for $u_R$, from the scale $\delta_0 R$ to the smaller scale $\tau \delta_0 R$, will give a better estimate.
That is, 
\begin{eqnarray}\label{5a24}
\fint_{CB_{y}}|\nabla u_R|^2 
&\leq& C(\tau \delta_0 R)^{-2}  \left(\mathrm{osc}_{2CB_{y}} u_R\right)^2 
\leq C (\tau \delta_0 R)^{-2} \tau^{2\alpha}  \left(\mathrm{osc}_{B(Q_{y}, C \delta_0 R)} u_R\right)^2 
\nn\\
&\leq& C\lambda^2(\tau \delta_0 R)^{-2}\tau^{2\alpha} \sup_{z\in B(y, C\delta_0 R)} u^2_D(z),
\end{eqnarray}
where the decay rate $\alpha > 0$ depends on geometric and ellipticity constants and 
comes from iterations of \cite[Theorem 4.5]{Robin1} (it is important that the constants there do not depend on $a$), and we used \cite[Lemma 3.3]{Robin1} and our assumption \eqref{5a8}again.
Putting all of this together, we showed that for $y$ in any $\d_k^-$, 
 \begin{equation}\label{e:firstminusest}
 u_R^2(y) \leq C\lambda^2 \sup_{z\in B(y, C\delta_0)} u_D^2(z)
\  \left(\frac{(\tau\delta_0 R)^{n-2}}{a\sigma(B(Q_{y},\tau\delta_0 R))} +  \tau^{2\alpha} \right).
 \end{equation}
 We first estimate the first term 
 \begin{equation}\label{5a26}
 X =  \frac{(\tau\delta_0 R)^{n-2}}{a\sigma(B(Q_{y},\tau\delta_0 R))} = \frac{1}{I_{Q_y}(\tau\delta_0 R)},
\end{equation}
where we are happy to recognize the dimensionless number $I_{Q_y}(\tau\delta_0 R)$,
in terms of the quantity
\begin{equation}\label{5c28}
J = \sup_{z \in \d\Omega \cap B(Q, 9R/10)} I_z(\delta_0 R)^{-1}
\end{equation}
which appears implicitly in the statement.
We use the doubling property the usual iteration of \eqref{1n4} to get that for 
$\rho = \tau\delta_0 R < \delta_0 R$, 
$\sigma(B(Q_{y},\delta_0 R)) \leq C \tau^{-K}\sigma(B(Q_{y},\tau\delta_0 R))$, 
where $C \geq 1$ and $K \geq 1$ depend only on the doubling constant in \eqref{1n4}.
Then
$$
X = \frac{(\tau\delta_0 R)^{n-2}}{a\sigma(B(Q_{y},\tau\delta_0 R))} 
\leq C \tau^{-K}  \frac{(\tau\delta_0 R)^{n-2}}{a\sigma(B(Q_{y},\delta_0 R))} 
= C \tau^{-K +n - 2} I_{Q_y}(\delta_0 R)^{-1} \leq C \tau^{-K +n - 2} J
$$
because $Q_y \in B(Q, 9R/10)$ when $\xi \in B(Q, 4R/5)$. So \eqref{e:firstminusest}
says that 
\begin{equation}\label{5c29}
 u_R(y) \leq C  \lambda  \Big(\tau^{-K +n - 2} J + \tau^{2\alpha}\Big)^{1/2}
 \sup_{z\in B(y, C\delta_0)} u_D(z) 
\end{equation}
and now we choose $\tau$. We do not know exactly the value of $K$, but we may assume that
$K-n+2 \geq 1$, say, because otherwise we can use the worse estimate with $K-n+2 = 1$.
The simplest is to choose $\tau$ so that 
$\tau^{-K +n - 2} J =  \tau^{2\alpha}$, i.e., $\tau^{K-n+2+\alpha} = J$. 
Notice that $J \leq M^{-1}$ by \eqref{5a8}, so this gives $\tau$ small, as needed.
With this value of $\tau$, we get that
\begin{equation}\label{5c30}
 \Big(\tau^{-K +n - 2} J + \tau^{2\alpha}\Big)^{1/2} \leq C \tau^\alpha \leq C J^\beta, 
\ \text{ with } \beta = {\frac{\alpha}{K-n+2+\alpha}}
\end{equation}
Next we estimate the supremum in \eqref{5c29}. We claim that for $y \in \d_k^-$ as above
\begin{equation}\label{5a28}
\sup_{z\in B(y, C\delta_0)} u_D(z) \leq C (1+k)^L u_D(\xi)
\end{equation}
for some $L > 0$ that depends only on geometric and ellipticity constants.
 The argument will be the same as in the paragraph above \eqref{5a19}. First,
 $\sup_{z\in B(y, C\delta_0)} u_D^2(z) \leq C u_D(\xi_y)$, where $\xi_y$ is a corkscrew point for
 $B(y, C \delta_0 y)$, because of \cite[Lemma 3.3]{Robin1} (which incidentally is easier here, because we deal
 with a Dirichlet solution). Then we need to connect $\xi_y$ to $\xi$ with a Harnack chain, and estimate the length of that chain. We keep the same argument as above, and find a chain of length $C \log(1+k)$; again a simple first estimate would use intermediate corkscrew points $A_\ell$ at scales $2^\ell \delta_0 R$, but we need to 
 worry about the largest scales because maybe the corkscrew ball gets out of $B(Q,R)$. But this happens
 only for scales larger than $C^{-1} R$, and we can fix the argument by taking $C$ extra points close to
 $\d\Omega$. So \eqref{5a28} holds. 
 
 We may finally estimate the contribution of the $\partial_k^-$:
 \begin{eqnarray} \label{5a29}
\sum_k \int_{\partial^-_k} u_R(y)d\omega^\xi(y) 
&\leq& \sum_k  \omega^\xi(\partial^-_k) \sup_{y \in \partial^-_k} u_R(y)
\leq C \sum_k  e^{-ck}  \sup_{y \in \partial^-_k} u_R(y)
\nn \\
&\leq& C \sum_k  e^{-ck} \lambda  J^\beta (1+k)^L u_D(\xi)
\leq C \lambda  J^\beta u_D(\xi)
\end{eqnarray}
by \eqref{e:hmdecay}, \eqref{5c29}, \eqref{5c30}, and \eqref{5a28}.

\medskip
\noindent {\bf Part III: Putting it all together.} 
We are now ready for our final estimate. 
Our last piece is easy, because 
\begin{equation}\label{5a30}
\sum_k \int_{\partial^+_k} u_R(y)d\omega^\xi(y) 
\leq \int_{\d^+} \lambda u_D(u) d\omega^\xi(y) \leq  
\int_{\partial Q_R} u_D(z)\, d\omega^\xi(z)
= \lambda u_D(\xi)
\end{equation}
by definition of $\d^+$ and because $u_R \leq \lambda u_D$ whenever $\mathrm{dist}(x, \partial \Omega) \geq \delta$. We sum this with \eqref{5a19} and \eqref{5a29} and get that
\begin{equation}\label{5a31}
u_R(\xi) = \sum_k \int_{\d_k} u_R(y) d\omega^\xi(y)  \leq \lambda u_D(\xi)
\, \Big\{ 1 + C 
e^{-\frac{c}{20\delta_0}} +  C \lambda  J^\beta  \Big\}.
\end{equation}
Let us now check that this is what we announced. Observe that the supremum in \eqref{5a10b}
is precisely $J^\beta$ (see \eqref{5c28}). We assumed that $\delta_0 < 1/M$ and  $J \leq 1/M$,
so \eqref{5a10} holds with $\eta = M^{-\beta}$.
But we also get the extra decay announced in \eqref{5a10b}, because 
$e^{-\frac{c}{20\delta_0}} \leq  e^{-\frac{\beta}{\delta_0}}$ if $\beta$ was chosen small enough, 
and $\lambda  J^\beta$ is the supremum in \eqref{5a10b}.  This concludes our proof of
Lemma \ref{l:iterateforguy}, Theorem~\ref{t:GuysConj}, and 
Theorem \ref{thm:GFest}.
\end{proof}

\section{Properties of Harmonic Measure}\label{sec:hmproperties}
In this section we establish some basic properties of the 
Robin harmonic measure we introduced in \cite{Robin1} and of functions with homogeneous Robin boundary conditions. We shall denote by $\omega^X_R$ the Robin harmonic measure on $\d\Omega$, 
with the pole $X \in \Omega$; 
for the construction and basic properties, 
see Theorems 5.4 and 5.5 in \cite{Robin1}. Here we want to complete \cite{Robin1} with 
other useful properties 
in the flavor of \cite{kenigbook}, in particular Chapter 1.3 (see also \cite{JKNTA}). In particular, the structure and some of the arguments in this section follow \cite{kenigbook}. However, there are several results which require substantially different proofs in the Robin setting, in particular Lemma \ref{l:Bourgain} and Theorem~\ref{t:boundarycomp}. 
In this section too, we make sure that our constants do not depend on $a$. %
		
		We begin with what is usually called the ``Bourgain estimate", a lower bound on the harmonic measure of a ball taken from a corkscrew point. Recall that we systematically denote by $A_r(Q)$ a corkscrew point
for the ball $B(Q, r)$ centered on $\d\Omega$. 
		
		\begin{lemma}\label{l:Bourgain}[Bourgain Estimate] 
There exists  a constant $M > 1$  
(depending on the geometric constants of $\Omega$ and the ellipticity of $A$) such that 
for all $Q\in \partial \Omega$ and $r < \mathrm{diam}(\Omega)/10$, 
\begin{equation}\label{e:bourgain}
	\omega_R^{A_r(Q)}(B(Q,r)\cap \partial \Omega) \geq M^{-1} \min\{1, ar^{2-n}\sigma(B(Q,r))\}.
\end{equation}
		\end{lemma}

Notice that in the Neumann case and for small balls, we expect the harmonic measure to be spread 
homogeneously on $\d\Omega$, hence the measure of $B(Q,r)$ to be very small.

		\begin{proof} 
			Let $u(X)=\omega_R^X(B(Q,r)\cap \partial \Omega)$ for $X \in \Omega$. 
			We divide the proof in two cases.\\
			
			\noindent \textbf{Case 1.} 
Set $\Delta_1 = \d\Omega \cap B(Q,r/4)$ and first assume that $\sup_{\Delta_1}u > \frac{3}{4}$.
This is what we would expect close to the Dirichlet regime.
	Let $\widetilde{u}=1-u$. Then $0\leq \widetilde{u}\leq 1$, and $\inf_{\Delta_1}\widetilde{u}< 1/4$. 
	Note that since $\widetilde{u} = \omega^X_{R}(\d\Omega \sm B(Q,r))$ for $X \in \Omega$, 
	$\widetilde{u}$ (weakly) satisfies homogeneous Robin conditions on $B(Q,r)\cap \partial \Omega$. 
	Let $P\in \Delta_1$ be such that $\tilde{u}(P)\leq 1/4$. 
	We are setting up to apply the oscillation estimate in Theorem 4.5 of \cite{Robin1}, on sufficiently small
balls centered at the origin $P$, and zero right-hand side $f$. 
Let $\eta\in (0,1)$ be the constant from that theorem and let $l \geq 1$ be large enough that $\eta^l \leq \frac{1}{4}$. Note that $\eta$ (and thus $l$) depends only on the geometric constants of $\Omega$ an the ellipticity of $A$ (but not on $a$). Let $\bar{X}$ be a corkscrew point for $P$ at scale $r/(4K^l)$ (where $K$ is the geometric constant
in the theorem). Applying $l$ times 
Theorem 4.5 of \cite{Robin1} to $\widetilde{u}$, we get that 
 $$
 \widetilde{u}(\bar{X})- \widetilde{u}(P)  
 < \eta^l  \, \mathrm{osc}_{B(P, r/4)}\, \widetilde{u} < 1/4 .
 $$ 
    Thus $u(\bar{X}) \geq 1/2$. After connecting $\bar{X}$ to $A_r(Q)$ by Harnack chains, we get 
    $u(A_r(Q)) \geq c$, as needed. \\
 
 \noindent	\textbf{Case 2.} 
We may now assume that $\sup_{\Delta_1}u \leq\frac{3}{4}$. 
We may think of this case as the ``Neumann regime" (see the comment below). 
We need a localization argument, so let  $T = T(Q, r/4K)$ be the tent domain of \ref{l:tentspaces},
with $K > 1$ as in the lemma; this way 
\begin{equation}\label{6a2}
\Omega \cap B(Q, r/4K) \subset T \subset B(Q,r/4) \cap \Omega.
\end{equation}
Also let $S = S(Q, r/(4K))=\Omega \cap \partial T$ be the ``exterior boundary'' of $T$,
and recall that $T$ comes with a measure $\sigma_\star$ on $\d T$, which coincides with $\sigma$
on $\d T \cap \d\Omega$.
By Theorem 2.10 in \cite{Robin1} we can solve the following partial Neumann problem on $T$
(morally, $\frac{\d v}{\d n} = \frac{1}{8}$): find $v \in W^{1,2}(T)$ such that 
$v = 0$ on $S$ (that is, its trace vanishes on $S$) and 
\begin{equation}\label{6a3}
\int_{T} A\nabla v \nabla \varphi =\int_{ \d T \cap \d \Omega} \frac{\varphi}{8} \, d\sigma
\end{equation} 
for all $\varphi$ with trace zero on $S$. 

Define $h:=u- av$.  Then $h\geq 0$ on $S$. 
Also recall from Theorem 5.5 in \cite{Robin1} that since $u$ is the Robin harmonic measure of $\d\Omega \cap B(Q,r)$,
\begin{equation}\label{6a4}
\frac{1}{a} \int_{\d\Omega} A\nabla u \nabla \varphi + \int_{\d\Omega} u \varphi d\sigma
= \int_{\d\Omega \cap B(Q,r)} \varphi d\sigma
\end{equation}
for all $\varphi \in C^{\infty}_c(\R^n)$, so for every  $0 \leq \varphi \in W^{1,2}(T)$ 
with vanishing trace on $S$, there holds 
\begin{equation}\label{6a5}
\int_{T} A\nabla h \nabla \varphi 
	= a\int_{\d T \cap \d \Omega}\varphi -\varphi u -\frac{\varphi}{8}  \geq 0,
\end{equation}
because $u \leq\frac{3}{4}$ on $\Delta_1 =  \d\Omega \cap B(Q,r/4)$ and by \eqref{6a2}.
Plugging in $h^{-}$ in the inequality above 	it follows that $h(x)\geq 0$ for $x\in T$, so that $u\geq v$ on $T$. 

\begin{lemma} \label{neumanndensity2}  
Let $v$ solve \eqref{6a3} as above. 
There is exist a constant $c_0> 0$, that depends only on on the ellipticity of $A$ 
and the geometric constants of $\Omega$, such that 
\begin{equation}\label{6a6} 
v(\xi_0)   \geq c_0 r^{2-n} \sigma(B(Q, r))  \ \text{ for some $\xi_0 \in B(Q, r/2)$ such that }
\delta(\xi_0) \geq c_0 r.
\end{equation} 
\end{lemma}

Let us first check that Lemma \ref{l:Bourgain}  will follow at soon as we prove this.
Since $\xi_0$ is relatively far from $\d\Omega$, we can connect it to $A_r(Q)$ with a 
Harnack chain, and so $\omega_R^{A_r(Q)}(B(Q,r)\cap \partial \Omega) =: u(A_r(Q))
\geq C^{-1} u(\xi_0) \geq C^{-1} v(\xi_0) \geq C^{-1} r^{2-n} \sigma(B(Q, r))$, as needed.

So we are left with Lemma \ref{neumanndensity2} to prove. 
Note that this should be easier to do, as we no longer have Robin boundaries to take care of. In fact, the lemma is not much more than the fact that $v > 0$ on $\Omega \cap B(Q,r)$, plus a little bit of uniformity and scale invariance.

\ms
\begin{proof} [Proof of Lemma \ref{neumanndensity2}].   
We first test \eqref{6a3} against a standard bump function $\varphi$ supported on $B(Q, r/(4K))$
and equal to $1$ on $B(Q, r/(8K))$; we get that
\begin{equation}\label{6c7}
C^{-1}\sigma(B(Q,r)) \leq C^{-2}\sigma(B(Q, r/(8K))\leq  \int_{\d T \cap \d \Omega} \frac{\varphi}{8} d\sigma 
=  \int_{T} A\nabla v \nabla \varphi 
\leq Cr^{-1} \int_{T} |\nabla v|.
\end{equation}
Now Theorem 2.10 in \cite{Robin1} gives a little more than the existence of $v$, it also says that
$||v||_{W^{1,2}(T)} \leq C$, where unfortunately we did not write the explicit dependence of $C$
on $r$ and $\sigma(B(x,r))$. So let us not use this directly, but say that $C$ is an absolute constant in the 
case when $r=1$ and $\sigma(B(Q,r)) = 1$. So we finish the proof of the lemma in this case first, 
and then we will deduce the general case by scaling.

We set $B = B(Q, 1/4)$ so that \eqref{6c7} gives $\int_B |\nabla v| \geq C^{-1}$ while \cite[Theorem 2.10]{Robin1} gives
$\int_{B} |\nabla v|^2 \leq C$. We first want to check that $|\nabla v|$ is also reasonably large somewhere
far from $\d\Omega$. Let $\tau > 0$ be small, to be chosen soon, and
set $Z_\tau = \big\{ X \in B(Q,1/4)  \, ; \, \delta(X) \leq 4\tau \big\}$. We need to know that 
\begin{equation}\label{6c8}
\big|Z_\tau\big| \leq C \tau^\alpha   
\end{equation}
for some constants $C$ and $\alpha > 0$ that depend only on the corkscrew constant for $\Omega$.
This is a standard porosity argument; let us provide it at the end of the proof. Then by Cauchy-Schwarz
$$
\int_{Z_\tau} |\nabla v|  \leq \big|Z_\tau\big|^{1/2} \Big\{\int_{Z_\tau} |\nabla v|^2 \Big\}^{1/2}
\leq C \big|Z_\tau\big|^{1/2} \leq C \tau^{\alpha/2} \leq  \frac12 \int_B |\nabla v|
$$
as long as $\tau$ is chosen small enough. So
$$
\int_{B \sm Z_\tau} |\nabla v| \geq (2C)^{-1}.
$$
From now on, $\tau$ is chosen and we let our constants $C$ depend on $\tau$ as well.
Cover stupidly $B \sm Z_\tau$ by $C$ balls of radius $\tau$ centered on it, and then choose 
one such ball $B(y, \tau)$, $y \in B \sm Z_\tau$, such that 
$\int_{B(y, \tau)} |\nabla v| \geq C^{-1}$. Then 
$\int_{B(y, \tau)} |\nabla v|^2 \geq C^{-1}$ as well, by Cauchy-Schwarz.
Now we use the Caccioppoli inequality (here simpler than usual because $v$ is elliptic in $B(y, 2\tau)$, 
and find that the oscillation of $v$ in $B(y, 2\tau)$ is more than $C^{-1}$. Since $v \geq 0$,
we find $\xi_0 \in B(y, 2\tau) \subset B(Q, r/2)$ such that $v(\xi_0) \geq C^{-1}$,
which completes the proof of the lemma when $r=1$ and  $\sigma(B(Q,r)) = 1$.

In the general case, we start from $v$ and define $w$ by
$w(X) = v(rX)$ for $X \in r^{-1} T$. Then write \eqref{6a3}, but call the variable
$r X$, with $X \in r^{-1}T$. This yields
\begin{equation}\label{6c9}
\int_{T} A(rX) \nabla v(rX)  \nabla \varphi(rX)  r^n dX 
=\int_{ \d T \cap \d \Omega} \frac{\varphi(rX)}{8} \, d\sigma(rX).
\end{equation}
For the integral on the right, we use the pushforward $\sigma^\sharp$ of $\sigma$
by $Y \mapsto r^{-1} Y$, the one which does not change the total masses. Thus
$\sigma^\sharp(r^{-1} (\d T \cap B(Q, r))) = \sigma ( \d T \cap B(Q, r))$.
We finish the change of variable and write \eqref{6c9} as 
\begin{equation}\label{6c10}
\int_{T} \wt A(X) r^{-1}\nabla w(X)  r^{-1}\nabla \wt\varphi(X)  r^n dX
=\int_{ \d T \cap \d \Omega} \frac{\wt\varphi(X)}{8} \, d\sigma^\sharp(X),
\end{equation}
where $\wt \varphi(X) = \varphi(rX)$ and similarly for $\wt A$. That is,
$w$ satisfies an equation similar to \eqref{6a3}, but now with the radius $1$
and the measure $r^{2-n} \sigma^{\sharp}$, which gives $B(r^{-1}Q,1)$ the mass 
$m = r^{2-n} \sigma(B(Q, r))$. So $m^{-1} w$ is as in our special case, and we can find
 $\xi_1 \in B(r^{-1}Q,1/2)$, at distance $\tau$ from the boundary, such that $m^{-1} w(\xi_1) \geq C^{-1}$. 
Then $\xi_0 = r \xi_1$ does the job, in particular because $v(r\xi_1) = m w(\xi_1) \geq C^{-1}m$. 
The general case follows.

We promised a proof of \eqref{6c8} to complete the argument. We will essentially show that $\d\Omega$
is porous because of the corkscrew condition, and hence its (Minkowski) dimension is $< n$.

Denote by $\bD$ the set of usual dyadic cubes in $\R^n$, then call $\ell(Q)$ the side length of the cube $Q \in \bD$,
let $\bD_k$ be the set of cubes $Q \in \bD$ such that $\ell(Q) = 2^{-k}$. Observe that by the corkscrew condition there is an integer $m > 0$ such that, as long as $k$ is large enough so that $2^{-k} \leq \diam(\Omega)$, for each $Q \in \bD_k$, with
 we can find a cube $R \in \bD_{k + m}$ such that $R \subset Q$ and 
$2R$ (the cube with the same center and $2$ times the side length) does not meet $\d\Omega$.
This can be for trivial reasons, if $\frac12 Q$ does not meet $\d\Omega$, or else because if we can find
$\xi \in \frac12 Q \cap \d\Omega$, a corkscrew ball for $B(\xi, 2^{-k-2})$ necessarily contains $2R$ for some
$R \in \bD_{k + m}$. 

Start from a cube $Q_0$ that meets $Z_{\tau}$, with $\ell(Q_0) = 2^{-k_0} \sim 1$, and then for each $j \geq 0$
cut $Q_0$ into the collection $\bD_j(Q_0)$ that are contained on $Q_0$ and such that $\ell(Q) = 2^{-k_0 - jm}$.
We restrict our attention to the indices $j$ such that $2^{-k_0 - jm} \geq C \tau $, i.e., restrict to cubes
such that $\ell(Q) \geq C \tau $. Thus we stop when $2^{-jm} \sim \tau$, i.e., when 
$j \sim \frac{\log(1/\tau)}{m \log(2)}$. 

Each time $R \in \bD_j(Q_0)$ is such that $2R$ does not meet $\d\Omega$, we know that 
$R$ does not meet $Z_\tau$, so $R$ cannot contribute to $|Z_\tau|$. So we eliminate $R$ from
our collection, and we also eliminate its children and descendants. But we observed that for each cube $Q$ of our collection,
at least one cube $R \subset Q$ of the next generation is eliminated, and so at each generation we eliminate 
a proportion of at least $c 2^{-mn} |Q|$ of the remaining mass.
Because of this, at the last generation $j$ we only kept a mass of at most $(1-c 2^{-mn})^j |Q_0|$. 
So $\log( \frac{|Z_\tau \cap Q_0|}{|Q_0|}) \leq j \log(1-c 2^{-mn}) \leq - c j$, for some very small $c>0$.
Since $j \geq c\log(1/\tau)$ (for $\tau$ small), we get that 
$\frac{|Z_\tau \cap Q_0|}{|Q_0|} \leq e^{- c \log(1/\tau)} \leq C \tau^\alpha$ for some small $\alpha > 0$.
Of course we can do this for each cube $Q_0$ that meets $Z_{\tau}$, and there is at most $C$ such cubes.
This proves \eqref{6c8}, and now Lemma  \ref{neumanndensity2} follows.
\end{proof}
As checked above, this also concludes the proof of Lemma \ref{l:Bourgain}.
\end{proof}

\ms
From here we aim to compare the harmonic measure and the Green function $G_R = G_R^a$. 
The first comparison follows easily from our estimates on the Green function in Corollary~\ref{cor:upperestalways}:
		
		\begin{lemma}\label{l:lowermeasure}
Let $n\geq 3$. There exist constants $C > 0, M > 0$ (depending on the geometry of $\Omega$ and the ellipticity of $A$) such that for all $Q \in \partial \Omega$ and $r < \mathrm{diam}(\Omega)/10$ with 
$r^{2-n} a \sigma(B(Q,r)) \geq C^{-1}$ and for all $X\in \Omega \setminus B(A_r(Q), r/C)$ 
	\begin{equation}\label{e:lowermeasure}
		r^{n-2}G_R(X, A_r(Q)) \leq M \omega_R^X(B(Q,2r)\cap \partial \Omega).
	\end{equation}
		\end{lemma}
		
		
		\begin{proof}
Since both the Green function and the harmonic measure are $A$-harmonic functions of $X$ 
in $W^{1,2}(\Omega\setminus (B(A_r(Q), r/C)))$ , it suffices to check the behavior on $\partial \Omega$ 
and on $\d_A = \partial B(A_r(Q), r/C)$. We start on $\d_A$, where we apply
Corollary \ref{cor:upperestalways} to get that $G_R(X, A_r(Q)) \leq C |X-A_r(Q)|^{2-n} = C r^{2-n}$.
To be fair, the corollary seems to require that $\diam(\Omega)^{2-n} a \sigma(\d\Omega) \geq 1$ while we 
only assume that $r^{2-n} a \sigma(B(Q,r)) \geq C^{-1}$
(and hence $\diam(\Omega)^{2-n} a \sigma(\d\Omega) \geq C^{-1}$), but the reader will easily check that
this only makes our constants a little larger. Then Lemma \ref{l:Bourgain} implies that the desired
inequality \eqref{e:lowermeasure} holds on $\d_A$. 

On $\partial \Omega$ we note that the Green function has homogeneous Robin data while the harmonic measure has non-negative Robin data. So by the maximum principle with mixed Robin and Dirichlet data,  Lemma \ref{lem:mixedmax}, our estimate \eqref{e:lowermeasure} holds. 
		\end{proof}
		
The following more precise result contains estimates in both directions.

\begin{theorem}\label{t:hmgfequiv}
Let $n \geq 3$. There exists a constant $C > 1$, which depends on the ellipticity of the matrix $A$ and the geometric constants of $(\Omega, \sigma)$, %
such that for all $Q\in \partial \Omega$ and $r < \mathrm{diam}(\Omega)/(10C)$,
\begin{equation}\label{e:hmgfequiv}
C^{-1} G_R(X, A_r(Q)) \leq \frac{\omega_R^{X}(B(Q, r)\cap \partial \Omega)}{r^{n-2}\min\{1, ar^{2-n}\sigma(B(Q,r))\}} \leq CG_R(X,A_r(Q)),
\end{equation}
for all $X\in \Omega \setminus B(Q, Cr)$.   
\end{theorem}
		
\begin{proof}
		
	Let $\varphi \in C_c^\infty(\mathbb R^n)$ be such that 
$ \chi_{B(Q,r)} \leq \varphi \leq \chi_{B(Q, 2r)}$.

We have two cases. When $ar^{2-n}\sigma(B(Q,r)) \gtrsim 1$ we first observe that by (5.4) in \cite{Robin1} 
\begin{equation} \label{6a13}
\omega_R^{X}(B(Q,r)\cap \partial \Omega) \leq \int_{\d\Omega} \varphi d\omega_R^{X}
= a\int_{\partial \Omega} \varphi(Y) G_R(X,Y)\, d\sigma(Y).
\end{equation}
also recall the representation formula (5.3) (and the notation of (5.2))  in \cite{Robin1}, which says that 
\begin{equation} \label{6a14}
0 = \varphi(X) = b(G_R(\cdot, X), \varphi) := 
\int_{\Omega} A \nabla G_R(\cdot, X), \nabla\varphi + a \int_{\d\Omega} \nabla G_R(Y, X) \varphi(Y) d\sigma(Y)
\end{equation}
where we used the fact that $X\notin \mathrm{spt}(\varphi)$.
The last term is the same as the last term in \eqref{6a13}, but for the adjoint matrix coefficient $\wt A$
(see \eqref{2b8}), so we apply \eqref{6a14} to the adjoint problem and get that
\begin{equation} \label{6a15}
\omega_R^{X}(B(Q,r)\cap \partial \Omega) \leq - \int_{\Omega} \wt A \nabla G_R(X, \cdot), \nabla\varphi
= -\int_\Omega A(Y) \nabla\varphi(Y) \nabla_Y G_R(X, Y) dY.
\end{equation}
	From here we can argue as in \cite{kenigbook}. Let's recall that argument quickly: 	
we use Cauchy-Schwartz, Caccioppoli and Lemma 3.3 of \cite{Robin1}
(the fact that solutions tend to get larger at corkscrew points) to get that 
\begin{multline*}
\omega_R^{X}(B(Q,r)\cap \partial \Omega) 
\leq Cr^{\frac{n-2}{2}}\left(\int_{B(Q,2r)\cap \Omega} |\nabla_Y G_R(X, Y)|^2 
\, dY\right)^{1/2}\\
\leq Cr^{n-2}\left(\fint_{B(Q,4r)\cap \Omega} |G_R(X, Y)|^2\, dY\right)^{1/2} 
\leq Cr^{n-2}G_R(X, A_{4r}(Q)). 
\end{multline*}
Here we can apply \cite[Lemma 3.3]{Robin1} because $G_R$ 
weakly satisfies homogeneous Robin boundary conditions and is non-negative on the boundary.
 Now $G_R(X, A_{4r}(Q)) \leq C G(X, A_r(Q))$ by Harnack (and because $X\in \Omega\backslash B(Q, Cr)$),
 and the upper bound in \eqref{e:hmgfequiv} follows. 
 The lower bound follows from \eqref{e:lowermeasure}.

In the second case, $ar^{2-n}\sigma(B(Q,r)) \ll 1$, we use again the representation formula ((5.4) in \cite{Robin1}) 
and Lemma 3.3 of \cite{Robin1} to get that 
$$
\omega_R^X(B(Q,r)\cap \partial \Omega) 
\leq a\int_{\partial \Omega} \varphi(Y)  G_R(X,Y)\, d\sigma(Y) 
\leq Ca\sigma(B(Q, 2r)\cap \partial \Omega) \, G_R(X, A_{2r}(Q)). 
$$ 
Again the upper bound in \eqref{e:hmgfequiv} 
follows by 
a Harnack chain argument and the doubling property of $\sigma$.

For the lower bound when $ar^{2-n}\sigma(B(Q,r)) \ll 1$ we use the representation formula a third time to  
see that 
$$
\omega_R^X(B(Q,r)\cap \partial \Omega) 
\geq a\int_{\partial \Omega} \widetilde{\varphi}(Y)  G_R(X,Y)\, d\sigma(Y), 
$$ 
where $\widetilde{\varphi} \in C_c^\infty(\mathbb R^n)$ and 
$\chi_{B(Q,r/2)} \leq \widetilde{\varphi} \leq \chi_{B(Q,r)}$. 
We now use the Harnack inequality for solutions with homogeneous Robin data 
(see Theorem 4.4 and Remark 7 following it in \cite{Robin1}) to see that (as long as $r^{2-n}\sigma(B(Q,r)) \ll 1$
and $X$ is far enough from $B(Q,r)$) 
$$
\omega_R^X(B(Q,r)\cap \partial \Omega) 
\geq C^{-1}a \sigma(B(Q, r/2)\cap \partial \Omega) \, G(X, A_r(Q)). 
$$ 
As above the lower bound follows by 
a Harnack chain argument and the doubling property of $\sigma$.
		\end{proof}
		
		From this comparison principle, we can prove one of the fundamental properties of the Robin harmonic measure, the fact that it is uniform doubling. 
		
\begin{theorem}\label{t:doubling}[Doubling]
Let $n\geq 3$. There exists a constant $C > 0$ (depending on the ellipticity of $A$ and the geometric constants 
of $(\Omega, \sigma)$), %
such that for all $Q\in \partial \Omega$ and 
$r < \mathrm{diam}(\Omega)/(10C)$, 
we have 
\begin{equation}\label{e:doubling}
				\omega^X_R(B(Q,2r)\cap \partial \Omega) \leq C\omega_R^X(B(Q,r)\cap \partial \Omega), \qquad \forall X\in \Omega\backslash B(Q,Cr).
			\end{equation}
		\end{theorem}
		
		\begin{proof}  
			We apply Theorem \ref{t:hmgfequiv} (twice) and the Harnack inequality to get $$\omega_R^X(B(Q, 2r)\cap \partial \Omega) \leq CG_R(X, A_r(Q))r^{n-2}\min\{1, ar^{2-n}\sigma(B(Q,r)) \} 
				\leq C\omega_R^X(B(Q,r)\cap \partial \Omega),$$
				as desired.
		\end{proof}
		
		Before we can prove the change of pole formula, we need to prove a boundary comparison principle. Roughly this result says that any two non-negative Robin solutions must grow at the boundary at the same rate. These results have been proven for Dirichlet boundary value problem (e.g. in \cite{JKNTA} and \cite{BoundaryHarnack}). We are unaware of any results for the Robin problem in this direction even in smooth domains, and believe it might be of independent interest.
		
\begin{theorem}\label{t:boundarycomp}
There exist constants $C, \tilde{K} > 0$ (depending on the ellipticity of $A$ and the geometric constants 
of $(\Omega, \sigma)$), 
such that if $u, v \geq 0$ weakly solve the Robin problem with homogeneous data in $B(Q, \tilde{K}r)$, $Q\in \p\Omega$, then \begin{equation}\label{e:boundarycomp}
C^{-1}\frac{u(X)}{v(X)} \leq \frac{u(A_r(Q))}{v(A_r(Q))} \leq C\frac{u(X)}{v(X)}, \qquad \forall X\in B(Q, r). 
			\end{equation}
		\end{theorem}
	
We should note that the proof of Theorem \ref{t:boundarycomp} would be much simpler if we allowed $C$ to 
depend on $ar^{2-n}\sigma(B(Q,r))$ and thus could invoke the Harnack inequality at the boundary. 
This more complicated iterative argument ensures 
that the same $C > 0$ works at all scales.

\begin{proof}
	We follow closely the proof of De Silva-Savin \cite{BoundaryHarnack}, which is written for Dirichlet boundary conditions but is extremely flexible. 
			
			\noindent {\bf Claim:} There exists constants $\theta\in (0,1), M > 1$, depending only on the geometric constants of $\Omega$ and the ellipticity of $A$, such that for all $Q\in \partial \Omega$ and $r\geq \rho > 0$ if $w$ is a weak solution to the homogeneous Robin problem in $B(Q, r+\rho)$ such that \begin{equation}\label{e:firstcondition}\begin{aligned} w(x) &\geq -1, \qquad \forall x \in B(Q,r+\rho)\cap \Omega\\
			 w(x) &\geq M,\qquad \forall x\in B(Q, r+\rho)\cap \{y\in \Omega \mid \mathrm{dist}(y,\partial \Omega) \geq \theta \rho\},\end{aligned}\end{equation} then $w \geq 0$ in $B(Q, r)$. 
			
\medskip 
To prove the claim we will show that there us a small constant $\tau > 0$ (depending on the geometric constants 
and the ellipticity of $A$) 
such that if $w$ satisfies \eqref{e:firstcondition} in $B(Q,r+\rho)$, then 
\begin{equation}\label{e:secondcondition}
\begin{aligned} 
w(x) &\geq -\tau, \qquad \forall x \in B(Q,r+\rho/2)\cap \Omega\\
	w(x) &\geq M\tau,\qquad \forall x\in B(Q, r+\rho/2)\cap \{y\in \Omega\mid \mathrm{dist}(y,\partial \Omega) \geq \theta \rho/2\}. 
\end{aligned}
\end{equation}
	If \eqref{e:secondcondition} holds, then $w/\tau$ satisfies \eqref{e:firstcondition} in $B(Q, r+\rho/2)$. 
Iterating this process we get that $w \geq -\tau^k$ in $B(Q, r + \rho/2^k)$. By continuity $w \geq 0$ in $B(Q,r)$ and our claim is proven. 
			
	To show that \eqref{e:firstcondition} implies \eqref{e:secondcondition} we first show that the lower bound on $w$ away from $\partial \Omega$ improves. Note that $w + 1 \geq 0$ in $B(Q, r+\rho)$ and for any $x\in B(Q,r+\rho)$ with $\mathrm{dist}(x, \partial \Omega) \geq \theta \rho$ we have $w(x)+1 \geq M+1$. 
Pick any $y\in B(Q, r+\rho/2)$ with $\mathrm{dist}(y, \partial \Omega) \geq \theta\rho/2$. 
	If $\mathrm{dist}(y, \partial \Omega) \geq \theta\rho$, we know that $w(y)+1 \geq M+1$ and
we are happy. Otherwise, 
pick $q \in \d\Omega$ such that $|q-y| \leq \theta \rho$, then choose a corkscrew point $x$ for the 
ball $B(q, C\theta \rho)$, where $C$ is just chosen so large that $\mathrm{dist}(x, \partial \Omega) \geq \theta\rho$.
There is a Harnack chain that connects $y$ to $x$ in (a probably larger ball) $B(q, C^2 \theta \rho)$, and if 
$\theta$ is small enough, this chain stays in $B(Q,r+\rho)$. Hence $w(x)+1 \geq M+1$ and
$w(y) + 1 \geq C^{-1}(w(x)+1) \geq C^{-1}(M+1)$. So, if $M$ is large enough, we get that
$w(y) \geq C^{-1}(M+1) - 1 \geq \tau M$, for some small $\tau > 0$. Here 
$M$ and $\tau$ depend only on ellipticity and geometric constants, and in particular are independent of $\theta > 0$ (as long as $\theta$ is small enough). This takes care of the part of \eqref{e:secondcondition} close to the boundary. 

For the lower bound on the entirety of $B(Q, r + \rho/2)$ we intend to use repeatedly a density lemma,  Lemma 3.5 in \cite{Robin1},
which gives an improvement of infimum by a small constant $c_\eta < 10^{-1}$. We will use it $\ell$ times, 
where $\ell$ is the smallest integer such that 
\begin{equation}\label{6a20}
(1-c_{\eta})^\ell \leq \tau.
\end{equation}
We will later choose $\theta$ small enough, depending on all these constants (that is, $\tau$, $\ell$, the geometric constant $K \geq 1$ from Lemma 3.5 in \cite{Robin1}, and an additional geometric constant
$C_0$).

We want to show that $w \geq -\tau$ whenever $y\in B(Q, r+\rho/2)$, so we can assume that 
$\mathrm{dist}(y, \partial \Omega) < \theta \rho/2$, because otherwise $w \geq M\tau$ by assumption. 
Then let $q \in \d\Omega$ be such that $|q-y| \leq \theta \rho/2 < \rho/4$ (if $\theta$ is small enough).
Set $B_k = B(q, r_k)$, with $r_k = C_0(2K)^{\ell -k} \theta \rho$, for $0 \leq k \leq \ell$; 
notice that the largest ball, $B_0 = B(q, C_0 (2K)^{\ell} \theta \rho)$, is contained in 
$B(q, \rho/4) \subset B(Q,r+\rho)$ if $\theta$ is small enough.

We start with an argument for $B_0$. 
Observe that since $w$ is a weak solution in $B_0 \subset B(Q,r+\rho)$, with homogeneous data, 
Lemma 3.4 in \cite{Robin1}, says that $1-w^-$ is a non-negative Neumann supersolution in $B_0$. 
It is also non-negative because $w \geq -1$.
We want to apply Lemma~3.5 in \cite{Robin1} to $1-w^-$ in the ball
$B_1 = (2K)^{-1} B_0$, so we need to check the main assumption, which is that 
\begin{equation}\label{6a21}
\big|\big\{ x \in \Omega \cap 2B_1 \, ; \, 1- w^-(x) \geq 1 \big\}\big| \geq \eta | \Omega \cap 2 B_1|,
\end{equation}
for some $\eta \in (0,1)$, and we may as well take $\eta = 1/2$. The condition is satisfied
as soon as $w \geq 0$, hence by assumption as soon as $x \in  \Omega\cap B_0 \sm Z_\theta$, where
\begin{equation}\label{6a22}
Z_{\theta} = \big\{ x \in \Omega \cap B_1 \, ; \, \dist(x, \d\Omega) \leq \theta \rho\big\}.
\end{equation}
So it will be enough to show that $|2B_1 \cap Z_\theta| \leq \frac12 |\Omega \cap 2B_1|$.
We apply \eqref{6c8}, which say that
$|2B_1 \cap Z_\theta| \leq C (\theta \rho / r_1)^\alpha r_1^{n} \leq C (\theta / r_1)^\alpha |\Omega \cap B_1|$, 
where $C$ and $\alpha > 0$ depend only on the corkscrew constants. But $\theta r/ r_1 \leq C_0^{-1}$,
so we may now choose $C_0$ so large that $|2B_1 \cap Z_\theta| \leq \frac12 |\Omega \cap B_1|$, 
and by the same argument 
\begin{equation}\label{6a23}
|2B_k \cap Z_\theta| \leq \frac12 |\Omega \cap B_k|
\end{equation}
for all $k \leq \ell$.

Return to \eqref{6a21}; with our choice of $C_0$, we get it, then we can apply 
Lemma 3.5 in \cite{Robin1}, and get that $1- w^- \geq c_\eta$ on $B_1$, where $c_\eta > 0$
is given by the lemma. In other words, $w^- \leq (1-c_\eta)$ on $B_1$.

We are now ready to prove by induction that $w^- \leq (1-c_\eta)^k$ on $B_k$ for $k \leq \ell$.
Indeed, suppose this is the case for some $k < \ell$. Consider $v_k = 1 - \frac{w^-}{(1- c_\eta)^k}$. Then
$v_k$ is a nonnegative Neumann supersolution on $B_k$, just as for $1-w^-$ before, and we may apply 
Lemma~3.5 in \cite{Robin1} to it, on $B_{k+1}$. We still have that $w^- = 0$ on $\Omega \cap B_k \sm Z_\theta$,
so \eqref{6a23} yields
\begin{equation}\label{6a24}
\big|\big\{ x \in \Omega \cap 2B_1 \, ; \, v_k \geq 1 \big\}\big| 
\geq |\Omega \cap 2B_{k+1} \sm Z_\theta| \geq \frac12  |\Omega \cap 2B_{k+1}|,
\end{equation}
so we can apply the lemma, find that $v_k \geq c_\eta$ and hence $w^- \leq (1- c_\eta)^{k+1}$
on $B_{k+1}$. For the final step, we see that $w^- \leq (1- c_\eta)^{l} \leq \tau$ on $B_\ell$.
But clearly $y \in B_\ell$ (because $|q-y| \leq \theta \rho$ and $r_\ell = C_0 \theta \rho > \theta \rho$), so
$w(y) \geq -\tau$, as needed. This completes the proof of \eqref{e:secondcondition}, and the claim follows.

		\medskip
Let us return to the statement of our Theorem, and let $u, v$ be the given non-negative solutions to the homogeneous Robin problem. Since the statement of the theorem  
is invariant under $u\mapsto \lambda u$ for any constant $\lambda > 0$ 
we can guarantee that $u(A_r(Q)) = v(A_r(Q)) = 1$. By symmetry it suffices to show that there exists a $C > 0$ (universal) such that $Cv-u \geq 0$ in $B(Q,r)$. 
We will show that there is a constant $C > 0$  such that 
$$
w:= \frac{Cv - u}{\sup_{B(Q, 2r)} u}
$$ 
satisfies the conditions of \eqref{e:firstcondition} in $B(Q, 2r) = B(Q,r+r)$. Indeed, by construction $w\geq -1$ in $B(Q, 2r)$. Given $\theta \in (0,1)$ as in the Claim above, we can connect any $y\in B(Q, 2r)$ with 
$\mathrm{dist}(y,\partial \Omega)\geq \theta r$ to $A_r(Q)$ by a Harnack chain with length depending on 
$\theta > 0$ and the NTA constants. Thus there exists an $\eta \in (0,1)$ (depending on $\theta$, the NTA constants and the ellipticity of $A$) such that 
$$w(y) + 1 \geq \eta\left(\frac{Cv(A_r(Q)) - u(A_r(Q))}{\sup_{B(Q, 2r)} u} + 1\right) \geq \eta \frac{C-1}{\sup_{B(Q,2r)} u}
$$ 
for all such $y$. 
Apply Lemma 3.3 in \cite{Robin1} to $u$ to see that $\sup_{B(Q,2r)} u \leq \tilde{C}u(A_r(Q)) = \tilde{C}$ for some geometric constant $\tilde{C} > 0$. So picking $C > 0$ large enough such that $\eta \frac{C-1}{\tilde{C}} \geq M+1$ we have that $w$ satisfies the conditions of the Claim. Thus $w\geq 0$ in all of $B(Q,r)$, which means
that $u \leq C v$ there.
The theorem follows. 
	\end{proof}

An immediate consequence of Theorem \ref{t:boundarycomp}, 
the change of pole formula for the 
Robin harmonic measure.
			
	\begin{corollary}\label{c:changeofpole}
Let $n\geq 3$. There exists a $C  >0$, which depends on the ellipticity of $A$ and the geometric constants 
of $(\Omega, \sigma)$, 
such that if $E\subset B(Q,r)\cap \partial \Omega$ and $X,Y \notin B(Q, CR)\cap \Omega$ then 
\begin{equation}\label{e:changeofpole}
					C^{-1}\frac{\omega^X_R(E)}{\omega_R^X(B(Q,r)\cap \partial \Omega)} \leq \frac{\omega^Y_R(E)}{\omega_R^Y(B(Q,r)\cap \partial \Omega)} \leq C\frac{\omega^X_R(E)}{\omega_R^X(B(Q,r)\cap \partial \Omega)}.
\end{equation}
	\end{corollary}
			
			\begin{proof}
	Since $\omega^X_R$ and $\omega^X_R$ are outer regular, it is enough to prove this when $E$ is open
in $\d\Omega$, and by a covering argument we can further reduce to the case when 	
$E= B(P, s)\cap \partial \Omega \subset B(Q,r)\cap  \Omega$, with $P \in \d\Omega$. 	
				
				Let $X, Y \notin B(Q, Cr)\cap \partial \Omega$. Then $$\frac{\omega^X_R(B(P,s)\cap \partial \Omega)}{\omega^X_R(B(Q,r)\cap \partial \Omega)}\simeq \frac{G_R(X, A_s(P))s^{n-2}\min\{1,as^{2-n}\sigma(B(P,s))\}}{G_R(X,A_r(Q))r^{n-2}\min\{1, ar^{2-n}\sigma(B(Q,r))\}},$$
				by Theorem \ref{t:hmgfequiv}. We get a similar relation for the pole $Y$. 
				
				Thus, to prove the desired result, it suffices to show that 
		\begin{equation}\label{e:almostdone}
		\frac{G_R(X, A_s(P))}{G_R(X, A_r(Q))}\simeq \frac{G_R(Y, A_s(P))}{G_R(Y, A_r(Q))}
		\end{equation}
		or equivalently 
		\begin{equation}\label{e:almostdone2}
		\frac{G_R(X, A_s(P))}{G_R(Y, A_s(P))} \simeq \frac{G_R(X, A_r(Q))}{G_R(Y, A_r(Q))}.\end{equation}
				
	Using symmetry (i.e.,  \eqref{2b8}),  
note that $u(-) := G_R(X, -)$ and $v(-) := G_R(Y, -)$ satisfy the hypothesis of Theorem \ref{t:boundarycomp} in $B(Q, 4r)$, and thus \eqref{e:almostdone} follows immediately. 
			\end{proof}
			
			\subsection{The $A_\infty$ character of $\omega_R$}\label{ss:Ap}
			
			We end this section with a result mentioned in the introduction: a more precise description of the $A_\infty$ character of the Robin harmonic measure in the ``Dirichlet regime".  First we start with a quantitative result, which makes more precise the intuition that Robin harmonic measure is a ``smoothing out" of Dirichlet harmonic measure. We use the natural notation that $\omega_D^{X_0}$ is the Dirichlet harmonic measure for the operator $\mathrm{div}(A\nabla -)$ and domain $\Omega$. 
						
			\begin{lemma}\label{l:ainfinitylocal}
			Let $P \in \partial \Omega$ and $r = r_P > 0$ such that $I_{P}(r) = 1$. Furthermore, assume $X_0 \in \Omega \sm B(P, Cr)$. Then \begin{equation}\label{e:ainfinitylocal}\frac{d\omega_R^{X_0}}{d\sigma}(P) \simeq \frac{\omega_D^{X_0}(B(P,r))}{\sigma(B(P,r))}.\end{equation} The constants of comparability depend only on the geometric constants of $\Omega$ and the ellipticity of $A$, not on $P, r, a$. 
			
			\end{lemma}

			\begin{proof}
		From the representation formula, \eqref{e:representationformula}, and the boundary Harnack inequality we have $$\frac{d\omega_R^{X_0}}{d\sigma}(P) = aG_R(X_0, P) \simeq a G_R(X_0, A_r(P)).$$
		
		Applying the key estimate,  \eqref{e:GFestclose}, and the analogue of Theorem \ref{t:hmgfequiv} for the Dirichlet harmonic measure (which we know holds in the domains we consider, see, e.g. \cite{HMMTZ}) we get that $$\frac{d\omega_R^{X_0}}{d\sigma}(P) \simeq ar^{2-n} \omega_D^{X_0}(B(P, r)).$$ The result follows by the property that $I_P(r) = 1$. 
			\end{proof}

			Naively, Lemma \ref{l:ainfinitylocal} should immediately imply sharp estimates on the $A_\infty$ character of $\omega_R$ in $B(Q,r)$ as $I_Q(r)\uparrow \infty$. Unfortunately, the smoothing in \eqref{e:ainfinitylocal} seems to be hard to work with. However, similar arguments to those above show that if the Dirichlet harmonic measure is in $A_\infty$ then so is the Robin harmonic measure. 
			
			\begin{theorem*}[Theorem \ref{t:keepainfinity} restated]
			Let $A, \Omega$ be such that $\omega_D \in A_\infty(\sigma)$ where $\omega_D$ is the Dirichlet elliptic measure associated to the operator $\mathrm{div}(A\nabla -)$. Then for any $0 < a < \infty$ we have $\omega_R \in A_\infty(\sigma)$ with constants that depend only on the $A_\infty$-character of $\omega_D$ (and in particular are independent of $a$). 
			\end{theorem*}

To be precise, in convention with what it means for the harmonic measure to satisfy $A_\infty$ estimates, we will show that there exists a $C > 0$ and a $\theta \in (0,1)$ such that if $Q\in \partial \Omega, r_0 > 0$ and $E \subset B(Q,r_0)\cap \partial \Omega$ then \begin{equation}\label{e:ainfinityrewrite}\frac{\omega_R^{X_0}(E)}{\omega_R^{X_0}(B(Q, r_0))} \leq C \left(\frac{\sigma(E)}{\sigma(B(Q, r_0))}\right)^{\theta}, \qquad \forall X_0 \in \Omega \sm B(Q, Cr_0).\end{equation}
		Furthermore, we will show that the constants $C, \theta$ depend only on the corresponding constants for $\omega_D$ and not on $A$. 
		
\begin{proof}		Recall that \cite{Robin1} tells us \eqref{e:ainfinityrewrite} holds with $\theta = 1$ as long as $I_Q(r_0) \leq 1$. So we may assume that $I_Q(r_0) > 1$. Furthermore, by the change of pole formula (Corollary \ref{c:changeofpole}) we can assume that $X_0 = A_{Cr_0}(Q)$ for some large $C > 0$. 
		
		Cover $E \subset \partial \Omega$ with balls $B(P_i, r_i)$ such that $P_i \in \partial \Omega$ and $B(P_i,r_i)$ have bounded overlap (if we like, we can ask that $B(P_i, r_i/5)$ are disjoint, using the Vitali covering theorem). We pick the $r_i$ such that $I_{P_i}(r_i) = 1$. Then using the representation formula, \eqref{e:representationformula}, and the boundary Harnack inequality, Theorem \ref{thm:bdryharnack} $$\omega_R^{X_0}(E) \lesssim a\sum_{i=1}^n G_R(X_0, A_{r_i}(P_i)) \sigma(B(P_i, r_i)\cap E).$$ 
		
		Applying the key estimate, \eqref{e:GFestclose} and then the comparability of the Dirichlet Greens function and harmonic measure (i.e. the analogue of Theorem \ref{t:hmgfequiv} for the Dirichlet problem, which is known to hold on uniform domains, see, e.g. \cite{AH}, \cite{HMMTZ}) we get that $$\begin{aligned} \omega_R^{X_0}(E) \lesssim& a\sum_{i=1}^n G_D(X_0, A_{r_i}(P_i)) \sigma(B(P_i, r_i)\cap E)\\
		\simeq& \sum_{i=1}^n ar_i^{2-n} \omega_D^{X_0}(B(P_i, r_i)) \sigma(B(P_i,r_i)\cap E)\\
		\simeq& \sum_{i=1}^n \omega_D^{X_0}(B(P_i, r_i))\frac{\sigma(B(P_i, r_i)\cap E)}{\sigma(B(P_i, r_i))}.\end{aligned}$$ Here the last line follows because $I_{P_i}(r_i) = 1$. 
		
		Let $Mf(x) :=\sup_{x\in B} \frac{1}{\sigma(B)} \int_{B} f(y)\, d\sigma(y),$ be the Hardy-Littlewood maximal function, where we take the supremum over all balls centered on $\partial \Omega$ containing $x$. For any $y\in B(P_i, r_i/5)\cap \partial \Omega$ it is easy to see that $$M\chi_E(y) \geq \frac{\sigma(B(P_i, r_i)\cap E)}{\sigma(B(P_i, r_i))}.$$ Using the doubling of Dirichlet Harmonic measure (which also holds in uniform domains, see again \cite{AH, HMMTZ}) we get that $$\begin{aligned}\omega_R^{X_0}(E) \leq& C\left(\sum_{i=1}^n \left(\omega_D^{X_0}(B(P_i, r_i/5))\right)^{q(1-1/p)}\right)^{1/q}\left(\sum_{i=1}^n \omega_D^{X_0}(B(P_i, r_i/5))\left(\frac{\sigma(B(P_i,r_i)\cap E)}{\sigma(B(P_i, r_i))}\right)^p\right)^{1/p}\\
		\leq& C\left(\sum_{i=1}^n \omega_D^{X_0}(B(P_i, r_i/5))\right)^{1/q} \left(\int_{\partial \Omega \cap B(Q, r_0)} (M\chi_E)^p(y)\, d\, \omega_D^{X_0}(y)\right)^{1/p}.\end{aligned}$$ Above we could have picked $p, q$ to be any conjugate pair, we will fix them shortly. 
		
		Note that each $B(P_i, r_i/5) \subset B(Q, 2r_0)$ and so $\sum_{i=1}^n \omega_D^{X_0}(B(P_i, r_i/5)) \leq C\omega_D^{X_0}(B(Q, r_0))$ by doubling.  We finally use the assumption that $\omega_D^{X_0} \in A_p(\sigma)$ for some $p > 1$. It follows that $$\omega_R^{X_0}(E) \leq  C\omega_D^{X_0}(B(Q, r_0))^{1/q}\left(\int_{B(Q, r_0)} \chi_E^p\, d\omega^{X_0}_D\right)^{1/p} = C\omega_D^{X_0}(B(Q, r_0)) \left(\frac{\omega^{X_0}_D(E)}{\omega^{X_0}_D(B(Q, r_0))}\right)^{1/p}.$$
		
		Using the Bourgain estimate (Lemma \ref{l:Bourgain})  for the Robin harmonic measure and then using one more time that $\omega_D \in A_\infty(\sigma)$ we get that $$\frac{\omega_R^{X_0}(E)}{\omega_R^{X_0}(B(Q, r_0))} \leq C\left(\frac{\omega^{X_0}_D(E)}{\omega^{X_0}_D(B(Q, r_0))}\right)^{1/p} \leq \tilde{C} \left(\frac{\sigma(E)}{\sigma(B(Q, r_0))}\right)^{\tau/p},$$ for some $\tilde{C} > 0$ and $\tau \in (0,1)$. The result follows. 
		\end{proof}
		
 An attentive reader may be concerned that $\omega_D \in A_p(\sigma)$ whereas it seems we have proven that $\omega_R \in A_{\tilde{p}}(\sigma)$ for some $\tilde{p}> p$. This doesn't comport with our understanding that $\omega_R$ somehow converges to $\omega_D$ when $a\uparrow \infty$. However, we should note that we made a large overestimate in the above when we wrote $\sum \omega_D^{X_0}(B(P_i, r_i/5)) \leq C\omega_D^{X_0}(B(Q, r_0)) \simeq C$. Indeed, as $a\uparrow \infty$ we see that $\sum \omega_D^{X_0}(B(P_i, r_i/5)) \rightarrow \omega_D^{X_0}(E)$ (since the $B(P_i, r_i)$ become a finer and finer cover of $E$). Tracking this through the rest of the proof we see that we do not lose an exponent in the ``Dirichlet-limit". 
 
\section{Applications to the total Flow problem}\label{s:Flux}

We turn now to the computations of total flow through a ``pre-fractal" lung. 
Recall that $\Omega$ will be a domain such that 
\begin{equation} \label{7a1}
B(0,4) \subset \Omega \subset B(0,C)
\end{equation}
for some ``geometric'' constant $C \geq 4$.

We also take a measure $\sigma$ on $\d\Omega$, and we assume that 
\begin{equation} \label{7a2}
(\Omega, \sigma) \text{ is a one-sided pair of mixed dimension, as in Definition \ref{d:mixed}.}
\end{equation}
In addition, for our main estimates, we assume that 
\begin{equation} \label{7a3}
(\Omega, \sigma) \text{ is pre-fractal, i.e. satisfies the conditions in Definition \ref{d:piecewisesmooth}.}
\end{equation}

We seek information on the total flow problem described below, and it will be important that the constants 
in our various estimates only depend on $(\Omega, \sigma)$ through the constant $C$ in \eqref{7a1},
and the geometric constants in \eqref{7a2} and \eqref{7a3}, 
not on the Robin parameter $a$. We'll also try to use \eqref{7a3} as seldom as we can,
because the definitions and basic properties in the general case are probably interesting.

For $a\in (0,\infty)$ we let $u_a$ be the unique solution of the system
\begin{equation}\label{7a4}
\begin{aligned} 
-\Delta u_a =& 0, \qquad \text{in}\,\, \Omega \backslash B_1 \text{ (with, here and below, $B_1 = B(0,1)$)} \\
u_a=& 1, \qquad \text{on}\,\, B_1\\
\partial_\nu u_a + au_a =& 0, \qquad \text{on}\,\, \partial \Omega.
\end{aligned}
\end{equation}
There are different ways to obtain existence and uniqueness here. 
One way is to use all our assumptions, notice that 
$\partial \Omega$ is piecewise Lipschitz, , 
and solve \eqref{7a4} classically (the reader may even assume that $\d\Omega$ is smooth,
if they find this more comfortable, but we will never use smoothness in the estimates).
The advantage of this local Lipschitz assumption is that we will feel free to compute various 
quantities by integrating by parts (for real, without using weak versions of everything).

But we slightly prefer to use the methods of this paper, notice that since $B_4\subset \Omega$,
$\Omega \sm B_1$ is also a one-sided NTA domain (see Lemma \ref{l:onesidedpunch}),
and let $u_a$ be the unique function of $W^{1,2}(\Omega \sm B_1)$ such that $u = 1$ on $B_1$
(and hence $Tr(u_a)|_{\partial B_1}= 1$) and which is a weak solution in the sense that 
\begin{equation}\label{7a5}
 \int_{\Omega} \nabla u_a \nabla \varphi + a\int_{\partial \Omega}u_a \varphi d\sigma = 0, \qquad \forall \varphi \in C_c^\infty(\mathbb R^n\backslash \ol B_1). 
\end{equation}
Then the existence of a solution follows by arguing as in \cite[Theorem 2.5]{Robin1} and uniqueness 
follows as well from the weak maximum principle Lemma \ref{lem:mixedmax}. 
And we can use our results  (i.e., \cite[Theorem 4.5]{Robin1}) 
to check that $u_a$ is H\"older continuous all the way to $\d\Omega$, as usual with uniform bounds
that do not depend on $a$.

Finally,  here we decided to work with the Laplacian to simplify our formulas, and in particular 
$A \equiv I_n$ is a symmetric matrix, so we are allowed to use a variational definition as well: 
we may let $u_a$ be the unique minimizer in $W^{1,2}(\Omega)$ of the functional
\begin{equation}\label{7a6}
J_a(u) = \int_{\Omega} |\nabla u|^2 + a \int_{\d\Omega} u^2 d\sigma, 
\end{equation}
under the constraint that $u = 1$ on $B_1$.
In this context, which we patiently avoided using up to now, the existence and uniqueness follow because the energy $\int_{\Omega} |\nabla u|^2$ (together with the given values on $B_1$) 
controls the $L^2$ norm of the trace of $u$ on $\d\Omega$, and then our functional is convex. 
We skip the easy details because we will not need the variational definition after all. 
The fact that the variational solution also solves the weak
problem \eqref{7a5} is fairly easy (compare $J_a(u_a)$ to $J_a(u_a + t\varphi)$ and differentiate in $t$), 
but we skip this too.

In this section we fix $(\Omega, \sigma)$ and want to study the way the total flow, $F(a)$, of $u_a$
through $\d\Omega$ varies with $a$. In the context of the lung, we think about how much oxygen is 
absorbed by unit of time, in a stationary diffusion process.
As we will explain soon, the most flexible definition consists in measuring the total flow near $\d B_1$. We take 
\begin{equation}\label{7a7}
F(a) := \int_{\d B_1}\partial_{\nu} u_a \, d\mathcal H^{n-1}|_{\partial B_1}= a \int_{\d \Omega} u_a d\sigma,
\end{equation}
but let us explain a bit. First $\nu$ is the unit normal to $\partial B_1$ oriented in the direction of the origin. The normal derivative exists (and is even 
continuous) because $u = 1$ on $\d B_1$ and is harmonic on some $B_r \sm \ol B_1$, $r > 1$,
but the worried reader may also compute the flux on a sphere $\d B_r$, $r > 1$ close to $1$,
where $u$ is smooth; the result is the same since $u$ is harmonic.
Moreover, the total flow is positive because Lemma~\ref{lem:mixedmax} (applied
to $1-u$ and $f = a$) says that $u_a \leq 1$ on $\Omega$, or by the second formula.
For the second formula, we use the definition \ref{7a5}, with a smooth radial function 
$\varphi$ which is equal to $0$ near $B_1$ and to $1$ on $B_R \sm B_r$ for some
$r > 1$ close to $1$ and $R$ so large that $\ol \Omega \subset B_R$, and interpolates nicely on the remaining annuli. When $r$ tends to $1$, $ \int_{\Omega} \nabla u_a \nabla \varphi$ tends to $-\int_{\d B_1}\partial_{\nu} u_a \, d\mathcal H^{n-1}|_{\partial B_1}$, which is thus equal to the common value of  
$a\int_{\partial \Omega}u_a \varphi d\sigma = a\int_{\partial \Omega}u_a d\sigma$; 
\eqref{7a7} follows.

\begin{remark}\label{r7a1}
{\bf About a specific choice of $\sigma$.}
So far we do not need the pre-fractal condition \eqref{7a3} on $\sigma$,
and we shall state some of our results without this condition.
But if we use it,  Definition \ref{d:piecewisesmooth} says that, up to multiplicative constants,
$\sigma(B(Q,r))$, $Q \in \d\Omega$, depends only on $r$. That is, $\sigma$ has some homogeneity
with respect to translations.
In particular, since $\d\Omega$ is also Lipschitz at small scales, a small density argument shows that
$\sigma$ is absolutely continuous with respect to $\sigma_0 = \H^{n-1}_{\vert \d\Omega}$
(the usual surface measure), and by the homogeneity condition, we can even say that 
\begin{equation}\label{7a8}
\sigma = f \sigma_0 \ \text{ with a density $f$ such that }
C^{-1} \lambda \leq f(y) \leq C \lambda 
\ \text{ for } y \in \d\Omega,
\end{equation}
where $\lambda > 0$ is a constant. Since multiplying $\sigma$ by $\lambda$ amounts to dividing
$a$ by $\lambda$, we shall assume without loss of generality that $\lambda = 1$, which is a way
of normalizing $\sigma$ so that it is similar to $\sigma_0$. 
This will make our discussion of cases a little easier.

We may even add the requirement that 
\begin{equation}\label{7a9}
\sigma = \sigma_0 = \H^{n-1}_{\vert \d\Omega}.
\end{equation}
This will make it easier to define the total flow and compute it in different ways, but then we lose a little bit of generality,
because allowing a nonconstant density $f$ in \eqref{7a8} is a hidden way of allowing a Robin parameter $a$ which depends on the point $y\in \partial \Omega$. That is, replacing $a$ with $a(x)$ and $d\sigma$ with $\frac{a}{a(x)} d\sigma$ does not change the solution. 

 Yet observe that when \eqref{7a9} actually holds we have
\begin{equation} \label{7a10}
a\int_{\partial \Omega} u_a\, d\sigma  = F(a) = -\int_{\partial \Omega}\partial_{\nu} u_a \, d\sigma
\end{equation}
We can see this in a number of ways but the easiest is to use the divergence theorem to show the second equality. The first equality is in \eqref{7a7}. Underlying this confusion is that the weak and classical (\eqref{7a5} and \eqref{7a4}) formulations of the Robin condition are only equivalent when $\sigma = \mathcal H^{n-1}|_{\partial \Omega}$. 
\end{remark}

\begin{remark} \label{r7a2}
{\bf About the lung.}
As mentioned in the introduction, we are attempting to model the deep lung, where we assume 
that oxygen density is governed by  diffusion, and is absorbed at the boundary with local permeability governed by $a$\footnote{Physically, $a$ is the local permeability of $\partial \Omega$ divided by the diffusion constant, but we set the latter constant equal to 1 for the sake of simplicity and because it does not change our arguments}. Here $a = 0$ (Neumann) would mean  no absorption, and $a = +\infty$
means total absorption. We are not thinking of optimizing
$\Omega$ here, but we allow coefficient $a$ to vary. In particular, we
ask how the total flow $F(a)$ (the quantity of oxygen that gets absorbed through $\partial \Omega$ per unit of time) depends on $a$.
For instance, $a$ may change in a given lung because of some pathology.
In the case of the human lung, $a$ is quite small, and the lung compensates by being vaguely fractal, and having 
a very large value of $\sigma(\d\Omega)$. But we could imagine other situations, and we do not want
to limit ourselves to a self-similar case.
That is, our homogeneity assumption says that $\sigma(B(Q,r)) \sim g(r)$ for some reasonably nice function $g$. 
We can take $g(r) = C r^d$ for $r \geq \ell$, with $n-2 < d < n$ (because of the corkscrew and  mixed dimension assumptions), but we want to allow more complicated functions $g$ as well. For instance, the reader
may take their favorite fractal of dimension $d \in (n-2, n)$, smooth it up at scale $\ell \ll 1$, and get an example.

We decided to take a model with a slightly diffused source, i.e., with the boundary constraint $u = 1$ on 
$B_1$. One could also have decided to put a point source at the origin, i.e., replace $u_a$
with the Green function $G_R^a$ with a pole at the origin. We sill see soon that the two functions are
often comparable, so the two problems are not so different.

A good part of the discussion below is valid in dimension $n=2$, but in the last subsection we will
use our estimates on the Robin Green function, and thus restrict to $n \geq 3$.

We also chose to restrict to the Laplacian to simplify the exposition, but we do not expect major changes for
elliptic operators, though at several points we use the variational structure of the total flow and so it is possible that our theorems only hold for symmetric elliptic operators. In any case, we did not check any details except for the Laplacian. 
\end{remark}

The rest of this paper is devoted to proving Theorem \ref{thm:Fluxest}.

\subsection{Monotonicity of $u_a$ and of the total flow}

The estimates in this subsection are still valid in ambient dimension $n=2$, as they do not use our
estimates on the Green function. 
We will need some general information on the solutions $u_a$ and the total flow. 
We start with the monotonicity of $a \mapsto u_a$. We have two obvious endpoints for the family 
of solutions $\{ u_a \}$, namely $u_0 = 1$ (the obvious solution of our equations with $a=0$, where for instance the uniqueness comes from Lemma \ref{lem:mixedmax}), and the solution $u_\infty$ of the Dirichlet problem, i.e., 
the continuous function on $\ol \Omega$ such that $u_\infty = 1$ on $B_1$,
$u_\infty = 0$ on $\d\Omega$, and $\Delta u = 0$ on $\Omega \sm \ol B_1$.

\begin{lemma}\label{l7a1}
Assume \eqref{7a1} and \eqref{7a2}. We have that 
\begin{equation} \label{7a11}
0 < u_\infty(X) < u_b(X) \leq u_a(X) < 1
\ \text{ for all $0 < a < b < +\infty$ and } X \in \ol\Omega \sm \ol B_1.
\end{equation}
\end{lemma}

\begin{proof}
Recall that all these functions exist and are continuous on $\ol\Omega$ for any $(\Omega, \sigma)$ of mixed type, and so the statements make sense even if we do not assume that $\partial \Omega$ is smooth. 

We said earlier that the fact $0 \leq u \leq 1$ comes from the maximum principle
Lemma \ref{lem:mixedmax}. Then $u_a > 0$ on $\ol\Omega$, by  \cite[Theorem 4.4]{Robin1}
(see also Remark 7 above (4.19) there), which says the the whole $\d\Omega$ is active. 
For the central inequality, consider $v = (b-a)^{-1} (u_a- u_b)$; we want to check that $v \geq 0$.
We will write the computation with strong solutions to simplify, but the argument translates in weak case,
as in the proof of Lemma \ref{lem:changea}. Notice that $v$ vanishes on $\d B_1$, is a (weak) solution to $\Delta v=0$
on $\Omega \sm \ol B_1$, and satisfies the boundary condition 
\begin{equation} \label{7a12}
\d_{\nu} v + a v = \frac{- a u_a + b u_b + a(u_a-u_b)}{b-a} =  u_b 
\end{equation}
on $\d\Omega$. Lemma \ref{lem:mixedmax} then says that $v \geq 0$ on $\ol\Omega$, as needed.

At this point, if we wanted the strict inequalities in \eqref{7a11} to hold in $\Omega$ but not necessarily on $\partial \Omega$, we could invoke the strong maximum principle. As it is, a short argument gives the strict inequalities through $\partial \Omega$. Take $b > a$.
We know that $u_a \neq u_b$, so  $v > 0$ somewhere, and by the maximum
principle this has to happen also somewhere on $\d\Omega$. Then we want to apply 
\cite[Theorem 4]{Robin1} on a small ball $B = B(q,\rho)$ centered on $\d\Omega$. 
If $\rho$ is small enough, the near-Neumann condition (4.15) there is satisfied;
if $b-a$ is small enough, which we can safely assume, we also get that the right-hand side
$\beta = (b-a) u_a$ is small enough, and then the theorem says that 
$\inf_{\Omega \cap B} v \geq \theta \sup_{\Omega \cap B} v$. Iterating this with a finite number of balls
with the same small radius, centered all over the boundary, gives $v > c > 0$  on $\d\Omega$, and the strict inequality 
follows (by the maximum principle again).
The strict inequality in $u_\infty(X) < u_b(X)$ and $u_a(X) < 1$ follow too, because we can take intermediate values of the parameter.
\end{proof}

In order to prove monotonicity of the total flow, it is useful for us to differentiate, pointwise, $u_a$ as a function of $a$.

\begin{lemma}\label{l7a2}
Still assume \eqref{7a1} and \eqref{7a2}. The map $a \mapsto  u_a$
(say, from $(0, +\infty)$ to $L^{\infty}(\ol \Omega)$) has a derivative at every $a > 0$,
namely the function $w_a \in W^{1,2}(\Omega \sm \ol B_1)$ which has a vanishing trace on $\d B_1$,
is harmonic on $\Omega \sm \ol B_1$, and satisfies (weakly) the boundary condition $\d_\nu w_a + a w_a = -u_a$ on 
$\d\Omega$. We take $w_a = 0$ on $B_1$ to complete the definition.
\end{lemma}

\begin{proof}
The formula \eqref{7a12} makes this appear reasonable, but let us check the details. 
We start with the existence and uniqueness, for $\psi \in L^\infty(\d\Omega)$ of a weak solution of
$\Delta w = 0$ in $W^{1,2}(\Omega \sm \ol B_1)$, with a vanishing trace on $\d B_1$, 
and which satisfies (weakly) the Robin boundary condition $\d_\nu w + a w = \psi$ on $\d\Omega$. 
In \cite[Theorem 2.10.]{Robin1}, we solved a similar equation, but only with the (weak) Neumann
condition $\d_\nu w  = \psi$. This was unfortunate, but the result also works with the Robin condition above;
the proof goes as in \cite[Theorem 2.10.]{Robin1}, where we use the restriction to 
$W_0 = \big\{ u \in W \, ; \, Tr(u) = 0 \text{ on } \d B_1\big\}$ of the bilinear form of
\cite[Theorem 2.5.]{Robin1} (i.e., with the extra term $a\int_{\d\Omega} Tr(u) Tr(\varphi) d\sigma$).

Now we choose $\psi = -u_a$ to obtain our candidate $w_a$,  and also consider a solution $w^{1}$ corresponding to $\psi = 1$. The latter will be a useful majorant proving that $a \mapsto u_a$ is locally Lipschitz on $(0, +\infty)$.

Fix $a$ and consider the same $v_b = (b-a)^{-1}(u_a - u_b)$ as for \eqref{7a12}. Then $v_b$ is bounded by the maximum principle, since it is a solution with bounded data. To be more precise,
$\d_{\nu} v_b + a v_b = u_b$, and since $0 \leq u_b \leq 1$ (see \eqref{7a11})
the maximum principle (Lemma \ref{lem:mixedmax}) implies that $0 \leq v_b \leq w^1$.
But $w^1$ is H\"older continuous on $\Omega$ and thus bounded (arguing similarly as $u_a$),
so $0 \leq v_b \leq C$ and $a \mapsto u_a$ is locally Lipschitz.

We want to make sure that the candidate we built by solving a boundary problem is indeed the derivative that we were seeking, and that essentially follows from uniqueness. Indeed, let $b$ tend to $a$ in the boundary condition $\d_{\nu} v_b + a v_b = u_b$,
or equivalently  $\d_{\nu} (v_b+w_a) + a (v_b+w_a) = u_b -u_a$.
The right-hand side tends to $0$ in $L^{\infty}$-norm, and the theorem of existence and uniqueness 
used above (for the boundary condition $\d_\nu w + a w = \psi$), yields $\lim_{b \to a} ||v_b + w_a||_{W^{1,2}} = 0$.
Yet as always, our H\"older control of solutions also yields 
$||v_b+w_a||_{L^\infty(\d\Omega)} \leq C ||v_b+w_a||_{W^{1,2}}$, and the lemma follows (recall that we changed the sign of $v_b$ for convenience).
\end{proof}

Let us also check that 
\begin{equation} \label{7a13}
u_\infty = \lim_{a \to +\infty} u_a 
\end{equation}
where the convergence is uniform on $\d\Omega$, hence on $\ol\Omega$.
We already know that $u_a$ is a decreasing function of $a$, so it has a limit $v$.
In addition, the $u_a$ are uniformly H\"older continuous near $\d\Omega$
\footnote{Here we are using the important fact that the constants in the estimates from \cite{Robin1}, in particular Theorem 4.5 there, do not depend on $a$}. So $v$ also is 
H\"older continuous, and by Dini the convergence is uniform.  
We now want to show that $v \equiv 0$ on $\d \Omega$. If otherwise, then $v > \varepsilon > 0$ on a subset $E$ of $\d\Omega$, and $u_a \geq \varepsilon$
there too (by monotonicity), and $J_a(u_a) \geq a \varepsilon^2 \sigma(E)$. 
But if $u_{00}$ is a smooth function that is equal to $1$ on $B_1$ and $0$ outside of $B_2$ then $J_a(u_{00}) \leq c_d$ (a dimensional constant). For $a$ large enough we  contradict the fact that $u_a$ minimizes $J_a$ amongst all $W^{1,2}$ functions with trace 1 on $\partial B_1$.\footnote{
Note that here we would need a different argument
if we were using a non symmetric operator $- {\rm div} A \nabla$.}
Since $v \geq 0$ because $u_a \geq 0$, we get $v = 0$ on $\d\Omega$,
and \eqref{7a13} follows.

We are ready to prove that $F$ is nondecreasing, and even compute its derivative in terms
of $w_a$. We claim that 
\begin{equation} \label{7a14}
F'(a) = \int_{\d B_1} \d_{\nu} w_a d\sigma \geq 0,
\end{equation}
where again $\d_{\nu} w_a$ exists (and is continuous) on $\d B_1$, because $w_a$ is continuous,
equal to $0$ on $\d B_1$, and is harmonic on some $B_r \sm \ol B_1$, $r > 0$. 
The identity in \eqref{7a14} holds by taking the derivative in $a$ of the flux definition \eqref{7a7}
of  $F$; there is no real convergence problem, but the worried reader may also take a radius $r > 1$,
notice that the total flow can also be computed on $\d B_r$, and use the fact that the derivatives
of $u_a$ and $w_a$ on $\d B_r$ are controlled by the $L^\infty$ norms. The facts that 
$F'(a)$ and $\d_{\nu} w_a d\sigma$ are nonnegative come from the fact that $w_a = 0$ on $\d B_1$
and $w_a \leq 0$ on $\Omega$.

We end this section with other computations of $F$ and its derivative, this time in the prefractal case.

\begin{lemma}\label{l7a3}
Assume \eqref{7a1}, \eqref{7a2}, \eqref{7a3}, and even \eqref{7a9}. Then  
\begin{equation}\label{e:fluxequalsenergy}
F(a) = \int_{\Omega \backslash B_1} |\nabla u_a|^2 + a\int_{\partial \Omega} u_a^2 d\sigma= J_a(u_a)
\end{equation}
and
\begin{equation} \label{7a16}
F(\infty) : = \lim_{a \to +\infty} F(a) 
= \int_{\d B_1}  \d_{\nu} u_a \, d\sigma 
= \int_{\Omega \backslash B_1} |\nabla u_\infty|^2,
\end{equation}
where $u_\infty$ is the Dirichlet solution.
\end{lemma}

\begin{proof}
We use our regularity assumption to integrate by parts. We start with the energy
\begin{equation} \label{7a17}
\int_{\Omega \backslash B_1} |\nabla u_a|^2
= \int_{\d(\Omega \sm \ol B_1)} u_a \, \d_{\nu} u_a \, d\sigma
= \int_{\d B_1} u_a \, \d_{\nu} u_a \, d\sigma 
- a \int_{\d\Omega} u_a^2 \, d\sigma
\end{equation}
since $\Delta u = 0$ on $\Omega \sm \ol B_1$, and because 
$\d_{\nu} u_a = - a u_a$ on $\d\Omega$. Then observe that 
\begin{equation} \label{7a18}
\int_{\d B_1} u_a \, \d_{\nu} u_a \, d\sigma
= \int_{\d B_1}  \d_{\nu} u_a \, d\sigma
= F(a)
\end{equation}
because $u_a = 1$ on $\d B_1$ and because the total flow of $u_a$ on 
$\d(\Omega \sm \ol B_1)$ vanishes ($u_a$ is harmonic there); 
now \eqref{e:fluxequalsenergy} follows.

When $a$ tends to $+\infty$, we know that $u_a$ tends to $u_\infty$, and by uniform convergence and maybe computing the flux on a sphere $\d B_r$, $r < 1$, we see that $F(a)$ tends to the total flow of $u_\infty$. Finally, the first part of
\eqref{7a17} yields 
$$
\int_{\Omega \backslash B_1} |\nabla u_\infty|^2 
= \int_{\d(\Omega \sm \ol B_1)} u_\infty \, \d_{\nu} u_\infty \, d\sigma
= \int_{\d B_1} u_\infty \, \d_{\nu} u_\infty \, d\sigma = \int_{\d B_1}  \d_{\nu} u_\infty \, d\sigma = F(a)
$$
because $u_\infty = 0$ on $\d\Omega$ and $u_\infty = 1$ on $\d B_1$.
\end{proof}

Before moving onto finer computations of the total flow, we give one last estimate on $F(\infty)$. If $u_\infty^\Omega$ is the solution to \eqref{7a4} with $a =+\infty$ and in $\Omega$ then by the maximum principle we have that if $\Omega \subset \tilde{\Omega}$ that $u_\infty^\Omega \leq u_\infty^{\tilde{\Omega}}$. The Hopf maximum principle then implies that $F(\infty, \Omega) \geq F(\infty, \tilde{\Omega})$ (where we modify the notation introduced in \eqref{7a7} in the obvious way). Since $B(0, 4)\subset \Omega \subset B(0, C)$ (for some constant $C > 0$) we have $F(\infty) = F(\infty, \Omega) \simeq c$ where $c > 0$ depends on the dimension and the constant $C > 0$ in \eqref{7a1}. Indeed, we can explicitly compute $c = F(\infty, B(0,4)) = (n-2) \sigma(\d B_1) (1-4^{2-n})^{-1}$.

\subsection{Global bounds on the total flow} 

From now on, we restrict again to $n \geq 3$. 
For the estimates in this subsection, we only use  \eqref{7a1} and \eqref{7a2},
and we compare $u_a(X)$ with $G_R^a(X,0)$ (the Robin Green function for $\Omega$)
to get some bounds on the total flow $F(a)$. In particular, all constants $C > 0$ depend only on the constants in \eqref{7a1} and \eqref{7a2}. 

We start with the Dirichlet-like case. Let us show that 
\begin{equation}\label{7a19}
C^{-1} \leq F(a) \leq C  \ \text{ when } a\sigma(\partial \Omega) \geq 1.
\end{equation}
We already know that $F(a) \leq F(\infty) \leq F(\infty, B(0,4))$ (the Dirichlet flux for $B(0,4)$), 
so we only need the lower bound. 
We first want to check that (when $a\sigma(\partial \Omega) \geq 1$)
\begin{equation}\label{7a20}
C^{-1} \leq G_R^a(X,0) \leq C  \ \text{ for } X \in \d B_1.
\end{equation}

If $a\sigma(\partial \Omega) \leq \diam(\Omega)^{n-2}$ then $(a\sigma(\partial \Omega))^{1/(n-2)}=:\rho \geq 1$ and \eqref{7a20} follows immediately from \eqref{e:GFestNeumannend}. If $a\sigma(\partial \Omega) \geq \diam(\Omega)^{n-2}$, we note that $|X-0| = 1 \leq \delta(X)$ and $|X|\leq  \delta(0)$ for any $X \in \d B_1$. Thus \eqref{e:GFestMedium} gives $G_R^a(X,Y) \simeq G_D(X,0) \simeq |X|^{2-n} = 1$ by the classic estimates of Gr\"uter-Widman \cite{GW}. This proves \eqref{7a20}.

Recall that $u_a = 1$ on $\d B_1$. We deduce from \eqref{7a20} and the maximum principle of 
Lemma~\ref{lem:mixedmax} that $u_a \simeq G_R^a(\cdot,0)$ on $\Omega \sm B_1$
(note that it is easy to check, arguing as in Lemma \ref{lem:donutsareconnected}, 
that $\Omega \backslash B_1$  is $1$-sided NTA). Finally,
$$
F(a) = a \int_{\d \Omega} u_a d\sigma  \simeq  a \int_{\d \Omega} G_R^a(x,0) d\sigma(x) = 1  
$$
by \eqref{7a7},  \eqref{e:fluxcondition}, and the symmetry of $G_R^a$. This completes the proof of 
\eqref{7a19}.

\ms
Now we go to the Neumann-like situation, and prove that 
\begin{equation}\label{7a21}
F(a) \simeq a \sigma(\d \Omega)
\ \text{ when } a \sigma(\partial \Omega) \leq 1.
\end{equation}
As before, $\rho := (a\sigma(\partial \Omega))^{\frac{1}{n-2}} \leq 1$, so we can apply 
Theorem \ref{l:neumannregime}, and \eqref{e:Neumannregimefar} says that $G_R^a(X,0) \simeq \rho^{2-n}$
for $X \in \d B_1$. The maximum principle now yields
$G_R^a(X, 0) \simeq u_a(X)\rho^{2-n}$ for $X\in \Omega \sm B_1$, and hence as before
$$
F(a) = a \int_{\d \Omega} u_a d\sigma  \simeq  a  \int_{\d \Omega} \rho^{n-2}  G_R^a(x,0) d\sigma(x) = 
\rho^{n-2} = a \sigma(\partial \Omega),
$$
(by \eqref{7a7} and  \eqref{e:fluxcondition}), which is \eqref{7a21}.

\subsection{The difference in total flow when $a \sigma(\partial \Omega) \geq \diam(\Omega)^{n-2}$.}
In this last subsection, we use the full list of assumptions, i.e.,  \eqref{7a1}, \eqref{7a2}, \eqref{7a3}, and \eqref{7a9},
and prove a more precise asymptotic expansion of $F(a)$ near $+\infty$, in terms of a Makarov entropy function.
More precisely, we bound the quantity $F(\infty) - F(a)$ computed in Lemma \ref{l7a3},
first in terms of $G_R^a$, and then in terms of the entropy. Note that the quantity 
$\displaystyle \frac{1}{F(a)} - \frac{1}{F(\infty)} = \frac{F(\infty) - F(a)}{F(a) F(\infty)} 
\simeq \frac{F(\infty) - F(a)}{F(\infty)^2}$ was studied numerically in 
\cite{Marcelfractal}, 
as a sort of difference of resistivities, in a situation close to ours. 
Starting with Lemma \ref{l7a3} and integrating by parts several times we obtain the key algebraic relation: \begin{equation}\label{e:magicdiff}
\begin{aligned} 
F(\infty)-F(a) =& \int_{\Omega \sm B_1} |\nabla u_\infty|^2 -\int_{\Omega \sm B_1} |\nabla u_a|^2
-a\int_{\partial \Omega}u_a^2 d\sigma
\\ =&  \int_{\Omega \sm B_1} \nabla (u_\infty -u_a) \cdot \nabla (u_\infty + u_a)
-a\int_{\partial \Omega}u_a^2 d\sigma
\\ =& \int_{\partial (\Omega \sm \ol B_1)} (u_\infty -u_a) \partial_\nu (u_\infty +u_a)
 - a\int_{\partial \Omega} u_a^2 d\sigma
 \\ =& \int_{\partial \Omega} (-u_a) (\partial_\nu u_\infty - a u_a)
 - a\int_{\partial \Omega} u_a^2 d\sigma
= -\int_{\partial \Omega }u_a \, (\partial_{\nu} u_\infty) \, d\sigma,
\end{aligned}
\end{equation}
because $u_\infty$ and $u_a$ are harmonic, $u_\infty - u_a = 0$ on $\d B_1$,
and because $\d_\nu u_a = - a u_a$ and $u_\infty =0$ on $\d\Omega$. Here $\partial_{\nu} u_\infty$
exists, and the integration by parts make sense,  because $u_\infty$ is positive harmonic, vanishes at the boundary,
and $\d\Omega$ is locally Lipschitz; 
the reader may add the assumption that $\d\Omega$ is locally smooth (with no estimate), and not lose much information. 

We will feel better with $-\d_\nu u_\infty$ replaced by  $-\partial_{\nu_X} G_D(\cdot, 0)$,
the Poisson kernel for $\Omega$, so let us check that
$$
-\d_\nu u_\infty \simeq -\partial_{\nu_X} G_D(\cdot, 0) 
\ \text{ on } \d\Omega.
$$
Here we do not imply nice uniform estimates for these two functions, but just for the ratio of the two.
This is easy, because we already checked for the proof of \eqref{7a19} that $G_D(\cdot, 0) \simeq u_\infty$
on $\d B_1$ (recall that $u_\infty = 1$ there, \eqref{7a13}, and use \eqref{7a20}), hence on 
$\Omega \sm \ol B_1$ by the maximum principle. Then this stays true with the normal derivatives,
as announced. Now by \eqref{e:magicdiff}, 
\begin{eqnarray}\label{e:intbypartsmagic}
F(\infty)-F(a) &=& -\int_{\partial \Omega }u_a \, (\partial_{\nu} u_\infty) \, d\sigma
\simeq -\int_{\partial \Omega }u_a \, \partial_{\nu_X} G_D(\cdot, 0) \, d\sigma
= \int_{\partial \Omega} u_a \, d\omega_D
\nn\\
& \simeq& \int_{\partial \Omega} G_R^a(\cdot,0)\, d\omega_D,
\end{eqnarray}
where we denote by $\omega_D$ the Dirichlet harmonic measure on $\d\Omega$, with pole at  $0$,
and we used the fact (proved \eqref{7a19}) that $u_a(X) \simeq G_R^a(X,0)$. 
All this happens with constants that depend only on our geometric constants, and not on $a$.

We now use Theorm \ref{t:DirichletRegime} to simplify the expression in \eqref{e:intbypartsmagic}.
That is, we want to  discretize the expression  at some correct scale $r_a$ and replace 
$G_R^a$ by values of $G_D$ computed at corkscrew points. Roughly speaking, $r_a$ will be the scale at which we
switch from Neumann to Dirichlet behavior, and we shall see that, thanks for the homogeneity condition
in our pre-fractal assumption,  it does not really depend on $X \in \d\Omega$.

For each  $X \in \d\Omega$, we chose in \eqref{5a2} a radius $\rho_X \in (0,1)$.
Here, we are interested in $X \in \d\Omega$ (and then the definitions are simpler because $Q_X = X$),
and also $Y = 0$ (and then we pick any $Q_0 \in \d\Omega$). Then we set
$r_X = \min(10^{-1}, \rho_X)$ and let $A_X$ be a corkscrew point for $B(X, r_X)$; for the origin, we
use $A_0 = 0$. Theorm \ref{t:DirichletRegime} says that 
\begin{equation} \label{7a24}
C^{-1} G_D(A_X, 0) \leq G_R^a(X,0) \leq CG_D(A_X,0).
\end{equation}
We promised a fixed radius $r_a$, so we let $r_a = r_{X_0}$ for any $X_0 \in \d\Omega$.
We claim that 
\begin{equation}\label{7a25}
r_X \simeq \rho_X \simeq r_a
\ \text{ for all } X \in \d\Omega.
\end{equation}
Indeed, \eqref{5a3} says that $C^{-1} \leq I_X(\rho_X) \leq 1$, where 
$I_X(\rho) = a \rho^{2-n} \sigma(B(X,\rho))$ is the Neumann-to-Dirichlet index of  \eqref{3c12}
and \eqref{5a3}. Then the first condition of Definition \ref{d:piecewisesmooth}
says that $I_X(r_a) \simeq I_{X_0}(r_a) \simeq 1$, and now the fact that $\rho_X \simeq r_a$
follows because  $C^{-1} \leq I_X(\rho_X) \leq 1$ and \eqref{3b12} says that $I_X(\rho)$
decays at some definite speed (so $\rho_X$ is essentially unique). 

Now we discretize. Cover $\d\Omega$ with balls  $B_{i} = B(Q_{i}, r_{a})$ centered on $\d\Omega$, and with bounded overlap.
Let $A_i$ be a corkscrew point for $B_i$. Then \eqref{7a24} (together with the Harnack inequality) says that 
$G_R^a(X,0) \simeq G_D(A_X,0) \simeq G_D(A_i,0)$ for $X \in \d\Omega \cap B_i$, so
\begin{equation}\label{7a26}
F(\infty)-F(a) \simeq \int_{\partial \Omega} G_R^a(\cdot,0)\, d\omega_D
\simeq \sum_i \omega_D(B_i) G_D(A_i,0).
\end{equation}
But $G_D(A_i,0) \simeq r_a^{2-n} \omega_D(B_i)$, in the domains considered here (see \cite{HMT, AHMT}. See also \cite{BrunoJoseph}).

Finally 
\begin{equation}\label{7a26}
F(\infty)-F(a) \simeq  r_a^{2-n} \sum_i \omega_D(B_i)^2 =: r_a^{2-n} S(\omega_D, r_a, 2), 
\end{equation}
 where $S(\omega_D, r_a, 2)$ is the entropy function of Makarov \cite[Page 9]{Makarov}. The keen eyed reader will note that this is not precisely the entropy functional, but rather is comparable to it, 
 due to the doubling of Harmonic measure in 1-sided NTA domains, see \cite{AH}. The constants of comparability depend only on the geometric constants of $\Omega$.

The entropy functional encodes some average regularity of $\omega_D$. 
In general we cannot say much about its behavior, 
especially in the current generality, but when $a$ is so large that $r_a \leq C\ell$, 
we can use the local regularity of $\d\Omega$ and Dahlberg's theorem to estimate it.

\medskip

\noindent {\bf Analysis of the Entropy when $r_a \leq C\ell$.} 
When $r_a \leq C\ell$ we know that $B_i \cap \partial \Omega$ is given by the graph of a Lipschitz function 
(or the union of a small number of 
such graphs). Then, by \eqref{7a9}, 
 $\sigma(B_i) \simeq r_a^{n-1}$ and $r_a \simeq \frac{1}{a}$ (because $I_X(r_a) \simeq 1$). 
 So 
 $$ 
 F(\infty) - F(a) \simeq  a^{n-2} \sum_i \omega_D(B_i)^2  
 \simeq a^{-n} \sum_{i} \left(\frac{\omega_D(B_i)}{\sigma(B_i)}\right)^2  
 \simeq a^{-n}\sum_{i} \left(\fint_{B_{i}\cap \partial \Omega} \frac{d\omega_D}{d\sigma}(P)\, d\sigma(P)\right)^2.
 $$ 
 At this point we recall Dahlberg's theorem on Dirichlet harmonic measure in Lipschitz domains. 
 In this context it states that there exists a constant $M > 1$, that depends on the Lipschitz norm $L$
 in the pre-fractal assumption and NTA constants of $\Omega$,
 such that for any $r < C\ell$ and any $Q\in \partial \Omega$ we have 
\begin{multline*}
 \left(\fint_{B(Q,r)\cap \partial \Omega} \frac{d\omega_D}{d\sigma}(P)\, d\sigma(P) \right)^2 
 \leq \fint_{B(Q,r)\cap \partial \Omega} \left(\frac{d\omega_D}{d\sigma}\right)^2(P)\, d\sigma (P)\\
 \leq M \left(\fint_{B(Q,r)\cap \partial \Omega} \frac{d\omega_D}{d\sigma}(P)\, d\sigma(P) \right)^2.
\end{multline*}
 A consequence of this inequality is $\omega_D$ lies in $ A_\infty(\sigma)$ locally, but the inequality is stronger, as it states that $\omega_D^0$ actually satisfies a reverse 2-H\"older inequality. 
So we may write 
\begin{equation}\label{}
F(\infty) - F(a) 
\simeq a^{-n}\sum_i \fint_{B_i \cap \partial \Omega} \left(\frac{d\omega_D}{d\sigma}\right)^2(P)\, d\sigma(P) 
\simeq a^{-1} \int_{\partial \Omega}\left(\frac{d\omega_D}{d\sigma}\right)^2(P)\, d\sigma(P).
\end{equation}
This is our desired result and we emphasize again 
that the implicit constants of comparability depend on the constants in 
Definitions \ref{d:piecewisesmooth} and \ref{d:mixed}, and $C$ in \eqref{7a1},
but not on $a$.

\end{document}